\documentclass[12pt]{amsart}
\usepackage{lmodern}
\usepackage{ifthen}
\usepackage{amsfonts}
\usepackage{amsxtra}
\usepackage{amssymb}
\usepackage{upgreek}
\usepackage{bbm}
\usepackage{mathdots}
\usepackage{amsthm}
\usepackage{array}
\usepackage[margin=1.0in]{geometry}
\usepackage{xcolor}

\definecolor{cite}{rgb}{0.30,0.20,1.00}
\definecolor{url}{rgb}{0.00,0.00,0.80}
\definecolor{link}{rgb}{0.20,0.70,0.40}
\usepackage[colorlinks,linkcolor=link,urlcolor=url,citecolor=cite,pagebackref,breaklinks]{hyperref}
\usepackage{tikz-cd}
\usepackage{mathtools}
\usepackage{mathrsfs}
\usepackage{appendix}
\usepackage[all]{xy}
\usepackage[lite,abbrev,msc-links,alphabetic]{amsrefs}

\usepackage{graphicx}
\usepackage{multirow}
\usepackage{pstricks}
\usepackage{pst-pdf}
\usepackage[shortlabels]{enumitem}
\usepackage{pifont}
\usepackage{marvosym}
\usepackage{newtxtext}

\usepackage[OT2,T1]{fontenc}
\DeclareSymbolFont{cyrletters}{OT2}{wncyr}{m}{n}
\DeclareMathSymbol{\Sha}{\mathalpha}{cyrletters}{"58}

\usepackage{graphicx}

\makeatletter
\providecommand*{\Dashv}{%
  \mathrel{%
    \mathpalette\@Dashv\vDash
  }%
}
\newcommand*{\@Dashv}[2]{%
  \reflectbox{$\m@th#1#2$}%
}
\makeatother

\usepackage{nameref}

\makeatletter
\let\orgdescriptionlabel\descriptionlabel
\renewcommand*{\descriptionlabel}[1]{%
  \let\orglabel\label
  \let\label\@gobble
  \phantomsection
  \edef\@currentlabel{#1\unskip}%
  \let\label\orglabel
  \orgdescriptionlabel{#1}%
}
\makeatother

\numberwithin{equation}{section}

\theoremstyle{plain}
\newtheorem{proposition}{Proposition}[section]

\newtheorem{corollary}[proposition]{Corollary}
\newtheorem{lemma}[proposition]{Lemma}
\newtheorem{theorem}[proposition]{Theorem}

\theoremstyle{definition}
\newtheorem{definition}[proposition]{Definition}

\newtheorem{notation}[proposition]{Notation}
\newtheorem{assumption}[proposition]{Assumption}

\theoremstyle{remark}
\newtheorem{remark}[proposition]{Remark}

\renewcommand{\b}[1]{\mathbf{#1}}
\renewcommand{\c}[1]{\mathcal{#1}}
\renewcommand{\d}[1]{\mathbb{#1}}
\newcommand{\f}[1]{\mathfrak{#1}}
\renewcommand{\r}[1]{\mathrm{#1}}
\newcommand{\s}[1]{\mathscr{#1}}

\newcommand{\res}{\mathbin{|}}
\newcommand{\ol}[1]{\overline{#1}{}}

\newcommand{\ul}{\underline}
\renewcommand{\leq}{\leqslant}
\renewcommand{\geq}{\geqslant}

\newcommand{\bB}{\b B}

\newcommand{\bS}{\b S}

\newcommand{\bX}{\b X}

\newcommand{\bu}{\b u}

\newcommand{\bw}{\b w}

\newcommand{\cA}{\c A}
\newcommand{\cB}{\c B}

\newcommand{\cD}{\c D}

\newcommand{\cF}{\c F}

\newcommand{\cH}{\c H}

\newcommand{\cN}{\c N}
\newcommand{\cO}{\c O}

\newcommand{\cS}{\c S}

\newcommand{\cV}{\c V}
\newcommand{\cW}{\c W}
\newcommand{\cX}{\c X}
\newcommand{\cY}{\c Y}
\newcommand{\cZ}{\c Z}

\newcommand{\dA}{\d A}

\newcommand{\dC}{\d C}

\newcommand{\dE}{\d E}
\newcommand{\dF}{\d F}
\newcommand{\dG}{\d G}

\newcommand{\dL}{\d L}

\newcommand{\dP}{\d P}
\newcommand{\dQ}{\d Q}
\newcommand{\dR}{\d R}
\newcommand{\dS}{\d S}
\newcommand{\dT}{\d T}

\newcommand{\dZ}{\d Z}

\newcommand{\fE}{\f E}

\newcommand{\fS}{\f S}

\newcommand{\fW}{\f W}

\newcommand{\fZ}{\f Z}

\newcommand{\fe}{\f e}

\newcommand{\fg}{\f g}

\newcommand{\fm}{\f m}

\newcommand{\fp}{\f p}

\newcommand{\rH}{\r H}
\newcommand{\rI}{\r I}

\newcommand{\rN}{\r N}
\newcommand{\rO}{\r O}

\newcommand{\rU}{\r U}
\newcommand{\rV}{\r V}

\newcommand{\rZ}{\r Z}

\newcommand{\rd}{\r d}
\newcommand{\re}{\r e}

\newcommand{\rt}{\r t}

\newcommand{\sS}{\s S}

\newcommand{\tR}{\mathtt{R}}
\newcommand{\tS}{\mathtt{S}}
\newcommand{\tT}{\mathtt{T}}

\newcommand{\tV}{\mathtt{V}}

\newcommand{\tc}{\mathtt{c}}

\newcommand{\ts}{\mathtt{s}}

\newcommand{\tw}{\mathtt{w}}

\renewcommand{\ge}{\geqslant}
\renewcommand{\le}{\leqslant}

\newcommand{\biota}{\boldsymbol{\iota}}

\newcommand{\qbinom}[2]{\genfrac{[}{]}{0pt}{}{#1}{#2}}
\newcommand{\llangle}{\langle\!\langle}
\newcommand{\rrangle}{\rangle\!\rangle}

\newcommand{\pres}[2]{\prescript{#1}{}{#2}}

\newcommand{\ac}{\r{ac}}

\newcommand{\CF}{\mathbbm{1}}

\newcommand{\cl}{\r{cl}}

\newcommand{\cusp}{\r{cusp}}

\newcommand{\dr}{\r{dR}}

\newcommand{\fin}{\r{fin}}

\newcommand{\id}{\r{id}}

\newcommand{\ord}{\r{ord}}

\newcommand{\ram}{\r{ram}}

\newcommand{\spl}{\r{spl}}

\DeclareMathOperator{\BC}{BC}
\DeclareMathOperator{\CH}{CH}

\DeclareMathOperator{\diag}{diag}

\DeclareMathOperator{\Gal}{Gal}

\DeclareMathOperator{\GL}{GL}

\DeclareMathOperator{\Hom}{Hom}
\DeclareMathOperator{\Herm}{Herm}

\DeclareMathOperator{\Ind}{Ind}

\DeclareMathOperator{\Ker}{ker}

\DeclareMathOperator{\Lie}{Lie}

\DeclareMathOperator{\Mat}{Mat}

\DeclareMathOperator{\modulo}{mod}

\DeclareMathOperator{\rank}{rank}
\DeclareMathOperator{\RE}{Re}
\DeclareMathOperator{\Res}{Res}

\DeclareMathOperator{\Sh}{Sh}

\DeclareMathOperator{\SO}{SO}

\DeclareMathOperator{\Spec}{Spec}

\DeclareMathOperator{\Tr}{Tr}
\DeclareMathOperator{\tr}{tr}

\DeclareMathOperator{\vol}{vol}

\title[Spherical Representations and AIPF]{Spherical representations of unitary groups at ramified places and the arithmetic inner product formula}
\author{Zhuoni Chi}
\thanks{Email: znchi@zju.edu.cn}

\begin{document}
\begin{abstract}
    In this article, we study admissible representations of even unitary groups over local fields, where the quadratic extension is ramified, with invariant vectors under the action of the stabilizer of a unimodular lattice and some properties of the corresponding integral model of unitary Shimura varieties. As a direct application, we are able to improve the arithmetic inner product formula so that the places with local root number \((-1)\) are allowed to be ramified. 
\end{abstract}
\maketitle
\tableofcontents
\section{Introduction}

    In 1986, Gross and Zagier \cite{gross1986heegner} proved a remarkable formula that relates the N\'eron--Tate heights of Heegner points on a rational elliptic curve to the central derivative of the corresponding Rankin--Selberg \(L\)-function. A decade later, Kudla \cite{kudla1997central} revealed another striking relation between Gillet--Soul\'e heights of special cycles on Shimura curves and derivatives of Siegel Eisenstein series of genus 2, suggesting an arithmetic version of theta lifting and the Siegel--Weil formula (see, for example, \cites{kudla2002modular,kudla2002derivatives}). This was later further developed in his joint work with Rapoport and Yang \cites{kudla1997central,kudla1999arithmetic,kudla1999derivative,kudla2000height,kudla2004derivatives,kudla2006modular}.

    For the higher dimensional case, in a series of papers starting from the late 1990s, Kudla and Rapoport developed the theory of special cycles on integral models of Shimura varieties for GSpin groups in lower rank cases and for unitary groups of arbitrary ranks \cites{kudla2011special,kudla2014special}. They also studied special cycles on the relevant Rapoport--Zink spaces over non-Archimedean local fields. In particular, they formulated a conjecture relating the arithmetic intersection number of special cycles on the unitary Rapoport--Zink space to the first derivative of local Whittaker functions \cite{kudla2011special}*{Conjecture 1.3}.

    In his thesis work \cites{liu2011arithmetic,liu2011arithmetica}, Liu studied special cycles as elements in the Chow group of the unitary Shimura variety over its reflex field (rather than in the arithmetic Chow group of a certain integral model) and the Beilinson--Bloch height of the arithmetic theta lifting (rather than the Gillet--Soul\'e height). In particular, in the setting of unitary groups, he proposed an explicit conjectural formula for the Beilinson--Bloch height in terms of the central \(L\)-derivative and local doubling zeta integrals. Such a formula is completely parallel to the Rallis inner product formula \cite{rallis1984injectivity}, which computes the Petersson inner product of the global theta lifting and hence was named arithmetic inner product formula in \cite{liu2011arithmetic} and can be regarded as a higher dimensional generalisation of the Gross--Zagier formula. 
    
    In the case of \(\rU(1,1)\) over an arbitrary CM extension, such a conjectural formula was completely confirmed in \cite{liu2011arithmetica}, while the case for \(\rU(r,r)\) with \(r>1\) is significantly harder. 
    
    Recently, the Kudla--Rapoport conjecture for unitary groups over local fields has been proved in several cases:
    \begin{itemize}
      \item 
      If the quadratic extension is unramified, W. Zhang and C. Li proved the Kudla--Rapoport conjecture for self-dual and almost self-dual levels in \cite{li2021kudla} using a local method. For all maximal parahoric levels, Y. Luo proved the Kudla--Rapoport conjecture in \cite{luo2025kudlarapoport} using a global method.
      \item 
      If the quadratic extension is ramified, C. Li and Y. Liu proved the Kudla--Rapoport conjecture for the so-called exotic smooth model in even rank \cite{li2022chow} and H. Yao extended the result to odd rank in \cite{yao2024kudlarapoport}. Q. He, C. Li, Y. Shi and T. Yang proved the Kudla--Rapoport conjecture for Kr\"amer model in \cite{he2023proof}.
    \end{itemize}
    Their results make it possible to attack the cases for higher rank groups. In \cite{li2021chow}, Li and Liu proved that for certain cuspidal automorphic representations \(\pi\) of \(\rU(r,r)\), if the central derivative \(L'(1/2,\pi)\) is nonvanishing, then the \(\pi\)-nearly isotypic localisation of the Chow group of a certain unitary Shimura variety over its reflex field does not vanish. Their work proved part of the Beilinson--Bloch conjecture for Chow groups and \(L\)-functions (see \cite{li2021chow}*{Section 1} for a precise formulation in our setting). Moreover, using the modularity of Kudla's generating series proved by Raum \cite{raum2026geometrica}, they further proved the arithmetic inner product formula relating \(L'(1/2,\pi)\) and the height of arithmetic theta liftings. In the latter article \cite{li2022chow}, they improve the main results from \cite{li2021chow} in two directions: First, they allow ramified places in the CM extension \(\dE/\dF\) at which they consider representations that are spherical with respect to a certain special maximal compact subgroup, by formulating and proving an analogue of the Kudla--Rapoport conjecture for exotic smooth Rapoport--Zink spaces in even rank. Second, they lift the restriction on the components at split places of the automorphic representation, by proving a more general vanishing result on certain cohomology of integral models of unitary Shimura varieties with Drinfeld level structures. 
    
    However, because the (nontrivial) exotic smooth model only exists when the local root number is \(+1\) and the corresponding spherical representations which should be considered have not been classified in the literature, the results in \cite{li2022chow} still exclude the case when the local root number is \(-1\) at ramified places. Fortunately, the Kudla--Rapoport conjecture for Kr\"amer model is proved in \cite{he2023proof} and then we are able to treat the ramified places with local root number \(-1\) as well.

\subsection{Main results}
Let \(\dE/\dF\) be a CM extension of number fields with nontrivial Galois automorphism \(\tc\). We denote the set of finite places of \(\dF\) (resp. \(\dE\)) by \(\tV_\dF^{\fin}\) (resp. \(\tV_\dE^{\fin}\)). We denote by \(\tV_\dF^\spl\) (resp. \(\tV_\dF^\ram\), \(\tV_\dF^{\text{int}}\)) the set of places of \(\dF\) that split (resp. ramify, inert) in \(\dE\). We denote by \(\tV_\dF\) the set of places of \(\dF\), and by \(\tV_\dF^{(p)}\) (resp. \(\tV_\dF^{(\infty)}\)) the subset of places above a rational prime \(p\) (resp. archimedean places).

Let \(n=2r\) be an even positive integer and equip \(W_r:=\dE^n\) with the skew-hermitian form defined by the matrix \(\begin{pmatrix} & 1_r \\ -1_r & \end{pmatrix}\). We denote by \(G_r:=\rU(W_r)\) the corresponding quasi-split unitary group over \(\dF\). For each place \(v\) of \(\dF\), we denote the stabilizer of the lattice \(\cO_{\dE_v}^n\) by \(K_{r,v}\subset G_r(\dF_v)\) which is a special maximal compact subgroup. 

\begin{definition}\label{def:heart}
    We define the subset $\tV_\dF^\heartsuit$ of $\tV_\dF^\spl\cup\tV_\dF^{\text{int}}$ consisting of $v$ satisfying that for every $v'\in\tV_\dF^{(p)}\cap\tV_\dF^\ram$, where $p$ is the underlying rational prime of $v$, the subfield of $\ol{\dF_v}$ generated by $\dF_v$ and the Galois closure of $\dE_{v'}$ is unramified over $\dF_v$.
\end{definition}

\begin{remark}
  As explained in \cite{li2022chow}*{Remark~1.2}, this condition is to ensure that the reflex field of certain CM type is not ramified over \(\dE_{v}\).
\end{remark}

\begin{assumption}\label{asp:main}
Suppose that $\dF\neq\dQ$ and $\tV_\dF^\spl$ contains all 2-adic places. We consider a cuspidal automorphic representation $\pi$ of $G_r(\dA_\dF)$ realized on a space $\cV_\pi$ of cusp forms, satisfying:
\begin{enumerate}
  \item For every $v\in\tV_\dF^{(\infty)}$, $\pi_v$ is the holomorphic discrete series representation of Harish-Chandra parameter $\{\tfrac{n-1}{2},\tfrac{n-3}{2},\dots,\tfrac{3-n}{2},\tfrac{1-n}{2}\}$ (see \cite{li2021chow}*{Remark~1.4(1)}).

  \item For every $v\in\tV_\dF^\ram$, $\pi_v$ is either spherical or regularly almost spherical (see Definition \ref{def:almost_spherical_representation}) with respect to $K_{r,v}$.

  \item For every $v\in\tV_\dF^\text{int}$, $\pi_v$ is either unramified or almost unramified (see \cite{li2021chow}*{Remark~1.4(3)}) with respect to $K_{r,v}$; moreover, if $\pi_v$ is almost unramified, then $v$ is unramified over $\dQ$.

  \item For every $v\in\tV_\dF^{\fin}$, $\pi_v$ is tempered.

  \item We have $\tR_\pi\cup\tS_\pi\subseteq\tV_\dF^\heartsuit$ (Definition \ref{def:heart}), where
     \begin{itemize}
         \item $\tR_\pi\subseteq\tV_\dF^\spl$ denotes the (finite) subset for which $\pi_v$ is ramified,

         \item $\tS_\pi\subseteq\tV_\dF^\text{int}$ denotes the (finite) subset for which $\pi_v$ is almost unramified.
        \end{itemize}
  \item Let $\tS'_\pi\subseteq\tV_\dF^\ram$ denote the (finite) subset for which $\pi_v$ is regularly almost spherical with respect to $K_{r,v}$.
\end{enumerate}
\end{assumption}

Let \((\pi,\cV_\pi)\) be as in Assumption \ref{asp:main}. Denote by $L(s,\pi)$ the doubling $L$-function associated to \(\pi\) \cite{yamana2014lfunctions}. The cuspidal automorphic representation $\pi$ determines a hermitian space $V_\pi$ over $\dA_\dE$ of rank $n$ via local theta dichotomy as follows (see \cite{liu2022theta}): for every place $v$ of $\dF$, there is a unique (up to isomorphism) hermitian space $V_{\pi,v}$ over $\dE_v$ of rank $n$ such that the local theta lifting of $\pi_v$ to $V_{\pi,v}$ is nontrivial. For a finite place \(v\) of \(\dF\), \(V_{\pi,v}\) is split (resp. non-split) if and only if the local root number \(\epsilon_v(\pi)=1\) (resp. -1) and for \(v\mid \infty\), \(V_{\pi,v}\) is positive definite. After patching these local spaces together, we obtain the desired hermitian space $V_\pi$ over $\dA_\dE$. It is called coherent (resp. incoherent) if it is a base change of a global hermitian space over $\dE$ (resp. not). Then we have $\epsilon(\pi)=(-1)^{r[\dF:\dQ]+|\tS_\pi\cup\tS'_\pi|}$ for the global (doubling) root number, so that the vanishing order of $L(s,\pi)$ at the center $s=\tfrac{1}{2}$ has the same parity as $r[\dF:\dQ]+|\tS_\pi\cup \tS'_\pi|$. Moreover, it is known that \(V_\pi\) is coherent if and only if the global root number \(\epsilon(\pi)=+1\). When \(\epsilon(\pi)=1\), we have the global theta lifting of \(\pi\) and the famous Rallis inner product formula \cite{rallis1984injectivity} which computes the Petersson inner product of the global theta lifting in terms of the central value \(L(\tfrac{1}{2},\pi)\). When \(\epsilon(\pi)=-1\), the space \(V_\pi\) is incoherent, and there is no global theta lifting of \(\pi\). However, in this case, one can consider the arithmetic theta lifting of \(\pi\), which is a space of algebraic cycles on the Shimura variety associated to \(V_\pi\) proposed by Kudla. Liu \cite{liu2011arithmetic} conjectured an arithmetic inner product formula which relates the Beilinson--Bloch height pairing of the arithmetic theta lifting to the central derivative \(L'(\tfrac{1}{2},\pi)\) and proved it for \(\rU(1,1)\) over an arbitrary CM extension in \cite{liu2011arithmetica}. In \cites{li2021chow,li2022chow}, Li and Liu verified the arithmetic inner product formula, under certain hypothesis.

Now suppose that $r[\dF:\dQ]+|\tS_\pi\cup\tS'_\pi|$ is odd hence $\epsilon(\pi)=-1$, which is equivalent to that $V_\pi$ is incoherent. In what follows, we take $V=V_\pi$ in the context of \cite{li2021chow}*{Conjecture~1.1}, hence $H=\rU(V_\pi)$. Let $\tR$ be a finite subset of $\tV_\dF^{\fin}$. We fix a `special'\footnote{In our setting, this group is indeed not special in the terminology of Bruhat--Tits theory, but the abstract Hecke algebra is still commutative.} maximal compact subgroup $L^\tR$ of $H(\dA_\dF^{\infty,\tR})$ that is the stabilizer of a lattice $\Lambda^\tR$ in $V\otimes_{\dA_\dF}\dA_\dF^{\infty,\tR}$. For a field $\dL$, we denote by $\dT^{\tR}_\dL$ the (abstract) Hecke algebra $\dL[L^{\tR}\backslash H(\dA_\dF^{\infty,\tR})/L^{\tR}]$, which is a commutative $\dL$-algebra. When $\tR$ contains $\tR_\pi$, the cuspidal automorphic representation $\pi$ gives rise to a character
\[
\chi_\pi^{\tR}\colon\dT^{\tR}_{\dQ^\ac}\to\dQ^\ac,
\]
where $\dQ^\ac$ denotes the subfield of $\dC$ of algebraic numbers; and we put $\fm_\pi^{\tR}\coloneqq\Ker\chi_\pi^{\tR}$, which is a maximal ideal of $\dT^{\tR}_{\dQ^\ac}$. 

\begin{remark}
  Now we explain more on the regularity assumption:
  \begin{enumerate}[(1)]
    \item 
    The existence of this character is given by \cite{li2021chow}*{Definition 6.8} and \cite{liu2022theta}*{Definition 5.3, Theorem 1.1(1)} at unramified primes. At ramified places the original construction in \cite{liu2022theta} needs some modification for \(\epsilon_v=-1\) and that is why we put the regularity assumption in Assumption \ref{asp:main} (2). We put the regularity to make  sure that the theta lifting of \(\pi_v\) is spherical with respect to a special maximal compact subgroup. However, we do not necessarily define \(\chi_\pi^\tR\) at those places in our application but we still define it for the integrity of the definition, see Remark \ref{rem:ideal_redundant} for more details\footnote{If we drop the definition of \(\chi_\pi^\tR\) at those places, the statement for main results need to be modified.};
    \item This is also why our doubling zeta integrals on the right hand side of Corollary \ref{co:aipf} are more complicated at those places \(v\in\tS'_\pi\): we are not using the right test vectors yet, they just happen to lie in the same representation under our regularity assumption.
  \end{enumerate}
  
  We hope to remove this condition in future work by giving a complete classification of the representations spherical with respect to the stabilizer of a unimodular or an almost \(\pi\)-modular lattice. 
\end{remark}

In what follows, we will fix an arbitrary embedding $\biota\colon \dE\hookrightarrow\dC$ and denote by $\{X_L\}$ the system of unitary Shimura varieties of dimension $n-1$ over $\biota(\dE)$ indexed by open compact subgroups $L\subseteq H(\dA_\dF^\infty)$. The following is the first main theorem of this article.

\begin{theorem}\label{thm:main}
  Let \((\pi,\cV_\pi)\) be as in Assumption \ref{asp:main} with \(r[\dF:\dQ]+|\tS_\pi\cup\tS'_\pi|\) odd\footnote{We do not need to  assume \cite{li2021chow}*{Hypothesis 6.6} by \cite{kisin2021stable}.}. If \(L'(\frac{1}{2},\pi)\neq0\), that is, \(\ord_{s=\frac{1}{2}}L(s,\pi)=1\), then as long as \(\tR_\pi\subset \tR\) and \(|\tR\cap \tV_\dF^\spl\cap\tV^\heartsuit_\dF|\geqslant 2\), the nonvanishing
  \[
    \varinjlim_{L_{\tR}}(\CH^r(X_{L_\tR L^\tR})_{\dQ^\ac}^0)_{\fm_\pi^{\tR}}\neq0
  \]
  holds, where the colimit is taken over open compact subgroups \(L_{\tR}\subseteq H(\dF_\tR)\).
\end{theorem}

  Our remaining results rely on the modularity of Kudla's generating series proved by Raum \cite{raum2026geometrica}.

\begin{theorem}\label{thm:aipf}
Let $(\pi,\cV_\pi)$ be as in Assumption \ref{asp:main} with $r[\dF:\dQ]+|\tS_\pi\cup\tS'_\pi|$ odd. By the modularity of Kudla's generating series \cite{raum2026geometrica},
\begin{enumerate}
  \item For every test vector
     \begin{itemize}
       \item $\varphi_1=\otimes_v\varphi_{1v}\in\cV_{\pi}$ and $\varphi_2=\otimes_v\varphi_{2v}\in\cV_{\pi}$ such that for every $v\in\tV_\dF^{(\infty)}$, $\varphi_{1v}$ and $\varphi_{2v}$ have the lowest weight and satisfy $\langle\varphi_{1v}^\tc,\varphi_{2v}\rangle_{\pi_v}=1$,

       \item $\phi^\infty_1=\otimes_v\phi^\infty_{1v}\in\sS(V^r\otimes_{\dA_\dF}\dA_\dF^\infty)$ and $\phi^\infty_2=\otimes_v\phi^\infty_{2v}\in\sS(V^r\otimes_{\dA_\dF}\dA_\dF^\infty)$,
     \end{itemize}
     the identity
     \[
     \langle\Theta_{\phi^\infty_1}(\varphi_1),\Theta_{\phi^\infty_2}(\varphi_2)\rangle_{X,\dE}^\natural=
     \frac{L'(\tfrac{1}{2},\pi)}{b_{2r}(0)}\cdot C_r^{[\dF:\dQ]}
     \cdot\prod_{v\in\tV_\dF^{\fin}}\fZ^\natural_{\pi_v,V_v}(\varphi^\tc_{1v},\varphi_{2v},\phi_{1v}^\infty\otimes(\phi_{2v}^\infty)^\tc)
     \]
     holds. Here,
     \begin{itemize}
      \item $\Theta_{\phi^\infty_i}(\varphi_i)\in\varinjlim_L\CH^r(X_L)^0_\dC$ is the arithmetic theta lifting, which is well-defined by the modularity of Kudla's generating series \cite{raum2026geometrica};

       \item $\langle\Theta_{\phi^\infty_1}(\varphi_1),\Theta_{\phi^\infty_2}(\varphi_2)\rangle_{X,\dE}^\natural$ is the normalized height pairing, which is constructed based on Beilinson's notion of height pairing;

       \item $b_{2r}(0)$ is the same as that in \cite{li2022chow}*{Notation 4.1 (F4)};

       \item $C_r=(-1)^r2^{r(r-1)}\pi^{r^2}\frac{\Gamma(1)\cdots\Gamma(r)}{\Gamma(r+1)\cdots\Gamma(2r)}$, which is the exact value of a certain archimedean doubling zeta integral; and

       \item $\fZ^\natural_{\pi_v,V_v}(\varphi^\tc_{1v},\varphi_{2v},\phi_{1v}^\infty\otimes(\phi_{2v}^\infty)^\tc)$ is the normalized local doubling zeta integral \cite{li2021chow}*{Section~3}, which equals $1$ for all but finitely many $v$.
     \end{itemize}

  \item In the context of \cite{li2021chow}*{Conjecture~1.1}, take ($V=V_\pi$ and) $\tilde\pi^\infty$ to be the theta lifting of $\pi^\infty$ to $H(\dA_\dF^\infty)$. If $L'(\tfrac{1}{2},\pi)\neq 0$, that is, $\ord_{s=\frac{1}{2}}L(s,\pi)=1$, then
     \[
     \Hom_{H(\dA_\dF^\infty)}\left(\tilde\pi^\infty,\varinjlim_{L}\CH^r(X_L)^0_\dC\right)\neq\{0\}
     \]
     holds.
\end{enumerate}
\end{theorem}

In the case where $\tR_\pi=\emptyset$, we have a very explicit height formula for test vectors that are new everywhere.

\begin{corollary}\label{co:aipf}
Let $(\pi,\cV_\pi)$ be as in Assumption \ref{asp:main} with $r[\dF:\dQ]+|\tS_\pi\cup\tS'_\pi|$ odd. By the modularity of Kudla's generating series \cite{raum2026geometrica}, in the situation of Theorem \ref{thm:aipf}(1), suppose further that
\begin{itemize}
  \item $\tR_\pi=\emptyset$;

  \item $\varphi_1=\varphi_2=\varphi\in\cV_\pi^{[r]\emptyset}$ (see Notation \ref{notation:automorphic_forms} for the precise definition of the one-dimensional space $\cV_\pi^{[r]\emptyset}$ of holomorphic new forms) such that for every $v\in\tV_\dF$, $\langle\varphi_v^\tc,\varphi_v\rangle_{\pi_v}=1$; 

  \item $\phi^\infty_1=\phi^\infty_2=\phi^\infty$ such that for every $v\in\tV_\dF^{\fin}$, $\phi^\infty_v=\CF_{(\Lambda_v^\emptyset)^r}$.
\end{itemize}
Then the identity
\[
\begin{aligned}
\langle\Theta_{\phi^\infty}(\varphi),\Theta_{\phi^\infty}(\varphi)\rangle_{X,\dE}^\natural
=&(-1)^r\cdot \frac{L'(\tfrac{1}{2},\pi)}{b_{2r}(0)}\cdot |C_r|^{[\dF:\dQ]} \\
&\cdot\prod_{v\in\tS_\pi}\frac{q_v^{r-1}(q_v+1)}{(q_v^{2r-1}+1)(q_v^{2r}-1)}\cdot\prod_{v\in\tS'_\pi}\frac{(q_v^{r-1}+1)}{q_v^{r^2-1}(q_v+1)(q_v^r-1)}
\end{aligned}
\]
holds.
\end{corollary}

\subsection{Structure of this paper}
    The proof is organized in three interacting layers (representation theory, geometry of the integral models, and arithmetic assembly). The new ingredient beyond \cites{li2021chow,li2022chow} is the allowance of ramified places with local root number $-1$, which forces new almost spherical local data, new normalized doubling zeta integrals, and a semi-global Kr\"amer-type integral model retaining enough regularity and vanishing to control special cycle intersections.

    There are several obstacles to overcome in order to resolve the above problems:
    \begin{enumerate}[(1)]
      \item 
      First, we need to study the spherical representations of unitary groups with respect to the stabilizer of a unimodular lattice and their theta lifting when the quadratic extension is ramified. This is different from the spherical representations considered in \cite{li2022chow}, which are with respect to the stabilizer of a \(\varpi\)-modular lattice because in a split Hermitian space, a unimodular lattice does not give a special vertex in the Bruhat--Tits theory \cite{pappas2008twisted}. 
      \item 
      Second, the section defined in \cite{he2023kudla}*{Section 12} involves some error terms so it will be hard to detect its behavior in the Weil representation and the corresponding doubling zeta integral is hard to compute explicitly. 
      \item
      Third, we still need to prove some vanishing result of the cohomology of the integral model of the Shimura variety.
    \end{enumerate}

    In this paper we combine the methods in \cites{li2021chow,li2022chow} and the result of \cite{he2023proof} to extend the main results of \cite{li2021chow} to allow ramified places with local root number \(-1\) and attack the above obstacles in the following manner:
    \begin{enumerate}[(1)]
      \item We study the spherical representations of unitary groups with respect to the stabilizer of a unimodular lattice in both quasi-split and non-quasi-split unitary groups and prove that, if we put some mild constraints on the Satake parameter of the component of the cuspidal representation at some ramified places,  such representations, at least under these constraints, are the same as those spherical with respect to a special maximal compact subgroup. Then we can study their theta liftings following the method \cite{liu2022theta}. We also prove a version of Satake isomorphism for such non-special maximal compact subgroups.
      \item We compute the doubling zeta integral with respect to the leading term of the section defined in \cite{he2023kudla}*{Section 12} and show that their critical values are the same.
      \item We give the moduli description of the special fiber of the corresponding integral model and prove the required vanishing result of the cohomology of the integral model of the Shimura variety following the method in \cite{li2021chow}.
    \end{enumerate}
    
\medskip
\noindent\textbf{(I) Representation-theoretic layer (Sections \ref{sec:quasi-split}, \ref{sec:non-quasi-split} and \ref{sec:doubling}).}
\begin{itemize}
  \item We refine unramified (and Iwahori-spherical) principal series for quasi-split unitary groups. We give explicit Iwahori bases, intertwining operator normalizations, and Casselman-type $c$-functions needed to track local constants in doubling integrals.
  \item We parallel these constructions for non--quasi-split unitary groups. Although the doubled group is split, the non--quasi-split calculations enter through test vectors at ramified root number $-1$ places and nearby space considerations.
  \item We isolate ``almost spherical'' vectors and compute their doubling zeta integrals in all lattice configurations (\(\pi\)-modular, almost \(\pi\)-modular, unimodular; split / non-split) obtaining Propositions \ref{prop:zeta_spherical_self_dual}, \ref{prop:zeta_spherical_unimodular_split}, \ref{prop:zeta_almost_spherical_almost_self_dual}, \ref{prop:zeta_almost_spherical_unimodular_non_split}.
  \item These explicit normalized local formulas feed simultaneously into (a) local theta correspondence (non-vanishing and identification of $V_\pi$) and (b) the analytic side of the Rallis inner product formula giving the $L'(\tfrac{1}{2},\pi)$ factor with precise local constants, including the new ramified contributions.
\end{itemize}

\noindent\textbf{(II) Geometric layer (Section \ref{sec:semi-global}).}
\begin{itemize}
  \item 
  We study a semi-global integral model of RSZ Shimura varieties adapted to a chosen ramified place with local root number $-1$ using a Kr\"amer condition; despite the ramification, it preserve regularity and semistability sufficient for intersection theory.
  \item
  We prove a vanishing result of the \(\ell\)-adic cohomology of this integral model (localized) at certain maximal ideal of the Hecke algebra following the method in \cite{li2021chow}*{Section~9 and Appendix~B}, which allows us to relate the height pairing to intersection numbers of special cycles.
\end{itemize}

\noindent\textbf{(III) Local indices and the Arithmetic inner product formula (Sections \ref{sec:local-indices}--\ref{sec:aipf}).}
\begin{itemize}
  \item We relate local intersection indices to derivatives of Whittaker functions via an enhanced local Siegel--Weil identity (Lemma \ref{lemma:local_siegel_weil}, Proposition \ref{prop:local_siegel_weil}), expressing them in terms of $b_{2r,\underline{u}}(0)$, lattice volumes, and $W'_{T^\square}$.
  \item The additional ``error'' terms at places in $\tS'_\pi$ (ramified almost spherical) are shown to vanish globally after summation by the Rallis inner product formula once the almost spherical sections from (I) are inserted with their prescribed normalizations.
  \item Putting the geometric decomposition together with the analytic (Rallis) identity yields the arithmetic inner product formula (Theorem \ref{thm:aipf}), giving the normalized height pairing as $L'(\tfrac{1}{2},\pi)$ times an explicit product of local constants (now including new ramified factors).
  \item The appendix collects some combinatorial identities on $q$-binomial coefficients which are related to the computation.
\end{itemize}
\subsection{Acknowledgements}
    The author wishes to thank my advisor Yifeng Liu who suggested this problem, provided encouragement and many helpful ideas for this work. The author gratefully acknowledges Yu Luo and Wei Zhang for the valuable insights gained from for their helpful discussions and corrections on an earlier draft of this paper which greatly improved the exposition. The author would also like to thank Jiu-Kang Yu and Xiaoxiang Zhou for their helpful comments on the representation theory of \(p\)-adic groups. Part of this work was carried out when the author was a visiting student at Massachusetts Institute of Technology, and he would like to thank the institution for its hospitality and the support of Qiushifeiying Program of Zhejiang University for this visit.
\section{Notations}
\label{sec:notations}
In this section, we will introduce some notations that will be used throughout this article.

\subsection{Notations for Local Theory}
\label{subsec:notations_local}

We shall always assume that \(p\) is an odd prime number.

If \(K\) is a \(p\)-adic field with its residue field \(k_K\), we will denote the cardinality of \(k_K\) by \(q_K\) and its characteristic by \(p\). We will denote the ring of integers of \(K\) by \(\cO_K\), the maximal ideal of \(\cO_K\) by \(\fm_K\). We will fix a uniformizer of \(K\) and denote it by \(\varpi_K\). The multiplicative valuation \(|\cdot|_K\) on \(K\) is normalized so that \(|\varpi_K|_K=q_K^{-1}\).

Let \(E/F\) be a quadratic extension of \(p\)-adic fields. We will denote the non-trivial element in \(\Gal(E/F)\) by \(\tc\). We will denote the ring of integers of \(E\) (resp. \(F\)) by \(\cO_E\) (resp. \(\cO_F\)), the maximal ideal of \(\cO_E\) (resp. \(\cO_F\)) by \(\fp_E\) (resp. \(\fp_F\)). We will fix an additive character \(\psi_F:F\rightarrow\dC^\times\) with conductor \(\cO_F\).

In most cases of this paper, \(E/F\) is ramified, and we will fix a uniformizer \(\varpi_E\) of \(E\) such that \(\varpi_E^2=\varpi_F\), then \(E=F(\varpi_E)\) and \(\varpi_E^c=-\varpi_E\). We will usually denote the cardinality of the residue fields by \(q=q_E=q_F\) as they are the same.

\begin{enumerate}[(LS1)]
    \item 
    \label{notion:local_skew_hermitian}
    We will denote a skew-Hermitian space with respect to \(E/F\) together with its skew-Hermitian form by \((W,\langle\cdot,\cdot\rangle_W)\). The unitary group associated to \((W,\langle\cdot,\cdot\rangle_W)\) will be denoted by \(\rU(W,\langle\cdot,\cdot\rangle_W)\) or \(\rU(W)\) if there is no confusion. 
    \item\label{notion:standard_skew_hermitian}
    We equip \(W_r:=\bigoplus_{i=1}^{2r}Ee_i\) with the skew-Hermitian form \(\langle\cdot,\cdot\rangle_r\) defined by the matrix
    \[
        \begin{pmatrix}
            0&I_r\\
            -I_r&0
        \end{pmatrix},
        \]
        it is called the standard split skew-Hermitian space of dimension \(2r\). If \(r=0\), then \(W_0=0\). The unitary group \(\rU(W_r)\) will also be denoted by \(G_r\). \(G_r\) is a quasi-split unitary group defined over \(F\) of rank \(r\).
    \item        
    \label{notion:vee_dual_lattice_skew_hermitian}
    For a lattice \(\Lambda\subset W_r\) we set the \(\vee\)-dual lattice by 
    \[\Lambda^\vee:=\left\{x\in W_r\mid\Tr_{E/F}\langle x,\Lambda\rangle_r\subset\cO_F\right\}.\] 

    \item 
    \label{notion:sharp_dual_lattice_skew_hermitian}
    \[\Lambda^\sharp=\left\{x\in W_r\mid\langle x,\Lambda\rangle_r\subset\cO_E\right\}.\]
    Because we are considering ramified extension where \(E=F(\sqrt{\varpi_F})\) by our definition, we have \(\Tr_{E/F}(x)\in\cO_F\Leftrightarrow x\in \varpi_E^{-1}\cO_E\) and thus \(\Lambda^\vee=\varpi_E^{-1}\Lambda^\sharp\).
    \item
    \label{notion:vee_integral_lattice_skew_hermitian}
    A lattice \(\Lambda\) is called \(\vee\)-integral if \(\Lambda\subset\Lambda^\vee\), in other words, \(\llangle\Lambda,\Lambda\rrangle_r\subset\cO_F\) where \(\llangle\cdot,\cdot\rrangle_r=\Tr_{E/F}(\langle\cdot,\cdot\rangle)\). 
    \item
    \label{notion:sharp_integral_lattice_skew_hermitian}
    A lattice \(\Lambda\) is called \(\sharp\)-integral if \(\Lambda\subset\Lambda^\sharp\), in other words, \(\langle\Lambda,\Lambda\rangle_r\subset\cO_E\).
    \item
    \label{notion:standard_lattice_skew_hermitian}
    The standard lattice \(\bigoplus_{i=1}^{2r}\cO_Ee_i\) satisfies \(\Lambda^\sharp=\Lambda\) then \(\Lambda^\vee=\varpi_E^{-1}\Lambda\). This lattice will be called the standard lattice in \(W_r\) and will be denoted by \(\Lambda_r\). 
    \item
    \label{notion:semi_standard_lattice_skew_hermitian}
    There is another lattice \(\Lambda_r'\) in \(W_r\) defined by
    \[
        (\bigoplus_{i=1}^r\cO_Ee_i)\oplus(\bigoplus_{i=r+1}^{2r}\varpi_E\cO_Ee_i)
    \]\footnote{It is equivalent to consider the lattice \((\bigoplus_{i=1}^r\varpi_E\cO_Ee_i)\oplus(\bigoplus_{i=r+1}^{2r}\cO_Ee_i)\) since they are conjugate in \(G_r\), but the choice of positive roots may need to be modified.}
        It satisfies \(\Lambda_r'^\vee=\varpi_E^{-2}\Lambda_r'\) and \(\Lambda_r'^\sharp=\varpi_E^{-1}\Lambda_r'\). It will be called the semi-standard lattice in \(W_r\).
        
    \end{enumerate}
\begin{remark}
    If \(E/F\) is unramified, then \(\vee\)-dual is the same as \(\sharp\)-dual
\end{remark}
\begin{remark}
    In fact, the condition of \(\Lambda'_r\) seems absurd, we will not consider it in \(W_r\), we will consider the corresponding lattice \(\Lambda_{V_r^+}'\) in \(V_r^+\), see \ref{unimodular_lattice_split} for a definition.
\end{remark}

\begin{enumerate}[(LH1)]
    \item 
    \label{notion:local_split_hermitian}
    We will denote a Hermitian space with respect to a ramified quadratic extension \(E/F\) together with its Hermitian form by \((V,(\cdot,\cdot)_V)\). The unitary group associated to \((V,(\cdot,\cdot)_V)\) will be denoted by \(\rU(V,(\cdot,\cdot)_V)\) and \(\rU(V)\) if there is no confusion.
    \item 
    \label{notion:vee_dual_lattice_hermitian}
    For a lattice \(\Lambda\) in Hermitian space \((V,(\cdot,\cdot)_V)\), we will also consider the \(\vee\)-dual lattice defined by
    \[
        \Lambda^\vee:=\left\{x\in V\mid \Tr_{E/F}(x,\Lambda)_V\subset\cO_F\right\}
        \]
        which is the usual notion used in \cites{liu2022theta,li2021kudla}.
    \item
    \label{notion:sharp_dual_lattice_hermitian}
    And the \(\sharp\)-dual lattice defined by
    \[
        \Lambda^\sharp:=\left\{x\in V\mid (x,\Lambda)_V\subset\cO_E\right\}
    \]
        which is the notion used in \cites{he2023proof}.
            
    For the same reason in the skew-Hermitian case, we have \(\Lambda^\vee=\varpi_E^{-1}\Lambda^\sharp\).
    \item
    \label{notion:vee_integral_hermitian}
    Similar to the skew-Hermitian case, we say that a lattice \(\Lambda\) is \(\vee\)-integral if \(\Lambda\subset\Lambda^\vee\).
    \item
    \label{notion:sharp_integral_hermitian}
        We say that a lattice \(\Lambda\) is \(\sharp\)-integral if \(\Lambda\subset\Lambda^\sharp\).
    \item\label{notion:unimodular_lattice}
        We say that a lattice \(\Lambda\) is unimodular if \(\Lambda=\Lambda^\sharp\).
    \item\label{notion:self_dual_lattice}
        If the Hermitian space \(V\) is \emph{split}, we say that a lattice \(\Lambda\) is \(\pi\)-modular if \(\Lambda=\varpi_E^{-1}\Lambda^\sharp=\Lambda^\vee\).
    \item\label{notion:almost_self_dual_lattice}
        If the Hermitian space \(V\) is \emph{non-split}. We say that a lattice \(\Lambda\) is almost \(\pi\)-modular if 
        \[\Lambda^\sharp\subset\Lambda\subset\varpi_E^{-1}\Lambda^\sharp=\Lambda^\vee,\quad\text{and }\dim_{k_E}\Lambda^\vee/\Lambda=2\]
    \item\label{notion:standard_split_hermitian}
        We equip \(V_r^+:=\bigoplus_{i=1}^{2r}Ev_i\) with the Hermitian form \((\cdot,\cdot)_r\) given by the matrix
    \[
        \begin{pmatrix}
            0&\varpi_E^{-1}I_r\\
            -\varpi_E^{-1}I_r&0
        \end{pmatrix},
    \]
    it is called the standard split hermitian space of dimension \(2r\). In fact, \((V_r^+,(\cdot,\cdot)_r)\cong(W_r,\varpi_E^{-1}\langle\cdot,\cdot\rangle_r)\) if we identify \(e_i\) with \(v_i\)\footnote{This identification is indeed \cite{li2022chow}*{Remark 3.1}}, then we also have \(G_r=\rU(V_r^+)\).
    \item \label{self_dual_lattice_split}
    We denote the lattice \[\bigoplus_{i=1}^{2r}\cO_Ev_i,\] by \(\Lambda_{V_r^+}\) it satisfies
    \[
        \Lambda_{V_r^+}=\Lambda_{V_r^+}^\vee,\quad\Lambda_{V_r^+}^\sharp=\varpi_E\Lambda_{V_r^+}
    \]
    this is called the standard \(\pi\)-modular lattice in \(V_r^+\).
    \item \label{unimodular_lattice_split}
    We denote the lattice \[(\bigoplus_{i=1}^r\cO_Ev_i)\oplus(\bigoplus_{i=r+1}^{2r}\varpi_E\cO_Ev_i),\] by \(\Lambda'_{V_r^+}\) it satisfies
    \[
        (\Lambda'_{V_r^+})^\vee=\varpi_E^{-1}\Lambda'_{V_r^+},\quad(\Lambda'_{V_r^+})^\sharp=\Lambda'_{V_r^+}
    \]
    this is called the standard unimodular lattice in \(V_r^+\).
    \item
    \label{notion:standard_non_split_hermitian}
    We fix an element \(\ts\) in \(\cO_F^\times\) so that \(\ts\) is not a square modulo \(\fp_F\). 
                
    For \(r\ge0\), we equip \(V_r^-:=\bigoplus_{i=1}^{2r+2}Ev_i\) with the Hermitian form \((\cdot,\cdot)_r^-\) defined by the matrix
    \[
    \begin{pmatrix}
        0&\varpi_E^{-1}I_r&0&0\\
        -\varpi_E^{-1}I_r&0&0&0\\
        0&0&1&0\\
        0&0&0&-\ts
    \end{pmatrix}
    \]
    We will call this Hermitian space the standard non-split Hermitian space.
    \item \label{almost_self_dual_lattice_non_split}
    We denote the lattice \[\bigoplus_{i=1}^{2r+2}\cO_Ev_i,\] by \(\Lambda_{V_r^-}\) it satisfies
    \[
        \Lambda_{V_r^-}^\vee=(\bigoplus_{i=1}^{2r}\cO_Ev_i)\oplus(\varpi_E^{-1}\cO_E v_{2r+1}\oplus\varpi_E^{-1}\cO_E v_{2r+2}) ,\quad\dim_{k_E}\Lambda_{V_r^-}^\vee/\Lambda_{V_r^-}=2
    \]
    this is called the standard almost \(\pi\)-modular lattice in \(V_r^-\).
    \item \label{unimodular_lattice_non_split}
    We denote the lattice \[(\bigoplus_{i=1}^{r}\varpi_E\cO_Ev_i)\oplus(\bigoplus_{i=r+1}^{2r}\cO_Ev_i)\oplus(\cO_Ev_{2r+1}\oplus\cO_Ev_{2r+2}),\] by \(\Lambda'_{V_r^-}\) it satisfies
    \[
        (\Lambda'_{V_r^-})^\vee=\varpi_E^{-1}\Lambda'_{V_r^-},\quad(\Lambda'_{V_r^-})^\sharp=\Lambda'_{V_r^-}
    \]
    this is called the standard unimodular lattice in \(V_r^-\). Note that this is different from the unimodular lattice in the split case because we need to take care of the choice of positive roots.
\end{enumerate}
\begin{remark}
    The lattices of our interests are the unimodular lattices and (almost) \(\pi\)-modular lattices. If we identify \(G_r\) with \(\rU(V_r^+)\), then the unimodular lattices are the semi-standard lattices, and the \(\pi\)-modular lattices are the standard lattices.
\end{remark}
\begin{enumerate}[(LG1)]
    \item\label{notion:torus}
    For each unitary group defined for the standard spaces with rank \(r\), we will fix the diagonal maximal split torus \(S\cong(\dG_m)^r\) and the diagonal maximal torus \(T\cong(\Res_{E/F}\dG_m)^r\). 
    \item\label{notion:roots_quasi_split}
    We choose the positive roots so that for quasi-split unitary groups (for example, \(G_r:=\rU(W_r)=\rU(V_r^+)\) above), the positive roots are given by\footnote{Our choice of positive roots is made so that they are compatible with the choice of the minimal parabolic subgroup in \cite{liu2022theta}.}
    \[
        \Phi^+=\{\epsilon_j\pm\epsilon_i\mid 1\le i<j\le r\}\cup\{2\epsilon_i\mid 1\le i\le r\},
    \]
    \[
    \Delta=\{\epsilon_{i+1}-\epsilon_{i}\mid 1\le i\le r-1\}\cup\{2\epsilon_1\}
    \]
    \item\label{notion:borel}
    The Borel subgroup \(P_r\) with respect to this choice of positive roots consists of elements of the form
    \[
        \begin{pmatrix}
            a&b\\
            0&^ta^{\tc,-1}
        \end{pmatrix}
    \]
    in which \(a\) is a \emph{lower-triangular} matrix in \(\Res_{E/F}\GL_r\).
    \item\label{notion:siegel_parabolic}
    We denote the Siegel parabolic subgroup of \(G_r\) by \(P_r^0\). It consists of elements of the form
    \[
        \begin{pmatrix}
            a&b\\
            0&^ta^{\tc,-1}
        \end{pmatrix}
    \]
    in which \(a\) is a matrix in \(\Res_{E/F}\GL_r\).
    \item\label{notion:stabilizer_standard}
    The stabilizer, of the standard lattice in \(W_r\), in \(G_r\) will be denoted by \(K_r\), it is a special maximal open compact subgroup of \(G_r\) \cite{li2022chow}*{Section 3}. It is identified with the stabilizer (in \(\rU(V_r^+)\)) of a \(\pi\)-modular lattice in \(V_r^+\).
    \item\label{notion:stabilizer_semi_standard}
    The stabilizer, of the semi-standard lattice in \(W_r\), in \(G_r\) will be denoted by \(L_r\). It can be identified with the stabilizer (in \(\rU(V_r^+)\)) of a unimodular lattice in \(V_r^+\).

    In the ramified case, this subgroup does not contain a special subgroup so we will treat it more carefully.
    \item\label{notion:iwahori}
    The Iwahori subgroup of \(K_r\) corresponding to \(P_r\) will be denoted by \(I_r\).
    \item\label{notion:siegel_parahoric}
    The parahoric subgroup of \(K_r\) corresponding to \(P_r^0\) will be denoted by \(I_r^0\).
    \item\label{notion:roots_non_quasi_split}
    For non-quasi-split unitary groups, the positive roots are given by
    \[
        \Phi^+=\{\epsilon_j\pm\epsilon_i\mid 1\le i<j\le r\}\cup\{\epsilon_i,2\epsilon_i\mid 1\le i\le r\},
    \]
    \[
    \Delta=\{\epsilon_{i+1}-\epsilon_{i}\mid 1\le i\le r-1\}\cup\{\epsilon_1\}.
    \]
    \item\label{notion:minimal_parabolic}
    The minimal parabolic subgroup \(P_r^-\) with respect to this choice of positive roots consists of elements of the form
    \[
        \begin{pmatrix}
            a&b&0\\
            0&^ta^{\tc,-1}&0\\
            0&0&c
        \end{pmatrix}
    \]
    in which \(a\) is a \emph{lower-triangular} matrix in \(\Res_{E/F}\GL_r\), and \(c\) is an element in \(\rU(V_0^-)\).
    \item\label{notion:stabilizer_almost_self_dual}The stabilizer of the almost \(\pi\)-modular lattice \(\Lambda_r^-\) in \(\rU(V_r^-)\) will be denoted by \(K_r^-\). 
    \item\label{notion:stabilizer_unimodular} The stabilizer of unimodular lattice, in \(V_r^-\), in \(\rU(V_r^-)\) will be denoted by \(L_r^-\).
    
    Both \(K_r^-\) and \(L_r^-\) contain a special maximal open compact subgroup of \(\rU(V_r^-)\) \cite{tits1979reductive} with index 2.
    \item The Iwahori subgroup of \(K_r^-\) corresponding to \(P_r^-\) will be denoted by \(I_r^-\).
    \item For any unitary group of rank \(r\) defined above, the Weyl group is always isomorphic to \(\fW_r=\{\pm1\}^r\rtimes S_r\) as an abstract group, with generators \(w_1:=w_{2\epsilon_1},w_{i,i+1}:=w_{\epsilon_{i+1}-\epsilon_i}\). We can then define the length of an element \(w\) and denote it by \(\ell(w)\). The longest element is \(\prod_{i=1}^r w_i=({-1}^r,\id)\in\fW_r\).
    \item \label{notion:weyl_group_decomposition}
    Every element \(w\in\fW_r\) can be uniquely written as \(w=w_I w_\tau\) where
    \(w_I\in\{\pm 1\}^r\) and \(w_\tau\in S_r\), \(I\) is then understood as the set of coordinates with \(-1\). Let \(\ell_{long}(w):=|I|\) be the function sending \(w\) to the number of \(-1\)'s in \(w_I\), which is also the cardinality of \(w_1\) in a reduced representation of \(w\).

\end{enumerate}
\begin{theorem}[Iwasawa decomposition, see, for example, \cite{cartier1979representations}]\label{thm:iwasawa_decomposition}
    We have refined Iwasawa decompositions for each group defined above:
    \[
        G_r=P_rK_r=\bigsqcup_{w\in\fW_r}P_rwI_r,\qquad \rU(V_r^-)=P_r^-K_r^-=P_r^- L_r^-=\bigsqcup_{w\in\fW_r}P_r^-wI_r^-
    \]
    and all are in good positions.
\end{theorem}
\begin{remark}
    Here we use \(w\) to indicate some representative of \(w\in\fW_r\). This decomposition is independent of the choice of the representative by \cite{cartier1979representations}.
\end{remark}
\subsection{Notations for Global Theory}
\label{subsec:notations_global}
        

    For most of the notations we will follow \cite{li2022chow}*{Notation 4.1, 4.2 and 4.3}.

    We state some notations which will be used frequently in this article, some are already used in the introduction.
    \begin{notation} 
        Let \(\dE/\dF\) be a CM extension of number fields, so that \(\tc\) is the non-trivial element in \(\Gal(\dE/\dF)\). We also fix an embedding \(\biota:\dE\hookrightarrow\dC\). We denote by \(\bu\) the archimedean place of \(\dE\) induced by \(\biota\) and regard \(\dE\) as a subfield of \(\dC\) via \(\bu\).
        \begin{enumerate}[(F1)]
        \item\begin{itemize}
        \item For a global field \(\dL\), we denote the set of places (resp. finite places, resp. archimedean places) of \(\dL\) by \(\tV_\dL\) (resp. \(\tV_\dL^\fin\), resp. \(\tV_\dL^{(\infty)}\)).
        
        \item If $S$ is a place or a subset of places of $\dL$, we use $\tV_\dL^S$ to denote the places away from $S$ and $\tV_\dL^{(S)}$ to denote the places above $S$ (Comparing with \cite{li2022chow}*{4.1 (F1)}).
        \item We denote by \(\tV_\dF^\spl\) (resp. \(\tV_\dF^\ram\), \(\tV_\dF^{\text{int}}\)) the set of places of \(\dF\) that split (resp. ramify, inert) in \(\dE\). 
        \item For every place \(u\in\tV_\dE\), we denote by \(\underline{u}\in\tV_\dF\) the underlying place of \(u\) in \(\dF\).
        \item For every \(v\in\tV_\dF^\fin\), we denote by \(\fp_v\) the maximal ideal of \(\cO_{\dF,v}\) and put \(q_v:=|\cO_{\dF,v}/\fp_v|\).
        \item For every \(v\in\tV_\dF\), we put \(\dE_v:=\dE\otimes_{\dF}\dF_v\) and denote by \(||\cdot||_{\dE_v}:\dE_v\rightarrow\dC^\times\) the normalised norm character.
        
        \end{itemize}
        \item Let \(m\ge0\) be an integer.
        \begin{itemize}
        \item We denote by \(\Herm_m\) the subscheme of \(\Res_{\dE/\dF}\Mat_{m,m}\) of \(m\)-by-\(m\) matrices \(b\) satisfying \(^tb^\tc=b\). Put \(\Herm^\circ_m:=\Herm_m\cap\Res_{\dE/\dF}\GL_m\).
        
        \item For every ordered partition \(m=m_1+\cdots+m_s\) with \(m_i\) a positive integer, we denote by \(\partial_{m_1,\cdots,m_s}:\Herm_m\rightarrow\Herm_{m_1}\times\cdots\times\Herm_{m_s}\) the morphism that extracts the diagonal blocks with corresponding ranks.
        
        \item In addition we denote by \(\Herm_m(\dF)^+\) (respectively \(\Herm^\circ_m(\dF)^+\)) the subset of \(\Herm_m(\dF)\) of elements that are totally semi-positive definite (respectively totally positive definite).
        \end{itemize}

        \item For every \(u\in\tV_\dE^{(\infty)}\), we fix an embedding \(\iota_u:\dE\hookrightarrow\dC\) inducing \(u\) (with \(\iota_\bu=\biota\)) and identify \(\dE_u\) with \(\dC\) via \(\iota_u\).
        \item Let \(\eta:=\eta_{\dE/\dF}:\dA_\dF^\times\rightarrow\dC^\times\) be the quadratic character associated to \(\dE/\dF\). For every \(v\in\tV_\dF\) and every positive integer \(m\), put
        \[
            b_{m,v}(s):=\prod_{i=1}^{m} L(2s+i,\eta_v^{m-i}).
        \]
        Put \(b_m(s):=\prod_{v\in\tV_\dF} b_{m,v}(s)\).
        \item Fix an additive character \(\psi_\dF:\dA_\dF/\dF\rightarrow\dC^\times\). For every element \(T\in\Herm_m(\dA_\dF)\), let \(\psi_T:\Herm_m(\dA_\dF)\rightarrow\dC^\times\) be given by \(\psi_T(b):=\psi_\dF(\tr bT)\).
        \item Let \(R\) be a commutative \(\dF\)-algebra. A (skew-)Hermitian space over \(R\otimes_\dF \dE\) is a free \(R\otimes_\dF \dE\)-module \(V\) of finite rank, equipped with a (skew-)Hermitian form \((\cdot,\cdot)_V\) with respect to the involution \(\tc\) that is nondegenerate.
    \end{enumerate}

        \smallskip
        Let $(V,(\cdot,\cdot)_V)$ be a Hermitian space over $\dE$ of dimension $n=2r$ that is totally positive definite.
        \smallskip
    \begin{enumerate}[(H1)]
        \item For every commutative \(\dA_\dF\)-algebra \(R\) and every integer \(m\geq 0\), we denote by 
        \[
            T(x):=((x_i,x_j)_V)_{1\leq i,j\leq m}\in \Herm_m(R)
        \]
        the moment matrix of \(x=(x_1,\cdots,x_m)\in V^m\otimes_{\dA_\dF} R\).

        \item For every place \(v\in\tV_\dF\), we put \(V_v\coloneqq V\otimes_{\dA_\dF}\dF_v\), which is a Hermitian space over \(\dE_v\) and define the local Hasse invariant of \(V_v\) to be \(\epsilon(V_v):=\eta_v((-1)^r\det V_v)\in\{\pm1\}\), which equals 1 for all but finitely many \(v\). In what follows, we will abbreviate \(\epsilon(V_v)\) as \(\epsilon_v\). Recall that \(V\) is coherent (resp. incoherent) if \(\prod_{v\in\tV_\dF}\epsilon_v=1\) (resp. \(-1\)).
    \item Let $v$ be a place of $\dF$ and $m\geq 0$ an integer.
      \begin{itemize}
                \item For $T\in\Herm_m(\dF_v)$, we put $(V^m_v)_T\coloneqq\{x\in V^m_v\res T(x)=T\}$, and
            \[
                        (V^m_v)_{\mathrm{reg}}\coloneqq\bigcup_{T\in\Herm_m^\circ(\dF_v)}(V^m_v)_T.
            \]

        \item We denote by $\sS(V_v^m)$ the space of (complex valued) Bruhat--Schwartz functions on $V_v^m$. When $v\in\tV_\dF^{(\infty)}$, we have the Gaussian function $\phi^0_v\in\sS(V_v^m)$ given by the formula $\phi^0_v(x)=\re^{-2\pi\tr T(x)}$.

        \item We have a Fourier transform map $\widehat{\phantom{a}}\colon \sS(V_v^m)\to \sS(V_v^m)$ sending $\phi$ to $\widehat\phi$ defined by the formula
            \[
            \widehat\phi(x)\coloneqq\int_{V_v^m}\phi(y)\psi_{\dE,v}\left(\sum_{i=1}^m(x_i,y_i)_V\right)\rd y,
            \]
            where $\r{d}y$ is the self-dual Haar measure on $V_v^m$ with respect to $\psi_{\dE,v}$.

        \item In what follows, we will always use this self-dual Haar measure on $V_v^m$.
      \end{itemize}

  \item Let $m\geq 0$ be an integer. For $T\in\Herm_m(\dF)$, we put
      \[
      \mathrm{Diff}(T,V)\coloneqq\{v\in\tV_\dF\res(V^m_v)_T=\emptyset\},
      \]
      which is a finite subset of $\tV_\dF\setminus\tV_\dF^\spl$.
    \item Take a nonempty finite subset $\tR\subseteq\tV_\dF^\fin$ that contains
      \[
      \{v\in\tV_\dF^\ram\res\text{either $2\mid v$, or $v$ is ramified over $\dQ$}\}.
      \]
    Let $\tS$ be the subset of $\tV_\dF^\fin\setminus\tR$ consisting of $v$ such that $\epsilon_v=-1$ (Comparing with \cite{li2022chow}*{4.2 (H5)}).
    \item We fix a $\prod_{v\in\tV_\dF^\fin\setminus\tR}\cO_{\dE_v}$-lattice $\Lambda^\tR$ in $V\otimes_{\dA_\dF}\dA_\dF^{\infty,\tR}$ such that for every $v\in\tV_\dF^\fin\setminus\tR$:
  
  If \(v\) is unramified in \(\dE\), then $\Lambda^\tR_v$ is a subgroup of $(\Lambda^\tR_v)^\vee$ of index $q_v^{1-\epsilon_v}$, where
      \[
      (\Lambda^\tR_v)^\vee\coloneqq\{x\in V_v\res\psi_{\dE,v}((x,y)_V)=1\text{ for every }y\in\Lambda^\tR_v\}
      \]
      is the $\psi_{\dE,v}$-dual lattice of $\Lambda^\tR_v$.

    If \(v\) is ramified in \(\dE\), then \(\Lambda_v^\tR= (\Lambda^\tR_v)^\vee\) if \(\epsilon_v=1\); and \(\Lambda_v^\tR\) is unimodular , i.e. \(\Lambda_v^\tR=(\Lambda_v^\tR)^\sharp= \varpi_{\dE_v}(\Lambda^\tR_v)^\vee\), if \(\epsilon_v=-1\).
    \item Put $H\coloneqq\rU(V)$, which is a reductive group over $\dA_\dF$.

  \item Denote by $L^\tR\subseteq H(\dA_\dF^{\infty,\tR})$ the stabilizer of $\Lambda^\tR$, which is a special maximal subgroup.  We have the (abstract) Hecke algebra away from $\tR$\label{notation:special_hecke_algebra}
      \[
      \dT^\tR\coloneqq\dZ[L^\tR\backslash H(\dA_\dF^{\infty,\tR})/L^\tR],
      \]
      which is a ring with the unit $\CF_{L^\tR}$, and denote by $\dS^\tR$ the subring
      \[
    \varinjlim_{\substack{\tT\subseteq\tV_\dF^\spl\setminus\tR\\|\tT|<\infty}}
    \dZ[(L^\tR)_\tT\backslash H(\dF_\tT)/(L^\tR)_\tT]\otimes\CF_{(L^\tR)^\tT}
      \]
      of $\dT^\tR$.

    \item Suppose that $V$ is \emph{incoherent}, namely, $\prod_{v\in\tV_\dF}\epsilon_v=-1$. For every $u\in\tV_\dE\setminus\tV_\dE^\spl$, we fix a \emph{$u$-nearby space} $\pres{u}{V}$ of $V$, which is a hermitian space over $\dE$, and an isomorphism $\pres{u}{V}\otimes_{\dF}\dA_\dF^{\ul{u}}\simeq V\otimes_{\dA_\dF}\dA_\dF^{\ul{u}}$. More precisely,
      \begin{itemize}
                \item if $u\in\tV_\dE^{(\infty)}$, then $\pres{u}{V}$ is the hermitian space over $\dE$, unique up to isomorphism, that has signature $(n-1,1)$ at $u$ and satisfies $\pres{u}{V}\otimes_{\dF}\dA_\dF^{\ul{u}}\simeq V\otimes_{\dA_\dF}\dA_\dF^{\ul{u}}$;

                \item if $u\in\tV_\dE^\fin\setminus\tV_\dE^\spl$, then $\pres{u}{V}$ is the hermitian space over $\dE$, unique up to isomorphism, that satisfies $\pres{u}{V}\otimes_{\dF}\dA_\dF^{\ul{u}}\simeq V\otimes_{\dA_\dF}\dA_\dF^{\ul{u}}$.
      \end{itemize}
            Put $\pres{u}{H}\coloneqq\rU(\pres{u}{V})$, which is a reductive group over $\dF$. Then $\pres{u}{H}(\dA_\dF^{\ul{u}})$ and $H(\dA_\dF^{\ul{u}})$ are identified.
\end{enumerate}
\end{notation}

\begin{notation}\label{st:g}
Let $m\geq 0$ be an integer. We equip $W_m=\dE^{2m}$ and $\bar{W}_m=\dE^{2m}$ with the skew-hermitian forms given by the matrices $\tw_m$ and $-\tw_m$, respectively.
\begin{enumerate}[label=(G\arabic*)]
  \item Let $G_m$ be the unitary group of both $W_m$ and $\bar{W}_m$. We write elements of $W_m$ and $\bar{W}_m$ in the row form, on which $G_m$ acts from the right.

  \item We denote by $\{e_1,\dots,e_{2m}\}$ and $\{\bar{e}_1,\dots,\bar{e}_{2m}\}$ the natural bases of $W_m$ and $\bar{W}_m$, respectively.

  \item Let $P_m\subseteq G_m$ be the parabolic subgroup stabilizing the subspace generated by $\{e_{m+1},\dots,e_{2m}\}$, and $N_m\subseteq P_m$ its unipotent radical.

    \item We have
     \begin{itemize}
             \item a homomorphism $m\colon\Res_{\dE/\dF}\GL_m\to P_m$ sending $a$ to
          \[
          m(a)\coloneqq
          \begin{pmatrix}
              a &  \\
            & \pres{\rt}{a}^{\tc,-1}
          \end{pmatrix}
          ,
          \]
          which identifies $\Res_{\dE/\dF}\GL_m$ as a Levi factor of $P_m$.

       \item a homomorphism $n\colon\Herm_m\to N_m$ sending $b$ to
          \[
          n(b)\coloneqq
          \begin{pmatrix}
              1_m & b \\
            & 1_m
          \end{pmatrix}
          ,
          \]
          which is an isomorphism.
     \end{itemize}

    \item We define a maximal compact subgroup $K_m=\prod_{v\in\tV_\dF}K_{m,v}$ of $G_m(\dA_\dF)$ in the following way:
    \begin{itemize}
            \item for $v\in\tV_\dF^\fin$, $K_{m,v}$ is the stabilizer of the lattice $\cO_{\dE_v}^{2m}$;

    \item for $v\in\tV_\dF^{(\infty)}$, $K_{m,v}$ is the subgroup of the form
         \[
         [k_1,k_2]\coloneqq\frac{1}{2}
         \begin{pmatrix}
           k_1+k_2   & -ik_1+ik_2 \\
           ik_1-ik_2   & k_1+k_2 \\
         \end{pmatrix}
         ,
         \]
            in which $k_i\in\GL_m(\dC)$ satisfying $k_i\pres{\rt}{k}_i^\tc=1_m$ for $i=1,2$. Here, we have identified $G_m(\dF_v)$ as a subgroup of $\GL_{2m}(\dC)$ via the chosen embedding $\iota_u\colon\dE\hookrightarrow\dC$ inducing $u$ with $v=\underline{u}$.
    \end{itemize}

    \item For every $v\in\tV_\dF^{(\infty)}$, we have a character $\kappa_{m,v}\colon K_{m,v}\to\dC^\times$ that sends $[k_1,k_2]$ to $\det k_1/\det k_2$.\footnote{In fact, neither $K_{m,v}$ nor $\kappa_{m,v}$ depends on the choice of the embedding $\iota_v$ for $v=\underline{u}\in\tV_\dF^{(\infty)}$.}

    \item For every $v\in\tV_\dF$, we define a Haar measure $\r{d}g_v$ on $G_m(\dF_v)$ as follows:
    \begin{itemize}
            \item for $v\in\tV_\dF^\fin$, $\r{d}g_v$ is the Haar measure under which $K_{m,v}$ has volume $1$;

            \item for $v\in\tV_\dF^{(\infty)}$, $\r{d}g_v$ is the product of the measure on $K_{m,v}$ of total volume $1$ and the standard hyperbolic measure on $G_m(\dF_v)/K_{m,v}$.
    \end{itemize}
    Put $\r{d}g=\prod_{v}\r{d}g_v$, which is a Haar measure on $G_m(\dA_\dF)$.

    \item \label{notation:automorphic_forms} We denote by $\cA(G_m(\dF)\backslash G_m(\dA_\dF))$ the space of both $\cZ(\fg_{m,\infty})$-finite and $K_{m,\infty}$-finite automorphic forms on $G_m(\dA_\dF)$, where $\cZ(\fg_{m,\infty})$ denotes the center of the complexified universal enveloping algebra of the Lie algebra $\fg_{m,\infty}$ of $G_m\otimes_{\dF}\dF_\infty$. We denote by
      \begin{itemize}
                \item $\cA^{[r]}(G_m(\dF)\backslash G_m(\dA_\dF))$ the maximal subspace of $\cA(G_m(\dF)\backslash G_m(\dA_\dF))$ on which for every $v\in\tV_\dF^{(\infty)}$, $K_{m,v}$ acts by the character $\kappa_{m,v}^r$,

         \item $\cA^{[r]\tR}(G_m(\dF)\backslash G_m(\dA_\dF))$ the maximal subspace of $\cA^{[r]}(G_m(\dF)\backslash G_m(\dA_\dF))$ on which
             \begin{itemize}
             \item for every $v\in\tV_\dF^\fin\setminus(\tR\cup\tS)$, $K_{m,v}$ acts trivially; and

                             \item for every $v\in\tS$, the standard Iwahori subgroup $I_{m,v}$ acts trivially and $\dC[I_{m,v}\backslash K_{m,v}/I_{m,v}]$ acts by the character $\kappa_{m,v}^-$ (\cite{liu2022theta}*{Definition~2.1}),
             \end{itemize}

                \item $\cA_\cusp(G_m(\dF)\backslash G_m(\dA_\dF))$ the subspace of $\cA(G_m(\dF)\backslash G_m(\dA_\dF))$ of cusp forms, and by $\langle\;,\;\rangle_{G_m}$ the hermitian form on $\cA_\cusp(G_m(\dF)\backslash G_m(\dA_\dF))$ given by the Petersson inner product with respect to the Haar measure $\r{d}g$.
      \end{itemize}
            For a subspace $\cV$ of $\cA(G_m(\dF)\backslash G_m(\dA_\dF))$, we denote by
      \begin{itemize}
                \item $\cV^{[r]}$ the intersection of $\cV$ and $\cA^{[r]}(G_m(\dF)\backslash G_m(\dA_\dF))$,

        \item $\cV^{[r]\tR}$ the intersection of $\cV$ and $\cA^{[r]\tR}(G_m(\dF)\backslash G_m(\dA_\dF))$,

        \item $\cV^\tc$ the subspace $\{\varphi^\tc\res\varphi\in\cV\}$.
      \end{itemize}
\end{enumerate}
\end{notation}

\begin{notation}\label{st:w}
We review the Weil representation.
\begin{enumerate}[label=(W\arabic*)]
    \item For every $v\in\tV_\dF$, we have the Weil representation $\omega_{m,v}$ of $G_m(\dF_v)\times H(\dF_v)$, with respect to the additive character $\psi_{\dF,v}$ and the trivial splitting character, realized on the Schr\"{o}dinger model $\sS(V_v^m)$. For the readers' convenience, we review the formulas:
      \begin{itemize}
                \item for $a\in\GL_m(\dE_v)$ and $\phi\in \sS(V_v^m)$, we have
          \[
                    \omega_{m,v}(m(a))\phi(x)=|\det a|_{\dE_v}^r\cdot \phi(x a);
          \]

                \item for $b\in\Herm_m(\dF_v)$ and $\phi\in \sS(V_v^m)$, we have
          \[
          \omega_{m,v}(n(b))\phi(x)=\psi_{T(x)}(b)\cdot \phi(x)
          \]
                    where $\psi_{T(x)}(b)\coloneqq\psi_{\dF,v}(\tr(b\,T(x)))$;

        \item for $\phi\in \sS(V_v^m)$, we have
          \[
                    \omega_{m,v}(\tw_m)\phi(x)=\gamma_{V_v,\psi_{\dF,v}}^m\cdot\widehat\phi(x),
          \]
          where $\gamma_{V_v,\psi_{\dF,v}}$ is certain Weil constant determined by $V_v$ and $\psi_{\dF,v}$;

                \item for $h\in H(\dF_v)$ and $\phi\in \sS(V_v^m)$, we have
          \[
          \omega_{m,v}(h)\phi(x)=\phi(h^{-1}x).
          \]
      \end{itemize}
      We put $\omega_m\coloneqq\otimes_v\omega_{m,v}$ as the ad\`{e}lic version, realized on $\sS(V^m)$.

    \item For every $v$ of $\dF$, we also realize the contragredient representation $\omega_{m,v}^\vee$ on the space $\sS(V_v^m)$ as well via the bilinear pairing
      \[
      \langle\;,\;\rangle_{\omega_{m,v}}\colon \sS(V_v^m)\times \sS(V_v^m)\to\dC
      \]
      defined by the formula
      \[
      \langle\phi^\vee,\phi\rangle_{\omega_{m,v}}\coloneqq\int_{V_v^m}\phi(x)\phi^\vee(x)\rd x
      \]
      for $\phi,\phi^\vee\in\sS(V_v^m)$.
\end{enumerate}
\end{notation}

\begin{notation}\label{st:c}
For a locally Noetherian scheme $X$ and an integer $m\geq 0$, we denote by $\rZ^m(X)$ the free abelian group generated by irreducible closed subschemes of codimension $m$ and $\CH^m(X)$ the quotient by rational equivalence. Suppose that $X$ is smooth over a field $K$ of characteristic zero. Let $\ell$ be a rational prime.
\begin{enumerate}[label=(C\arabic*)]
  \item We denote by $\rZ^m(X)^0$ the kernel of the de Rham cycle class map
     \[
     \cl_{X,\dr}\colon\rZ^m(X)\to\rH^{2m}_{\dr}(X/K)(m),
     \]
     and by $\CH^m(X)^0$ the image of $\rZ^m(X)^0$ in $\CH^m(X)$.

  \item When $K$ is a non-archimedean local field, we denote by $\rZ^m(X)^{\langle\ell\rangle}$ the kernels of the $\ell$-adic cycle class map
      \[
      \cl_{X,\ell}\colon\rZ^m(X)\to\rH^{2m}(X,\dQ_\ell(m)).
      \]

  \item When $K$ is a number field, we define $\rZ^m(X)^{\langle\ell\rangle}$ via the following Cartesian diagram
      \[
      \xymatrix{
      \rZ^m(X)^{\langle\ell\rangle} \ar[r]\ar[d] &  \prod_{v}\rZ^m(X_{K_v})^{\langle\ell\rangle} \ar[d] \\
      \rZ^m(X) \ar[r] & \prod_{v}\rZ^m(X_{K_v})
      }
      \]
      where the product is taken over all non-archimedean places of $K$. We denote by $\CH^m(X)^{\langle\ell\rangle}$ the image of $\rZ^m(X)^{\langle\ell\rangle}$ in $\CH^m(X)$, which is contained in $\CH^m(X)^0$ by the comparison theorem between de Rham and $\ell$-adic cohomology.
\end{enumerate}
\end{notation}


\section{Spherical Representations and Almost Spherical Representations of Quasi-split Unitary groups}\label{sec:quasi-split}

When we talk about representations of a reductive group \(G\) defined over \(F\), we mean admissible representations of \(G(F)\) with coefficients in \(\dC\). We will not distinguish \(G(F)\) and \(G\) in the notation.
\subsection{Classical Results of Unramified Principal Series Representations}

Now we collect some classical results of Casselman on the spherical representations of quasi-split unitary groups. We will follow the exposition in \cite{casselman1980unramified} and \cite{liu2022theta}*{Section 4}. 
\begin{assumption}
    Assume that \(E/F\) is a quadratic ramified extension of \(p\)-adic fields. 
\end{assumption}

Let \((W_r,\langle\cdot,\cdot\rangle_r,G_r,T,S,P_r)\) be as in \ref{notion:local_skew_hermitian}, \ref{notion:torus} and \ref{notion:borel}. 

For an element \(\sigma=(\sigma_1,\cdots,\sigma_r)\in\left(\dC/(\frac{2\pi}{\log q}\dZ)\right)^r\), we define a character 
\[\chi_r^\sigma:T\rightarrow\dC^\times\quad t\mapsto \prod_{i=1}^r|a_i|_E^{\sigma_i}\]
in which \(a_i\) is the eigenvalue of \(t\) acts on \(e_i\) for \(1\le i\le r\). Such characters are called unramified characters of \(T\) and obviously every unramified character of \(T\) is uniquely written as \(\chi_r^\sigma\).

This gives a character of the parabolic subgroup \(P_r\) by \(T\cong P_r/N_r\) where \(N_r\) is the unipotent subgroup.

Set the normalized unramified principal series representation of \(G_r(F)\) with parameter \(\sigma\) by
\[\rI_{W_r}^\sigma:=\Ind_{P_r}^{G_r}(\delta_{P_r}^{\frac{1}{2}}\chi_r^\sigma)=\{f\in C^\infty(G_r(F))|f(pg)=\delta_{P_r}^{\frac{1}{2}}(p)\chi_r^\sigma(p)f(g)\}\]
which is a representation of \(G_r(F)\) via the right translation.

Throughout this article, our objects are representations of \(G_r(F)\) with Iwahori-fixed vectors. The following is a famous theorem of Borel (in our notations, and this is generally true for classical groups), which illustrates the importance of unramified principal series representations in the study of such representations:
\begin{theorem}[\protect{\cite{borel1976admissible}}]\label{thm:borel_iwahori}
    If \(\pi\) is an irreducible representation of \(G_r\) and \(\pi^{I_r}\neq 0\), then there exists an unramified character \(\chi^\sigma_r\) such that \(\pi\) is a subrepresentation of \(\rI^\sigma_{W_r}\).
\end{theorem}

We should specify a special compact subgroup of \(G_r\) to study spherical representations. We choose \(K_r\) as in \ref{notion:stabilizer_standard}, which is a special maximal parahoric subgroup of \(G_r\) \cite{tits1979reductive}. Then we have the following notations from \cite{casselman1980unramified}
\begin{notation}
    For \(\alpha\in\Phi^+\) (the set of positive roots), we have \(q_\alpha=q\) and \(q_{\alpha/2}=1\).
\end{notation}

Now we choose representatives for elements in \(\fW_r\) so that \(\fW_r\) will be considered as a subgroup of \(G_r\), this is from \cite{liu2022theta}*{Section 2}.
\begin{notation}\label{notion:representative_standard}
    We choose representative \(\omega(w_1)\) of \(w_1\) as
    \[
        \omega(w_1):=\begin{pmatrix}
            0&0&1&0\\
            0&1_{r-1}&0&0\\
            -1&0&0&0\\
            0&0&0&1_{r-1}
        \end{pmatrix}
    \]
    and \(\omega(w_{i,i+1})\) is the element in \(K_r\) that permutes \(\{e_1,\cdots,e_r\}\) by swap \(e_i,e_{i+1}\), then the definition is extended to all elements in \(\fW_r\).
\end{notation}

We have a good basis of \((\rI_{W_r}^\sigma)^{I_r}\) by \cite{cartier1979representations}*{(28)} or \cite{casselman1980unramified}*{Proposition 2.1}
\begin{proposition}\label{prop:casselman_basis}
    Let \(\phi_{w,\sigma}\) be the unique element of \(\rI_{W_r}^\sigma\) such that for \(p\in P_r,w'\in\fW_r\) and \(i\in I_r\),
    \[
        \phi_{w,\sigma}(p\omega(w')i)=
        \begin{cases}
        \delta_{P_r}^{\frac{1}{2}}\chi_r^\sigma(p)&w'=w\\
        0&w'\neq w   
        \end{cases}
    \]
    then the functions \(\left\{\phi_{w,\sigma}\right\}_{w\in\fW_r}\) form a basis of \(\rI_{W_r}^\sigma\).
\end{proposition}
\begin{remark}
    Our \(\phi_{w,\sigma}\) is \(\phi_{w,\chi}\) in Casselman's notation where \(\chi=\chi_r^\sigma\).

    We will sometimes omit the index \(\sigma\).
\end{remark}

Assuming for a while that \(\sigma\) is regular, i.e. for \(w\in\fW_r\), \((w.\sigma=\sigma)\Longrightarrow w=1\)
\begin{definition}[Intertwining Operator \protect{\cite{casselman1980unramified}*{Section 3, (1)}}]
    For each \(x\in K_r\) representing \(w\in \fW_r\), we denote the intertwining operator in loc. cit. also by \(T_w\).
\end{definition}
\begin{remark}
    See \cite{brubaker2024colored}*{Above Proposition 3.7} for a normalized version of the intertwining operator.
\end{remark}
\begin{definition}\label{def:casselman_c}
    For each \(1\le i\le r\), we define 
    \[
        c_{2\epsilon_i}(\sigma):=c_{2\epsilon_i}(\chi_r^\sigma)=\frac{1-q^{-1}q^{-2\sigma_i}}{1-q^{-2\sigma_i}}
    \]
    and for \(1\le i<j\le r\), we define
    \[
        c_{\epsilon_j-\epsilon_i}(\sigma):=c_{\epsilon_j-\epsilon_i}(\chi_r^\sigma)=\frac{1-q^{-1}q^{\sigma_i-\sigma_j}}{1-q^{\sigma_i-\sigma_j}}.
    \]

\end{definition}
\begin{remark}
    In Casselman's notation \cite{casselman1980unramified}*{Section 3}, we have \(q_{\alpha/2}=1,q_{\alpha}=q\),\(\chi(a_{\epsilon_i+\epsilon_j})=q^{-\sigma_i-\sigma_j}\) for \(1\le i,j\le r\) and \(\chi_{\epsilon_{j}-\epsilon_i}(a_{\epsilon_j-\epsilon_i})=q^{\sigma_i-\sigma_j}\) for \(i<j\). This definition easily extends to all roots, not only positive roots. If one wants \(|q^{-1+\sigma_i-\sigma_j}|<1\) for all positive roots \(\alpha\), then we have \(0<\RE\sigma_1<\RE\sigma_2<\cdots<\RE\sigma_r\), or equivalently, \(|q^{-\sigma_r}|<|q^{-\sigma_2}|<\cdots<|q^{-\sigma_1}|<1\).
    
    One can also refer to \cite{li1992nonvanishing}*{Section 3, around Page 188} and \cite{li2022chow}*{Section 3}.

    Use the notation from \cite{kato1981irreducibility}*{1.13}, we know that there are polynomials \(e_\alpha(\sigma),d_\alpha(\sigma)\) in \(\dC[q^{\sigma_i},q^{-\sigma_i}]\) such that \(c_\alpha(\sigma)=\frac{e_\alpha(\sigma)}{d_\alpha(\sigma)}\).
\end{remark}
\begin{remark}
    Also note that our definitions of \(T_w\) and \(c_\alpha(\sigma)\) are not correct for \emph{irregular} \(\sigma\), but most results will be extended to irregular ones by holomorphicity.
\end{remark}
\begin{definition}
    For \(w\in\fW_r\), we define
    \[
        c_w(\sigma):=c_w(\chi_r^\sigma)=\prod_{\alpha>0,w\alpha<0} c_\alpha(\sigma)
    \]
\end{definition}
The action of \(T_w\) on \((\rI^\sigma_{W_r})^{I_r}\) can be described explicitly by the following theorem of Casselman:
\begin{theorem}\label{thm:casselman_relation}\cite{casselman1980unramified}*{Theorem 3.4}
    If \(\alpha\in\Delta\) and \(\ell(w_\alpha w)>\ell(w)\), then
    \begin{equation}
        \begin{aligned}
            T_{w_\alpha}(\phi_{w,\sigma})=&(c_\alpha(\sigma)-1)\phi_{w,w_\alpha\sigma}+q^{-1}\phi_{w_\alpha w,w_\alpha\sigma}\\
            T_{w_\alpha}(\phi_{w_\alpha w,\sigma})=&\phi_{w,w_\alpha\sigma}+(c_\alpha(\sigma)-q^{-1})\phi_{w_\alpha w,w_\alpha\sigma}
        \end{aligned}
    \end{equation}
\end{theorem}

\subsection{Spherical Representations with Respect to the Standard Lattice}
In this subsection we summarize the basic properties of irreducible admissible representations of \(G_r\) with \(K_r\)-fixed vector. Some results are already known since 1970s.

Because the Hecke algebra \(\cH(G_r//K_r)\) is commutative (where we endow \(K_r\) with Haar measure 1), we know that, if \((\pi,\cV)\) is an irreducible representation of \(G_r\) such that \(\cV^{K_r}\neq 0\), then \(\dim_\dC \cV^{K_r}=1\). Furthermore, if we consider \(\rI_{W_r}^\sigma\) by Theorem \ref{thm:borel_iwahori}, we will have \((\rI_{W_r}^\sigma)^K_r\) is spanned by the function
\[
    \phi_{K_r,\sigma}(pk):=\delta_{P_r}^{\frac{1}{2}}(p)\chi^\sigma_r(p)\qquad p\in P_r,k\in K_r
\]
and 
\[
    \phi_{K_r}=\sum_{w\in\fW_r}\phi_w
\]
\begin{remark}
    Even if \(\sigma\) is not regular, \(\phi_w\)'s and \(\phi_{K_r}\) are well-defined.
\end{remark}
\begin{theorem}\label{thm:casselman_spherical_vector}\cite{casselman1980unramified}*{Theorem 3.1}
    For regular \(\sigma\),
    \[T_w(\phi_{K_r,\sigma})=c_w(\sigma)\phi_{K_r,w\sigma}\]
\end{theorem}
Now we define the so-called \textit{spherical representations with respect to the standard lattice} as in the title of this subsection.

\begin{definition}\label{def:spherical_representation_standard}
    A spherical representation with respect to the standard lattice is an irreducible representation \((\pi,\cV)\) of \(G_r\) such that
    \(\cV^{K_r}\neq 0\). 

    If \(G_r\) is identified with the unitary group of a Hermitian space, then such representations will also be called spherical representations with respect to \(K_r\).
    
    By the Iwasawa decomposition \ref{thm:iwasawa_decomposition} and \ref{thm:borel_iwahori}, the irreducible component with \(K_r\)-fixed vector of an unramified principal series representation \(\rI_{W_r}^\sigma\) is nonzero and will be denoted by \(\pi^{\sigma}_{W_r,+}\) as in \cites{liu2022theta,li2022chow}.
\end{definition}
\begin{remark}
    This is indeed the analogue of unramified representations for \(E/F\) unramified in our ramified setting.

    We also call such representations spherical representations with respect to \(K_r\).

\end{remark}
As an easy consequence of the Satake isomorphism, the isomorphism classes of spherical representations with respect to \(K_r\) are in one-to-one correspondence with the set of \(\sigma\) modulo the action of \(\fW_r\).

The action of the finite Hecke algebra can be described as follows:
\begin{definition}\label{def:character_spherical}
    If we endow \(I_r\) with Haar measure 1. Then we have a character \[\kappa_r:\dC[I_r\backslash K_r/I_r]\longrightarrow\dC\] such that
    \[
    1_{I_r\omega(w_1)I_r}\mapsto q,\qquad 1_{I_r\omega(w_{i,i+1})I_r}\mapsto q
    \]

    If \((\pi,\cV)\) is a spherical representation with respect to \(K_r\), then the action of the finite Hecke algebra on \(\cV^{K_r}\) is given by \(\kappa_r\).
\end{definition}
\begin{remark}
    This is the same as \cite{liu2022theta}*{Definition 2.1, where \(\epsilon=+\)} except now we have \([I_r\omega(w_{i,i+1})I_r:I_r]=q\) instead of \(q^2\).
\end{remark}

We prove a modification of the Langlands classification for the \(K_r\)-spherical representations:

\begin{theorem}\label{thm:Langlands_classification_spherical}
    Assume that \(|q^{-\sigma_r}|\le\cdots\le|q^{-\sigma_1}|\le 1\), then the \(K_r\)-spherical irreducible component \(\pi_K\) of \(\rI_{W_r}^\sigma\) is a quotient of \(\rI_{W_r}^\sigma\).
\end{theorem}
\begin{proof}
    We prove by induction on \(r\). If \(r=1\) this is trivial, we may assume that the statement is true for all groups with rank at most \(r-1\).

    By the Langlands classification, we know that \(\pi_K\) is the Langlands quotient of some \(\rI_P^{G_r}\tau\chi\) (see \cite{liu2022beilinson}*{Appendix C.1} for the notations, but they did not choose a set of positive roots so the conditions differed) where \(P=MN\) and it is possible that \(P=M=G_r\). If \(P=G_r\) then \(\pi_K\) is tempered and has to be a component of unitary principal series which is semisimple and forced to be a quotient. 
    
    Now we assume \(P\neq G_r\), if \(P=P_r\), then the result is trivial again. Next we prove that \(\tau\chi\) is a quotient of \(\rI_{P_r}^M\chi_r^\sigma\), this is because by Theorem \ref{thm:borel_iwahori}, it is a component of such principal series, and it must be \(\rI_{P_r}^M\chi_r^\sigma\) (up to conjugating by an element of the relative Weyl group) by comparing the action of \(T\), then applying the induction hypothesis to \(\tau\) and the Langlands classification of general linear groups to \(\chi\), we know that \(\tau\chi\) is a quotient of \(\rI_{P_r}^M\chi_r^\sigma\). Now we get
    \begin{equation}
        \rI_{W_r}^\sigma=\rI_{P_r}^{G_r}\chi_r^\sigma=\rI_{P}^{G_r}\rI_{P_r}^M\chi^\sigma_r\twoheadrightarrow\rI_{P}^{G_r}\tau\chi\twoheadrightarrow\pi_K.
    \end{equation}
    The conclusion follows.
\end{proof}
\begin{remark}
    Dually, we know that if \(1\le|q^{-\sigma_1}|\le\cdots\le|q^{-\sigma_r}|\), then the \(K_r\)-spherical irreducible component \(\pi_K\) of \(\rI_{W_r}^\sigma\) is a subrepresentation of \(\rI_{W_r}^\sigma\).
\end{remark}

\subsection{Almost Spherical Representations with Respect to the Standard Lattice}
This subsection is to establish some results in \cite{liu2022theta}*{Section 2, Section 4} in our circumstance.

\begin{definition}\label{def:character_almost_spherical}
    We define a character \(\kappa_r^-:\dC[I_r\backslash K_r/I_r]\rightarrow\dC\) similar to \cite{liu2022theta}*{Definition 2.1}, given by
    \[
    1_{I_r\omega(w_1)I_r}\mapsto -1,\qquad 1_{I_rw_{i,i+1}I_r}\mapsto q
    \]
    for \(1\le i\le r\).
\end{definition}

Below is an analogue of \cite{liu2022theta}*{Lemma 2.2}.
\begin{lemma}\label{lem:eigenspace_span}
    The eigenspace \(\dC[I_r\backslash K_r/I_r][\kappa_r^-]\) is spanned over \(\dC\) by the function
    \[
    \fe_r^-:=\sum_{i=0}^r(-q)^{-i}1_{I_r^0\omega(w_1\cdots w_i)I_r^0}.
    \]
\end{lemma}
\begin{proof}
    It is identical to that in \cite{liu2022theta}*{Lemma 2.2}.
\end{proof}
\begin{definition}\label{def:almost_spherical_representation}
    An irreducible admissible representation \((\pi,\cV)\) of \(G_r\) is called almost spherical with respect to the standard lattice, or almost spherical with respect to \(K_r\), if \(\cV^{I_r}[\kappa_r^-]\neq 0\) and the Satake parameter of \(\pi\) contains either \( q^{\pm\frac{1}{2}}\) or \(- q^{\pm\frac{1}{2}}\). We say it is regularly almost spherical when it does not contain \(-q^{\pm\frac{1}{2}}\).
\end{definition}
\begin{remark}
    The last condition is equivalent to say, if we realize \(\pi\) as an irreducible component of an unramified principal series representation \(\rI_{W_r}^\sigma\), then \(q^\sigma\) contains \(q^{\pm\frac{1}{2}}\) or \(-q^{\pm\frac{1}{2}}\). This irreducible component will be denoted by \(\pi_{W_r,-}^\sigma\) as in \cite{liu2022theta}*{Notion 5.5}.

    In most proofs in this article, the last condition is not necessary; it is only needed to reduce to the well-known spherical case.
\end{remark}
Similar to the spherical case, we want to know about \((\rI_{W_r}^\sigma)^{I_r}[\kappa_r^-]\). We have the following:
\begin{proposition}\label{prop:eigenspace_span_almost_spherical}
    The eigenspace \((\rI_{W_r}^\sigma)^{I_r}[\kappa_r^-]\) is one-dimensional and spanned by the function
    \[
        \phi_{K_r}^-:=\sum_{w\in\fW_r}(-q)^{-\ell_{long}(w)}\phi_w
    \]
\end{proposition}
\begin{proof}
    It suffices to prove that \(\phi_{K_r}^-\) lies in the eigenspace.

    We have
    \[
        \phi_{K_r}^-=\sum_{i=0}^r(-q)^{-i}\sum_{w\in\fW_r,\ell_{long}(w)=i}\phi_w
    \]
    then the proof follows from Lemma \ref{lem:eigenspace_span} and \cite{casselman1980unramified}*{Above Proposition 2.1}.
\end{proof}
We have the following easy identities:
\begin{lemma}\label{lem:casselman_relation_almost_spherical}
    Suppose that \(\sigma\) is regular. Assume that \(\ell(w_\alpha w)>\ell(w)\),
    then we have:
    \begin{equation}
        \begin{aligned}
            T_{w_\alpha}(\phi_{w,\sigma}-q^{-1}\phi_{w_\alpha w})=&(c_\alpha(\sigma)-q^{-1}-1)(\phi_{w}-q^{-1}\phi_{w_\alpha w})&\alpha=&2\epsilon_1\\
            T_{w_\alpha}(\phi_{w,\sigma}+\phi_{w_\alpha w})=&c_\alpha(\sigma)(\phi_{w}+\phi_{w_\alpha w})&\alpha=&\epsilon_{i+1}-\epsilon_i
        \end{aligned}
    \end{equation}
\end{lemma}
Our main theorem in this subsection is 
\begin{theorem}\label{thm:casselman_relation_almost_spherical}
    Assume that \(\sigma\) is regular, then
    \[T_w(\phi_{K_r}^-)=\prod_{\alpha>0,w\alpha<0} c_{\alpha}''(\sigma)\phi_{K_r}^-\]
    where
    \[
        c_\alpha''(\sigma)=\begin{cases}
            c_\alpha(\sigma)-q^{-1}-1=\frac{q^{-2\sigma_i+1}-1}{q(1-q^{-2\sigma_i})}&\alpha=2\epsilon_i,1\le i\le r\\
            c_\alpha(\sigma)&\alpha=\epsilon_j\pm\epsilon_i,1\le i<j\le r
        \end{cases}
    \]
\end{theorem}

Before we prove this theorem, we prove a group-theoretic lemma which should be known to experts and is indeed crucial in the proof of Theorem \ref{thm:casselman_relation_almost_spherical} and those similar argument in the following sections.
\begin{lemma}\label{lem:trick}
    Let \(w=w_Iw_\tau\in\fW_r\). Assume that \(\ell(w_1 w)>\ell(w)\), then \(1\notin I\). In particular, \(\ell_{long}(w_1 w)>\ell_{long}(w)\).
\end{lemma}
\begin{proof}
    The length function \(\ell\) has a combinatoric interpretation in \cite{bjorner2005combinatorics}*{Proposition 8.1.1}.
    Note that in the notation of the citation,
    \begin{equation}
        \begin{aligned}
            &|\{(i, j) \in[n] \times[n] : i \le j, v(-i) > v(j)\}|\\
            =&|\{(i, j) \in[n] \times[n] : i <j, 0 >v(i)+ v(j)\}|+|\{i \in[n] : 0 > v(i)\}|
        \end{aligned}
    \end{equation}
    The lemma follows by \cite{bjorner2005combinatorics}*{(8.2)}
\end{proof}
Now we turn to the proof of Theorem \ref{thm:casselman_relation_almost_spherical}.
\begin{proof}
    By using the induction argument in the end of the proof of \cite{casselman1980unramified}*{Theorem 3.1}, it suffices to prove that
    \[T_{w_\alpha}\phi_{K_r}^-=c_\alpha''(\sigma)\phi_{K_r}^-,\quad \alpha\in\Delta\]

    For \(\alpha=\epsilon_{i+1}-\epsilon_i\) this is easy as the left multiplication by \(w_{i,i+1}\) preserves \(\ell_{long}\).

    For \(\alpha=2\epsilon_1\), by Lemma \ref{lem:trick} we know that \(\ell(w_1w)>\ell(w)\Longrightarrow \ell_{long}(w_1w)=1+\ell_{long}(w)\). One can thus split the sum in the proof of Proposition \ref{prop:eigenspace_span_almost_spherical} into pairs of \((w_1 w_Iw_\tau,w_Iw_\tau)\) with \(1\notin I\) and \(\ell(w_1 w_Iw_\tau)>\ell(w_Iw_\tau)\). Hence we can apply Lemma \ref{lem:casselman_relation_almost_spherical} to get the result.
\end{proof}
\begin{remark}
    In fact, this theorem was used silently in the calculation of ``\(C_{\bw'_r}^-(s)\)'' in \cite{liu2022theta}*{(5.5)}.
\end{remark}
\subsection{Spherical Representations with Respect to the Semi-standard Lattice}
Recall that from \ref{unimodular_lattice_split}, \(L_r\) is the stabilizer of the semi-standard lattice
\[
\Lambda_r'=(\bigoplus_{i=1}^r\cO_Ee_i)\oplus(\bigoplus_{i=r+1}^{2r}\varpi_E\cO_Ee_i).
\]
It is \emph{not} a parahoric subgroup of \(G_r\) by \cite{pappas2008twisted}*{Page 133}. However, we will prove that it satisfies some good properties like Iwasawa decomposition.
\begin{definition}\label{def:spherical_representations_semi_standard}
    An irreducible representation \((\pi,\cV)\) of \(G_r\) is called spherical with respect to \(L_r\), or spherical with respect to the semi-standard lattice, if \(\cV^{L_r}\) is not zero.
\end{definition}

\begin{remark}
    We should warn that, if one consider the corresponding parabolic of \(I_r^0\) in \(\rO(2n,\dF_q)\), it becomes lower triangular! So if one consider the Iwahori \(I_r\), it also corresponds to a lower triangular Borel. To avoid this confusion, we will always use the basis \(\{e_1,\cdots,e_{2r}\}\) and the standard notations for the skew-Hermitian spaces for this section for consistency with previous subsections.
\end{remark}
\begin{proposition}\label{prop:iwahori_in_L}
    We have \(I_r^0\subset L_r\).
\end{proposition}
\begin{proof}
    It suffices to prove that for any element \(g\in I_r^0\), we have \(g\Lambda_r'\subset \Lambda_r'\). This is obvious by our choice of \(\Lambda_r'\).
\end{proof}
Now we prove that \(L_r\) satisfies the Iwasawa decomposition:
\begin{theorem}\label{thm:iwasawa_semi_standard}
    For each \(w\in \fW_r\), there exists a representative \(\Omega(w)\in L_r\) and
    \[
        G_r=\bigsqcup_{w\in\fW_r}P_r\Omega(w)I_r=P_rL_r
    \]
    and we have \(P_r\cap L_r=(T\cap L_r)(N_r\cap L_r)\).
\end{theorem}
\begin{proof}
    Consider \(\Omega(w_{i,i+1}):=\omega(w_{i,i+1})\) and
    \begin{equation}
    \begin{aligned}
    \Omega(w_{1}):=&
    \begin{pmatrix}
        0&0&\varpi_E^{-1}&0\\
        0&1_{r-1}&0&0\\
        \varpi_E&0&0&0\\
        0&0&0&1_{r-1}
    \end{pmatrix}\\
    =&
    \begin{pmatrix}
        \varpi_E^{-1}&0&0&0\\
        0&1_{r-1}&0&0\\
        0&0&-\varpi_E&0\\
        0&0&0&1_{r-1}
    \end{pmatrix}
    \begin{pmatrix}
        0&0&1&0\\
        0&1_{r-1}&0&0\\
        -1&0&0&0\\
        0&0&0&1_{r-1}
    \end{pmatrix}\\
    =&
    \begin{pmatrix}
        \varpi_E^{-1}&0&0&0\\
        0&1_{r-1}&0&0\\
        0&0&-\varpi_E&0\\
        0&0&0&1_{r-1}
    \end{pmatrix}
    \omega(w_1)\\
    =&
    \begin{pmatrix}
        0&0&1&0\\
        0&1_{r-1}&0&0\\
        -1&0&0&0\\
        0&0&0&1_{r-1}
    \end{pmatrix}
    \begin{pmatrix}
        -\varpi_E&0&0&0\\
        0&1_{r-1}&0&0\\
        0&0&\varpi_E^{-1}&0\\
        0&0&0&1_{r-1}
    \end{pmatrix}\\
    =&
    \omega(w_1)
    \begin{pmatrix}
        -\varpi_E&0&0&0\\
        0&1_{r-1}&0&0\\
        0&0&\varpi_E^{-1}&0\\
        0&0&0&1_{r-1}
    \end{pmatrix}
    \end{aligned}
    \end{equation}
    Because \(\Omega(w_1),\Omega(w_{i,i+1})\) satisfies the same generating relations (note that the conjugacy action of \(\Omega(w_{i,i+1})\) on \(\Omega(w_1)(\omega(w_1))^{-1}\) is trivial.) as \(\omega(w_1),\omega(w_{i,i+1})\), \(\Omega\) realize \(\fW_r\) as a subgroup of \(L_r\) by conjugacy action.

    It is easy to prove by induction that for any \(w\in\fW_r\)
    \[
        \Omega(w)(\omega(w))^{-1}\in T\subset P_r
    \]
    \[
        (\omega(w))^{-1}\Omega(w)\in T\subset P_r
    \]
    Then
    \[
        P_r\Omega(w)I_r=P_r\Omega(w)(\omega(w))^{-1}\omega(w)I_r=P_r\omega(w)I_r
    \]
    and it follows by the Iwasawa decomposition (Theorem \ref{thm:iwasawa_decomposition}) and Proposition \ref{prop:iwahori_in_L}.

It suffices to prove that \(P_r\cap L_r\subset(T\cap L_r)(N_r\cap L_r)\).

    Given
    \[
    g=\begin{pmatrix}
        a&b\\
        0&^ta^{\tc,-1}
    \end{pmatrix}\in L_r
    \]
    we know that \(a\in \Res_{E/F}\GL_r(\cO_F)\), thus we can take the diagonal entries of \(g\) to get \(t\in T\cap K_r=T\cap L_r\) such that \(t^{-1}g\in N_r\cap L_r\).
\end{proof}

Now it is natural to introduce some modified version of Proposition \ref{prop:iwahori_in_L}.

Note that if \(\sigma\) is fixed (it is allowed to be irregular)
\begin{equation}
    \begin{aligned}
        \phi_w(\Omega(w_I))=&\phi_{w}(\Omega(w_I)\omega(w_I)\omega(w_I))\\
        =&\chi_r^\sigma(\Omega(w_I)\omega(w_I))\\
        =&|\varpi_E^{-1}|^{\sum_{i\in I} (\sigma_i+i-\frac{1}{2})}_E\\
        =&q^{\sum_{i\in I} (\sigma_i+i-\frac{1}{2})}
    \end{aligned}
\end{equation}
here we used the fact that for any \(I\subset \{1,\cdots,r\},w_I^2=w_\emptyset=1\). This motivates the following definition
\begin{definition}\label{def:modified_basis}
    Given \(\sigma\), for each \(w\in\fW_r\), consider the function
    \[
        \phi'_{w,\sigma}:=(\phi_{w,\sigma}(\Omega(w)))^{-1}\phi_{w,\sigma}
    \]
    by Theorem \ref{thm:iwasawa_semi_standard} and the proof of Theorem \ref{thm:iwasawa_semi_standard}, they form a basis of \((\rI_{W_r}^\sigma)^{I_r}\).
\end{definition}
The following propositions are direct from the definition
\begin{proposition}\label{prop:basis_relation}
    \begin{enumerate}[(1)]
        \item For any \(w=w_Iw_\tau\in \fW_r,\sigma\),
        \[
            \phi_{w,\sigma}=\phi_{w,\sigma}(\Omega(w))\phi_{w',\sigma}=\chi^\sigma_r(\Omega(w)(\omega(w))^{-1})\phi_{w',\sigma}
        \]
        \item For \(w=w_Iw_\tau\)
        \begin{equation}
            \begin{aligned}
                \phi_{w,\sigma}(\Omega(w))=\phi_{1,\sigma}(\Omega(w)(\omega(w))^{-1})=&\phi_{1,\sigma}(\Omega(w_I)\Omega(w_\tau)(\omega(w_\tau))^{-1}(\omega(w_I))^{-1})\\
                =&\phi_{1,\sigma}(\Omega(w_I)(\omega(w_I))^{-1})\\
                =&q^{\sum_{i\in I} (\sigma_i+i-\frac{1}{2})}
            \end{aligned}
        \end{equation}
        so
        \[
            \phi'_{w,\sigma}=q^{-\sum_{i\in I} (\sigma_i+i-\frac{1}{2})}\phi_{w,\sigma},\qquad \phi_{w,\sigma}=q^{\sum_{i\in I} (\sigma_i+i-\frac{1}{2})}\phi'_{w,\sigma}
        \]
    \end{enumerate}
\end{proposition}
\begin{proposition}\label{def:one}
    Given \(\sigma=(\sigma_1,\cdots,\sigma_r)\), the function
    \begin{equation}
        \begin{aligned}
            \phi_{L_r}:=\phi_{L_r,\sigma}:=&\sum_{w\in\fW_r}\phi'_w=\sum_{w\in\fW_r}\phi_w(\Omega(w))^{-1}\phi_w\\
            =&\sum_{I\subset\{1,\cdots,r\}}q^{-\sum_{i\in I} (\sigma_i+i-\frac{1}{2})}\sum_{\tau\in \fS_r}\phi_{w_Iw_\tau}
        \end{aligned}
    \end{equation}
    is invariant under the action of \(L_r\) and spans \((\rI_{W_r}^\sigma)^{L_r}\).
\end{proposition}
\begin{proof}
    It is invariant under the action of \(L_r\) by Theorem \ref{thm:iwasawa_semi_standard}. \((\rI_{W_r}^\sigma)^{L_r}\) is also at most one dimensional as it is determined by its restriction on \(L_r\) by Theorem \ref{thm:iwasawa_semi_standard}.
\end{proof}
\begin{remark}\label{def:another_one}
    There is another function 
    \begin{equation}
        \begin{aligned}
            \phi^-_{L_r}:=\phi^-_{L_r,\sigma}:=&\sum_{w\in\fW_r}(-1)^{\ell_{long}(w)}\phi'_w=\sum_{w\in\fW_r}(-1)^{\ell_{long}(w)}\phi_w(\Omega(w))^{-1}\phi_w\\
            =&\sum_{I\subset\{1,\cdots,r\}}(-1)^{|I|}q^{-\sum_{i\in I} (\sigma_i+i-\frac{1}{2})}\sum_{\tau\in \fS_r}\phi_{w_Iw_\tau}
        \end{aligned}
    \end{equation}
\end{remark}
\begin{corollary}\label{cor:spherical_one_dim}
    Let \((\pi,\cV)\) be a spherical representation with respect to \(L_r\), then \(\cV^{L_r}\) is one dimensional.
\end{corollary}
\begin{proof}
    By Theorem \ref{thm:borel_iwahori} and Proposition \ref{prop:iwahori_in_L}, we know that \(\pi\) can be embedded into some \(\rI_{W_r}^\sigma\) and it hence follows by the proposition above.
\end{proof}
\begin{corollary}\label{cor:hecke_algebra_comm}
    Give \(L_r\) Haar measure 1, then the Hecke algebra with respect to \(L_r\), which is defined by \(\cH_{L_r}:=C_c^\infty(G_r//L_r)\), is commutative. 
\end{corollary}
\begin{proof}
    By \cite{gross1991applications}*{(4.1)}, it suffices to prove that any spherical representation \((\pi,\cV)\) with respect to \(L_r\) , \(\dim \cV^{L_r}=1\) which is the statement above.
\end{proof}

\begin{lemma}\label{lem:casselman_relation_semi_standard}
    Assume that \(\sigma\) is regular, \(w=w_Iw_\tau\), \(\alpha\in\Delta\) and \(\ell(w_\alpha w)>\ell(w)\).
    \begin{enumerate}[(1)]
        \item 
        For \(\alpha=2\epsilon_1\), we have
        \begin{equation}
            \begin{aligned}
                T_{w_1}\phi'_{w,\sigma}=&(c_\alpha(\sigma)-1)\phi'_{w,w_1\sigma}+q^{-\sigma_1-\frac{1}{2}}\phi_{w_1w,w_1\sigma}'\\
                T_{w_1}\phi'_{w_1w,\sigma}=&q^{-\sigma_1-\frac{1}{2}}\phi'_{w,w_1\sigma}+(c_\alpha(\sigma)-1)\phi_{w_1w,w_1\sigma}'
            \end{aligned}
        \end{equation}
        \item
        For \(\alpha=\epsilon_{i+1}-\epsilon_i\),
        \begin{enumerate}[(i)]
        \item
        If (\(i\in I\) and \((i+1)\in I\)) or ((\(i\notin I\)) and \((i+1)\notin I\)) then
        \begin{equation}
            \begin{aligned}
                T_{w_{1,2}}\phi'_{w,\sigma}=&(c_\alpha(\sigma)-1)\phi'_{w,w_{1,2}\sigma}+q^{-1}\phi_{w_{1,2}w,w_{1,2}\sigma}'\\
                T_{w_{1,2}}\phi'_{w_{1,2}w,\sigma}=&\phi'_{w,w_{1,2}\sigma}+(c_\alpha(\sigma)-q^{-1})\phi_{w_{1,2}w,w_{1,2}\sigma}'
            \end{aligned}
        \end{equation}
        \item
        if \(i\in I,(i+1)\notin I\) then
        \begin{equation}
            \begin{aligned}
                T_{w_{1,2}}\phi'_{w,\sigma}=&(c_\alpha(\sigma)-q^{-1})\phi'_{w,w_{1,2}\sigma}+\phi_{w_{1,2}w,w_{1,2}\sigma}'\\
                T_{w_{1,2}}\phi'_{w_{1,2}w,\sigma}=&q^{-1}\phi'_{w,w_{1,2}\sigma}+(c_\alpha(\sigma)-1)\phi_{w_{1,2}w,w_{1,2}\sigma}'
            \end{aligned}
        \end{equation}
        \item It is not the case that \(i\notin I, (i+1)\in I\) provided \(\ell(w_{i,i+1}w)>\ell(w)\). 
        \end{enumerate}
    \end{enumerate}
\end{lemma}
\begin{proof}
    \begin{enumerate}[(1)]
        \item By Theorem \ref{thm:casselman_relation}, we have
        \[
            T_{w_1}\phi_{w,\sigma}=(c_{2\epsilon_1}(\sigma)-1)\phi_{w,w_1\sigma}+q^{-1}\phi_{w_1w,w_1\sigma}
        \]
        and
        \[
            T_{w_1}\phi_{w_1w,\sigma}=\phi_{w,w_1\sigma}+(c_{2\epsilon_1}(\sigma)-q^{-1})\phi_{w_1w,w_1\sigma}
        \]
        proceed with Proposition \ref{prop:basis_relation} and the definition of \(\phi_w'\), we have
        \begin{equation}
            \begin{aligned}
                T_{w_1}\phi'_{w,\sigma}=&\phi_{w,\sigma}(\Omega(w))^{-1}T_{w_1}\phi_{w,\sigma}\\
                =&\phi_{w,\sigma}(\Omega(w))^{-1}((c_{2\epsilon_1}(\sigma)-1)\phi_{w,w_1\sigma}+q^{-1}\phi_{w_1w,w_1\sigma})\\
                =&\phi_{w,\sigma}(\Omega(w))^{-1}((c_{2\epsilon_1}(\sigma)-1)(\phi_{w,w_1\sigma}(\Omega(w)))\phi'_{w,w_1\sigma}\\
                +&q^{-1}(\phi_{w_1w,w_1\sigma}(\Omega(w_1w)))\phi'_{w_1w,w_1\sigma})\\
                =&(\phi_{w,\sigma}(\Omega(w_I))^{-1}\phi_{w,w_1\sigma}(\Omega(w_I)))(c_{2\epsilon_1}(\sigma)-1)\phi'_{w,w_1\sigma}\\
                +&(\phi_{w,\sigma}(\Omega(w_I))^{-1}\phi_{w_1w,w_1\sigma}(\Omega(w_1w_I)))q^{-1}\phi'_{w_1w,w_1\sigma}
            \end{aligned}
        \end{equation}
By Lemma \ref{lem:trick}, we know that \(1\notin I\), then we have
        \begin{equation}
            \begin{aligned}
                T_{w_1}\phi'_{w,\sigma}=&(c_{2\epsilon_1}(\sigma)-1)\phi'_{w,w_1\sigma}+(\phi_{w,\sigma}(\Omega(w_I))^{-1}\phi_{w_1w,w_1\sigma}(\Omega(w_1w_I)))q^{-1}\phi_{w_1w,w_1\sigma}'\\
                =&(c_{2\epsilon_1}(\sigma)-1)\phi'_{w,w_1\sigma}+q^{-1+(w_1\sigma)_1+\frac{1}{2}}\phi_{w_1w,w_1\sigma}'\\
                =&(c_\alpha(\sigma)-1)\phi'_{w,w_1\sigma}+q^{-\sigma_1-\frac{1}{2}}\phi_{w_1w,w_1\sigma}'
            \end{aligned}
        \end{equation}
        Similarly, we have
        \begin{equation}
            \begin{aligned}
                T_{w_1}\phi'_{w_1w,\sigma}=&\phi_{w_1w,\sigma}(\Omega(w_1w))^{-1}T_{w_1}\phi_{w_1w,\sigma}\\
                =&\phi_{w_1w,\sigma}(\Omega(w_1w))^{-1}(\phi_{w,w_1\sigma}+(c_{2\epsilon_1}(\sigma)-q^{-1})\phi_{w_1w,w_1\sigma})\\
                =&\phi_{w_1w,\sigma}(\Omega(w_1w))^{-1}(\phi_{w,w_1\sigma}(\Omega(w))\phi'_{w,w_1\sigma}\\
                +&\phi_{w_1w,w_1\sigma}(\Omega(w_1w))(c_{2\epsilon_1}(\sigma)-q^{-1})\phi'_{w_1w,w_1\sigma})\\
                =&\phi_{w_1w,\sigma}(\Omega(w_1w))^{-1}\phi_{w,w_1\sigma}(\Omega(w))\phi'_{w,w_1\sigma}\\
                +&\phi_{w_1w,\sigma}(\Omega(w_1w))^{-1}\phi_{w_1w,w_1\sigma}(\Omega(w_1w))(c_{2\epsilon_1}(\sigma)-q^{-1})\phi'_{w_1w,w_1\sigma}\\
                =&q^{-\frac{1}{2}-\sigma_1}\phi'_{w,w_1\sigma}+q^{-2\sigma_1}(c_{2\epsilon_1}(\sigma)-q^{-1})\phi_{w_1w,w_1\sigma}'\\
                =&q^{-\sigma_1-\frac{1}{2}}\phi'_{w,w_1\sigma}+(c_{2\epsilon_1}(\sigma)-1)\phi_{w_1w,w_1\sigma}'
            \end{aligned}
        \end{equation}
        \item 
        \begin{enumerate}[(i)]
            \item Note that the action of \(w_{i,i+1}\) on $w_I$ is trivial in this case. Combine this fact with Proposition \ref{prop:basis_relation} (2), we have
            \begin{equation}
                \begin{aligned}
                    &\begin{pmatrix}
                        \phi_{w,\sigma}(\Omega(w))^{-1}\phi_{w,w_{i,i+1}\sigma}(\Omega(w))\\\phi_{w,\sigma}(\Omega(w))^{-1}\phi_{w_{i,i+1}w,w_{i,i+1}\sigma}(\Omega(w_{i,i+1}w))\\
                        \phi_{w_{i,i+1}w,\sigma}(\Omega(w_{i,i+1}w))^{-1}\phi_{w,w_{i,i+1}\sigma}(\Omega(w))\\\phi_{w_{i,i+1}w,\sigma}(\Omega(w_{i,i+1}w))^{-1}\phi_{w_{i,i+1}w,w_{i,i+1}\sigma}(\Omega(w_{i,i+1}w))
                    \end{pmatrix}=
                    \begin{pmatrix}
                        1\\1\\
                        1\\1
                    \end{pmatrix}
                \end{aligned}
            \end{equation}
            The following is a direct calculation using the above equation and Theorem \ref{thm:casselman_relation}.
            \item Note that the action of \(w_{i,i+1}\) on \(w_I\) is non-trivial in this case, it replaces \(I\) by \(I\cup\{i+1\}-\{i\}\). Combine this fact with Proposition \ref{prop:basis_relation} (2), we have
            \begin{equation}
                \begin{aligned}
                    &\begin{pmatrix}
                        \phi_{w,\sigma}(\Omega(w))^{-1}\phi_{w,w_{i,i+1}\sigma}(\Omega(w))\\\phi_{w,\sigma}(\Omega(w))^{-1}\phi_{w_{i,i+1}w,w_{i,i+1}\sigma}(\Omega(w_{i,i+1}w))\\
                        \phi_{w_{i,i+1}w,\sigma}(\Omega(w_{i,i+1}w))^{-1}\phi_{w,w_{i,i+1}\sigma}(\Omega(w))\\\phi_{w_{i,i+1}w,\sigma}(\Omega(w_{i,i+1}w))^{-1}\phi_{w_{i,i+1}w,w_{i,i+1}\sigma}(\Omega(w_{i,i+1}w))
                    \end{pmatrix}=
                    \begin{pmatrix}
                        q^{\sigma_{i+1}-\sigma_i}\\q\\
                        q^{-1}\\q^{\sigma_i-\sigma_{i+1}}
                    \end{pmatrix}
                \end{aligned}
            \end{equation}
            \item This follows from the proof of Lemma \ref{lem:trick} and \cite{bjorner2005combinatorics}*{Proposition 8.1.1}.
        \end{enumerate}
    \end{enumerate}
\end{proof}
\begin{theorem}\label{thm:casselman_relation_semi_standard}
    Assume that \(\sigma\) is regular, then
    
    we have
    \[
        T_w(\phi_{L_r})=\prod_{\alpha>0,w\alpha<0}c_\alpha'(\sigma)\phi_{L_r}
    \]
    and
    \[
        T_w(\phi_{L_r}^-)=\prod_{\alpha>0,w\alpha<0}c_{\alpha}'''(\sigma)\phi_{L_r}^-
    \]
    where
    \begin{equation}
        \begin{aligned}
            c'_{2\epsilon_i}(\sigma):=&c_{2\epsilon_i}(\sigma)-1+\phi_{w_i}(\Omega(w_i))^{-1}\\
            =&c_{2\epsilon_i}(\sigma)-1+q^{-\sigma_i-\frac{1}{2}}\\
            =&\frac{q^{-\frac{1}{2}-\sigma_i}(1-q^{-\frac{1}{2}-\sigma_i})(1+q^{\frac{1}{2}-\sigma_i})}{1-q^{-2\sigma_i}},
        \end{aligned}
    \end{equation}
    \begin{equation}
        \begin{aligned}
            c'''_{2\epsilon_i}(\sigma):=&c_{2\epsilon_i}(\sigma)-1-\phi_{w_i}(\Omega(w_i))^{-1}\\
            =&c_{2\epsilon_i}(\sigma)-1-q^{-\sigma_i-\frac{1}{2}}\\
            =&\frac{q^{-2\sigma_i}(1-q^{-\frac{1}{2}+\sigma_i})(1+q^{-\frac{1}{2}-\sigma_i})}{1-q^{-2\sigma_i}},
        \end{aligned}
    \end{equation}
    and
    \[
        c'_{\alpha}(\sigma)=c'''_{\alpha}(\sigma)=c_{\alpha}(\sigma)
    \]
    for \(\alpha=\epsilon_{j}\pm\epsilon_{i},i<j\).
\end{theorem}
\begin{proof}
    Because \(\phi_{L_r}=\sum_{w\in\fW_r}\phi'_w\) and \(\phi^-_{L_r}=\sum_{w\in\fW_r}(-1)^{\ell_{long}(w)}\phi'_w\), the result follows from the argument in Theorem \ref{thm:casselman_relation_almost_spherical} and Lemma \ref{lem:casselman_relation_semi_standard}.
\end{proof}

Our next goal in this section is to prove that if \(\pi\) is spherical with respect to \(L_r\) and the Satake parameters do not contain \(- q^{\pm\frac{1}{2}}\), then \(\pi\) is spherical with respect to \(K_r\) as well. Before that, we need the following important result, which states that spherical representations with respect to \(L_r\) are not square-integrable.

\begin{theorem}\label{thm:spherical_L_is_spherical_K}
    Given \(\sigma\) and assume that the Satake parameters do not contain \(- q^{\pm\frac{1}{2}}\), let \(\pi_L\) be the unique component of \(\rI_{W_r}^\sigma\) which is spherical with respect to \(L_r\) and \(\pi_K\) be the unique component of \(\rI_{W_r}^\sigma\) which is spherical with respect to \(K_r\), then \(\pi_L=\pi_K\).
\end{theorem}
\begin{proof}
    Up to a conjugate by an element in the Weyl group, we may assume that \(\pi_L\) is an irreducible subrepresentation of \(\rI_{W_r}^\sigma\) (hence generated by \(\phi_{L_r}\)). It suffices to prove that \(1_{K_r}.\phi_{L_r}\neq 0\), this is obvious by Proposition \ref{prop:L-to-K-quasi-split} and our assumption on the Satake parameters.

\end{proof}


\subsection{Irreducibility of Unitary Principal Series Representations}\label{subsec:irreducibility_unitary_principal_series}
In this subsection, we will prove that for \(E/F\) ramified, any unitary principal series representation is irreducible.
We need the following theorem from \cite{kato1981irreducibility}*{Theorem 2.2}.
\begin{theorem}[Kato's Criterion on Irreducibility]\label{thm:kato_criterion}
    The principal series representation \(\rI^\sigma_{W_r}\) is irreducible iff
    \begin{enumerate}[(i)]
        \item \(e(\sigma)e(-\sigma)\neq0\).
        \item \(W_\sigma=W_{(\sigma)}\).
    \end{enumerate}
    where \(e(\sigma)=\prod_{\alpha>0}e_\alpha(\sigma)\), \(W_\sigma:=\left\{w|w.\sigma=\sigma\right\}\) and \(W_{(\sigma)}\) is the subgroup of \(W_\sigma\) generated by \(\{w_\alpha|d_\alpha(\sigma)=0\}\).
\end{theorem}
The following is our main theorem of this subsection:
\begin{theorem}\label{thm:irreducibility_unitary_principal_series}
    Let \(\sigma\) be unitary, then the principal series representation \(\rI^\sigma_{W_r}\) is irreducible.
\end{theorem}
\begin{proof}
    By the \(\sigma\) is unitary, we know that it suffices to prove that \(W_\sigma=W_{(\sigma)}\).

    Assume that \(\chi^\sigma_r\) is given by \(\sigma=(\sigma_1,\cdots,\sigma_s)\in (i\dR/\frac{i2\pi}{\log q})^s\).
    Then 
    \[
    d_\alpha(\sigma)=
    \begin{cases}
    1-q^{-\sigma_i-\sigma_j}&\alpha=\epsilon_i+\epsilon_j\\
    1-q^{-2\sigma_i}&\alpha=2\epsilon_i\\
    1-(q^{-\sigma_i})^2&\alpha=\epsilon_i
    \end{cases}
    \]
    The last case only occurs when the group is non-quasi-split, and \(q_{\frac{\alpha}{2}}\neq 1\) only occurs when \(\alpha=\epsilon_i\). In any case, we know that \(d_\alpha(\sigma)=0\) means \(\sigma_i+\sigma_j=0\in i\dR/\frac{i2\pi}{\log q}\), where \(i,j\) are the indices occurred in \(\alpha\) and \(i=j\) in the latter two cases.

    We divide the multiset \(\{\sigma_1,\cdots,\sigma_s\}\) into equivalence classes, where the equivalence relation is given by \(\sigma_i\sim\sigma_j\) if and only if \(\sigma_i=\pm\sigma_j\). Then if \(\sigma\) is invariant under some element of the Weyl group \(\{\pm1\}^s\rtimes \fS_s\), it must stabilize each class. 
    
    So we are allowed to assume that there exists only one equivalence class. If all elements of the class are equivalent to one of \(0\) or \(i\pi/\log q\) (i.e., \(q^{\sigma_i}=1\) for all \(i\) or \(q^{\sigma_i}=-1\) for all \(i\)), then \(W_\sigma\) is the whole Weyl group and \(d_\alpha(\sigma)\) is always zero, then \(W_{(\sigma)}\) contains all reflections and hence \(W_\sigma=W_{(\sigma)}\). 
    
    Now we assume that all elements are either \(t\) or \(-t\) and \(2t\neq 0\). If the number of \(t\) is \(x\), then it is obvious that \(\fS_{x}\times \fS_{s-x}\hookrightarrow W_{(\sigma)}\) as all reflections in the same class will make \(d_\alpha(\sigma)\) vanish. If an element \(w\in\{\pm1\}^s\rtimes \fS_s\) fixes \(\sigma\), its action can be factorized as: \(w\) sends some of \(t\)'s to \(-t\), it must send the same quantity of \(-t\)'s to \(t\), and then arrange all elements. This is equivalent to say that if we write \(w=w_p\cdot w_r\) where \(w_p\in \fS_s,w_r\in\{\pm1\}^2\), then \(w_r=w_2w_1\) where \(w_1,w_2\in \{\pm1\}^s\) and have the same \((-1)\)'s and \(w_p\in \fS_s\) send all \(t\)'s back to the original position of \(t\)'s and so for \(-t\)'s.
    
    Because permutations of \(t\)'s and permutations of \(-t\)'s are contained in \(W_{(\sigma)}\), we may assume that \(w\) is given by a perfect pairing of some \(t\) and \(-t\). 
    
    If \(\sigma_i=t\) and it is sent to \(\sigma_j=-t\) by \(w\), again as permutations are contained in \(W_{(\sigma)}\), we are allowed to assume that the \(\sigma_j\) is sent to \(\sigma_i\). Then \(w\) is a product of \(w_{i,j}w_iw_j\) where \((i,j)\) ranges over the pairs. It suffices to show that if \(\sigma_i+\sigma_j=0\) then \(w_{i,j}w_iw_j\in W_{(\sigma)}\), but \(w_{i,j}w_iw_j=w_{\epsilon_i+\epsilon_j}\), and \(d_{\epsilon_i+\epsilon_j}(\sigma)=0\), then \(w\in W_{(\sigma)}\). 
\end{proof}

\begin{remark}
    We should warn that the above theorem is \emph{not} true for unramified unitary groups, i.e. \(W_\sigma\neq W_{(\sigma)}\) may happen for some unitary \(\sigma\). For example, we consider the unramified unitary group \(\rU(2)\) defined over a \(p\)-adic field \(F\) with \(E/F\) unramified, the cardinality of residue field of \(F\) is \(q\). Let \(\sigma\in \frac{\dC}{i2\log q\dZ}\), we have \(c_{2\epsilon}(\sigma)=\frac{1-q^{-1-2\sigma}}{1-q^{-2\sigma}}\). When \(\sigma\) is unitary, the first condition in Theorem \ref{thm:kato_criterion} is always satisfied. Now if \(q^{-2\sigma}=-1\), then \(W_\sigma=\fW_1\), but \(W_{(\sigma)}=\{\id\}\) as \(d_{2\epsilon}(\sigma)=1-(-1)=2\neq0\) hence \(W_\sigma\neq W_{(\sigma)}\) and \(\rI_{W_1}^\sigma\) is reducible.
\end{remark}

\subsection{A Satake isomorphism}\label{subsec:Satake_isomorphism_quasi-split}
In this subsection, we will prove a Satake isomorphism for the Hecke algebra \(\cH(G_r//L_r),\) where \(E/F\) is ramified. The idea is that we can embed the pair \((G_r,L_r)\) into \((\GL_{2r}(E),\GL_{2r}(\cO_E))\) and then apply the classical theory of Bruhat--Tits theory. The notations in this subsection may differ from the previous ones for convenience, as the simple roots should be chosen to be compatible with the classical ones.

    \begin{notation}
        We use the following notations which are standard but slightly different from the previous ones:
        \begin{itemize}
            \item 
            Let \(V\) be a \(2r\) dimensional vector space over \(E\) with a basis \(\{e_1,\ldots,e_{2r}\}\). We equip \(V\) with a non-degenerate Hermitian form \((\cdot,\cdot)\) such that \((e_i,e_{2r-i})=1\) for all \(i\le r\) and \((e_i,e_j)=0\) for all other pairs.
            \item 
            Let \(G_r\) be the group of unitary similitudes of \(V\) and \(L_r\) be the stabilizer of the lattice spanned by \(\{e_1,\ldots,e_{2r}\}\). As stated in the last subsections, this is not a parahoric subgroup as if we consider the image modulo the uniformizer, it will be an even orthogonal group which is not connected.
            \item 
            Let \(T_r\) be the maximal diagonal torus of \(G_r\), we have \(T_r\cong (\Res_{E/F}\dG_m)^r(F)\). Let \(B_r\) be the Borel subgroup of upper triangular matrices containing \(T_r\). Equivalently, the simple roots are given by \(\alpha_i=\epsilon_i-\epsilon_{i+1}\) for \(i<r\) and \(\alpha_r=2\epsilon_r\) where \(\epsilon_i\) is the character of \(T_r\) sending a diagonal matrix to its \(i\)-th entry. The corresponding Iwahori subgroup is denoted by \(I_r\).
            \item We denote the finite Weyl group by \(\fW_r\). As used in the previous sections, we have a set \(\{\Omega(w)|w\in\fW_r\}\) representatives contained in \(L_r\).
        \end{itemize}
    \end{notation}
\begin{remark}
    The reason for this modification is that we can naturally embed \(G_r\) into \(\GL_{2r}(E)\) and \(L_r\) into \(\GL_{2r}(\cO_E)\) and the Borel subgroup \(B_r\) is also embedded into upper triangular matrices of \(\GL_{2r}(E)\).
\end{remark}

\begin{definition}
    Let \(\Lambda=T_r/(T_r\cap L_r)\cong\dZ^r\) which can be identified with the lattice of cocharacters. Let \(\Lambda^+\) be the set of dominant cocharacters with respect to \(B_r\). For any \(\lambda\in\Lambda\), we denote by \(\varpi^\lambda\) the standard representative of the image of \(\lambda\) in \(T_r\) evaluated at a uniformizer \(\varpi_E\) of \(E\).
    
    Explicitly, the identification is given by
\begin{center}
\begin{tikzcd}[column sep=large,row sep=small]
\dZ^r \arrow[r] & \Lambda \\
\lambda=(\lambda_1,\ldots,\lambda_r) \arrow[r, mapsto] & 
\begin{psmallmatrix}
\varpi_E^{\lambda_1} & & & & \\
& \ddots & & & \\
& & \varpi_E^{\lambda_r} & & \\
& & & \bar{\varpi}_E^{-\lambda_r} & \\
& & & & \ddots \\
& & & & & \bar{\varpi}_E^{-\lambda_1}
\end{psmallmatrix}
= \varpi^\lambda
\end{tikzcd}
\end{center}
    
    Moreover, \(\lambda\) is dominant if and only if \(\lambda_1\ge\lambda_2\ge\cdots\ge\lambda_r\ge0\).

\end{definition}

First, we prove the Cartan decomposition for the pair \((G_r,L_r)\).
\begin{theorem}\label{thm:cartan_decomposition_quasi_split}
    We have a disjoint decomposition
    \[
    G_r = \bigsqcup_{\lambda \in \Lambda^+} L_r \varpi^\lambda L_r.
    \]
\begin{proof}
    We first prove that every element of \(G_r\) can be written as an element in \(L_r\varpi^\lambda L_r\) for some \(\lambda\in\Lambda\). The proof of \cite{kaletha2023bruhat}*{Theorem 5.2.1 (1)} still applies as it only depends on the surjectivity of \((N_{G_r}(T_r)\cap L_r)/(T_r\cap L_r)\rightarrow\fW_r\) and the Bruhat decomposition. We get
    \[
    G_r=\bigcup_{\lambda\in\Lambda}L_r\varpi^\lambda L_r.
    \]

    Now we prove that if \(\lambda\) and \(\mu\) are two different dominant cocharacters, then \(L_r\varpi^\lambda L_r\cap L_r\varpi^\mu L_r=\varnothing\). It suffices to prove that 
    \[
        \GL_{2r}(\cO_E)\varpi^\lambda \GL_{2r}(\cO_E)\cap \GL_{2r}(\cO_E)\varpi^\mu \GL_{2r}(\cO_E)=\varnothing
    \]

    Note that for \(\lambda=(\lambda_1,\cdots,\lambda_r)\), \(\varpi^\lambda\) corresponds to the cocharacter \((\lambda_1,\cdots,\lambda_r,-\lambda_r,\cdots,-\lambda_1)\) of the diagonal torus of \(\GL_{2r}(E)\), which is still dominant. So the result follows from the classical Cartan decomposition for \((\GL_{2r}(E),\GL_{2r}(\cO_E))\).
\end{proof}
\end{theorem}

Now we define the Satake transform. We denote the unique maximal compact subgroup of \(T_r\) by \(T_r^o\).
\begin{definition}\label{def:satake_transform}
    The Satake transform is the algebra homomorphism
    \begin{center}
        \begin{tikzcd}
            \cS:&\cH(G_r,L_r)\arrow[r]&\cH(T_r,T_r^o)\\
            &f\arrow[r,mapsto]&\left(t\mapsto \delta^{\frac{1}{2}}_{B_r}(t)\int_{N_r}f(tn)\rd n\right)
        \end{tikzcd}
    \end{center}
    where \(N_r\) is the unipotent radical of \(B_r\) and \(\delta_{B_r}\) is the modular character of \(B_r\) and the Haar measure on \(N_r\) is normalized such that \(N_r\cap L_r\) has volume 1.
\end{definition}
\begin{remark}
    Note that the choice of the Haar measure on \(N_r\) is \emph{not} compatible with the classical one so that \(N_r\cap K_r\) where \(K_r\) is the special maximal parahoric subgroup of \(G_r\) has volume 1.
\end{remark}
Following the classical argument (see, for example, \cite{cartier1979representations}*{Theorem 4.1 (A) and (B)}), we have the following proposition:
\begin{proposition}\label{prop:satake_transform_properties}
    \begin{itemize}
        \item The Satake transform is an algebra homomorphism.
        \item The image of the Satake transform is contained in \(\cH(T_r,T_r^o)^{\fW_r}\).
    \end{itemize}
\end{proposition}
Next, we can prove the Satake isomorphism
\begin{theorem}[Satake isomorphism]\label{thm:satake_isomorphism_quasi_split}
    The Satake transform induces an isomorphism of algebras
    \[
    \cS:\cH(G_r,L_r)\xlongrightarrow{\sim}\cH(T_r,T_r^o)^{\fW_r}\cong\dC[\Lambda]^{\fW_r}\cong \dC[T_1^{\pm1},\cdots,T_r^{\pm1}]^{\fW_r}.
    \]
\begin{proof}
    It suffices to prove the analogue of \cite{cartier1979representations}*{Theorem 4.1 (C)}. The only non-trivial is to prove that \(L_r\varpi^{\lambda'}L_r\cap N_r\varpi^\lambda L_r\neq\varnothing\Longrightarrow \lambda' \leq \lambda\) in the dominance order.
     
    Assuming that \(L_r\varpi^{\lambda'}L_r\cap N_r\varpi^\lambda L_r\neq\varnothing\), then we have
    \[
    \GL_{2r}(\cO_E)\varpi^{\lambda'}\GL_{2r}(\cO_E)\cap N_{GL_{2r}}\varpi^\lambda \GL_{2r}(\cO_E)\neq\varnothing.
    \]
    where \(N_{GL_{2r}}\) is the unipotent radical of upper triangular matrices of \(\GL_{2r}(E)\). By the classical argument, we know that 
    \[(\lambda'_1,\cdots,\lambda'_r,-\lambda'_r,\cdots,-\lambda'_1)\le(\lambda_1,\cdots,\lambda_r,-\lambda_r,\cdots,-\lambda_1)\]
    in the dominance order of \(\GL_{2r}(E)\) which is equivalent to \(\lambda_i-\lambda_i' \ge 0\) for all \(i\le r\) and \((\lambda_{i}-\lambda'_i)-(\lambda_{i+1}-\lambda'_{i+1})\geq 0\) for \(i<r\). This is exactly the dominance order of \(G_r\) so \(\lambda'\le\lambda\).
\end{proof}
\end{theorem}

\section{Spherical Representations of Non-quasi-split Unitary groups}\label{sec:non-quasi-split}
Most of the results in this section are parallel to the results in the previous section. We will only state the results and give the references. The proofs are similar to the proofs in the previous section. The main difference is that the results in this section are for the non-quasi-split unitary groups. 

We should warn that, the essential use of this section is to prove some representation theoretic results for the non-quasi-split unitary groups. We will not use them in the computation of doubling method as the doubled space is always split.
\subsection{Classical Results of Unramified Principal Series Representations}
The results in this subsection are parallel to the results in Subsection 3.1. We will only state the results and give the references. The proofs are similar to the proofs in Subsection 3.1. The main difference is that the results in this subsection are for the non-quasi-split unitary groups associated to a non-split Hermitian space.

We recall the following notation:
\begin{itemize}
    \item 
    Let \(V=V_r^-\) be the standard non-split Hermitian space of dimension \(2r+2\) as in \ref{notion:standard_non_split_hermitian}. The Hermitian form is given by the matrix
    \[
        \begin{pmatrix}
            0&\varpi_E^{-1}1_r&0&0\\
            -\varpi_E^{-1}1_r&0&0&0\\
            0&0&1&0\\
            0&0&0&-\ts
        \end{pmatrix}.
    \]
    \item Let \(H_r\) be its associated unitary group which is non-quasi-split. 
    \item Let \(P^-_r\) be the minimal parabolic subgroup of \(H_r\) as in \ref{notion:minimal_parabolic}, \(M^-_r\cong(\Res_{E/F}\GL_{1})^r\times H_0\) be the Levi subgroup of \(P^-_r\), \(T\) be the standard diagonal torus of \(M^-_r\) and \(N^-_r\) be the unipotent radical of \(P^-_r\).
    \item Let \(P_r^0\) be the Siegel parabolic subgroup of \(H_r\) as in \ref{notion:siegel_parabolic}, \(M_r^0\cong(\Res_{E/F}\GL_{r})\times H_0\) be the Levi subgroup of \(P_r^0\) and \(N_r^0\) be the unipotent radical of \(P_r^0\).
    \item 
    Let \(\Delta\) be as in \ref{notion:roots_non_quasi_split}, but here we denote \(\Phi,\Phi^+\) the reduced and reduced positive roots. 
    \item 
    The lattice \(\Lambda_{V_r^-}:=\bigoplus_{i=1}^{2r+2}\cO_Ev_i\) is called the standard almost \(\pi\)-modular lattice in \(V_r^-\) \ref{almost_self_dual_lattice_non_split}\footnote{It is also called almost \(\varpi\)-modular.}. Its stabilizer in \(H_r\) is denoted by \(K_r^-\).
    \item 
    The lattice \(\Lambda'_{V_r^-}:=(\bigoplus_{i=1}^{r}\varpi_E\cO_Ev_i)\oplus(\bigoplus_{i=r+1}^{2r}\cO_Ev_i)\oplus(\cO_Ev_{2r+1}\oplus\cO_Ev_{2r+2})\) is called the standard unimodular lattice in \(V_r^-\)\ref{unimodular_lattice_non_split}\footnote{In some reference, it may be called self-dual.}. Its stabilizer in \(H_r\) is denoted by \(L_r^-\), we should emphasize that \(L_r^-\) has a subgroup \(L_r^{-\circ}\) of index 2 which is a special maximal compact subgroup of \(H_r\) \cite{tits1979reductive}.
    \item 
    The Iwahori subgroup corresponding to \(P_r^-\) is denoted by \(I_r^-\).
\end{itemize}
\begin{remark}
    The reason that we denote the Iwahori subgroup by \(I_r^-\) instead of \(I_r\) is to distinguish it from the Iwahori subgroup of the quasi-split unitary group defined before.
\end{remark}

For an element \(\sigma=(\sigma_1,\cdots,\sigma_r)\in\left(\dC/(\frac{2\pi}{\log q}\dZ)\right)^r\), we define a character 
\[\chi_r^\sigma:T\rightarrow\dC^\times\quad t\mapsto \prod_{i=1}^r|b_i|_E^{\sigma_i}\]
in which \(b_i\) is the eigenvalue of \(t\) acts on \(v_i\) for \(1\le i\le r\). Such characters are called unramified characters of \(T\) and obviously every unramified character of \(T\) is uniquely written as \(\chi_r^\sigma\).

This gives a character of the parabolic subgroup \(P_r^-\) by \(T\times H_0\cong P^-_r/N^-_r\).

Set the normalized unramified principal series representation of \(H_r(F)\) with parameter \(\sigma\) by
\[\rI_{V}^\sigma:=\Ind_{P^-_r}^{H_r}(\delta_{P^-_r}^{\frac{1}{2}}\chi_r^\sigma)=\{f\in C^\infty(H_r(F))|f(pg)=\delta_{P^-_r}^{\frac{1}{2}}(p)\chi_r^\sigma(p)f(g)\}\]
which is a representation of \(H_r(F)\) via the right translation.

\begin{remark}
    A very important difference here is that the restriction of \(\delta^{\frac{1}{2}}_{P_r^-}\) on the diagonal torus is sending \(t\mapsto |b_1|_E^{\frac{3}{2}}\cdots|b_r|_E^{r+\frac{1}{2}}\), which is different from the quasi-split case.
\end{remark}

Throughout this article, our objects are representations of \(H_r(F)\) with Iwahori-fixed vectors. The following is a famous theorem of Borel (in our notations, and this is generally true for classical groups), which illustrates the importance of unramified principal series representations in the study of such representations:
\begin{theorem}\cite{borel1976admissible}\label{thm:borel_iwahori_non_split}
    If \(\pi\) is an irreducible representation of \(H_r\) and \(\pi^{I_r^-}\neq 0\), then there exists an unramified character \(\chi^\sigma_r\) such that \(\pi\) is a subrepresentation of \(\rI^\sigma_V\).
\end{theorem}
\begin{notation}
    As in \ref{notion:roots_quasi_split}, we should specify a special maximal compact subgroup of \(H_r\) to study spherical representations. We choose \(L_r^{-\circ}\), which is a special maximal parahoric subgroup of \(H_r\) \cite{tits1979reductive}. 
    \begin{itemize}
        \item for \(\alpha=\epsilon_i-\epsilon_j\) or \(\alpha=\epsilon_i+\epsilon_j\) with \(i\neq j\), \(q_\alpha=q\) and \(q_{\alpha/2}=1\);
        \item for a multipliable short root \(\alpha=\epsilon_i\), \(q_\alpha=q\) and \(q_{\alpha/2}=q\).
    \end{itemize}
\end{notation}

Now we choose representatives for elements in \(\fW_r\) so that \(\fW_r\) will be considered as a subgroup of \(K^-_r\), this is from \cite{liu2022theta}*{Section 2}.
\begin{notation}\label{notion:representative_almost_self_dual_nonsplit}
    We choose representative \(\omega(w_1)\) of \(w_1\) as
    \[
        \omega(w_1):=\begin{pmatrix}
            0&0&1&0&&\\
            0&1_{r-1}&0&0&&\\
            -1&0&0&0&&\\
            0&0&0&1_{r-1}&&\\
            &&&&1&\\
            &&&&&1
        \end{pmatrix}
    \]
    and \(\omega(w_{i,i+1})\) is the element in \(K^-_r\) that permutes \(\{v_1,\cdots,v_r\}\) by swap \(v_i,v_{i+1}\) and stabilizes \(v_{2r+1}\) and \(v_{2r+2}\), then the definition is extended to all elements in \(\fW_r\).

    Then we choose \(\Omega(w)\) for each \(w\in \fW_r\) similarly as last section, but notice that, in this case we can choose \(\Omega\) all inside \(L_r^{-\circ}\), i.e. the determinant is 1 modulo \(\varpi_E\):

    We choose 
    \[
    \Omega(w_1):=\begin{pmatrix}
        0&0&\varpi_E^{-1}&0&&\\
        0&1_{r-1}&0&0&&\\
        \varpi_E&0&0&0&&\\
        0&0&0&1_{r-1}&&\\
        &&&&-1&\\
        &&&&&1
    \end{pmatrix}
    \]
    and consider the elements \(\Omega(w_{i,i+1}):=\omega(w_{i,i+1})t\)
    where \(t=\diag\{1_{2r-2},-1,1\}\). Then we extend the definition to all elements in \(\fW_r\) as before thus all \(\Omega(w)\in L_r^{-\circ}\).

\end{notation}

We have a good basis of \((\rI_{V}^\sigma)^{I_r^-}\) by \cite{cartier1979representations}*{(28)} or \cite{casselman1980unramified}*{Proposition 2.1}
\begin{proposition}\label{prop:casselman_basis_non_split}
    Let \(\phi_{w,\sigma}\) be the unique element of \(\rI_{V}^\sigma\) such that for \(p\in P_r^-,w'\in\fW_r\) and \(i\in I^-_r\),
    \[
        \phi_{w,\sigma}(p\Omega(w')i)=
        \begin{cases}
        \delta_{P_r^-}^{\frac{1}{2}}\chi_r^\sigma(p)&w'=w\\
        0&w'\neq w   
        \end{cases}
    \]
    then the functions \(\left\{\phi_{w,\sigma}\right\}_{w\in\fW_r}\) form a basis of \(\rI_{V}^\sigma\).
\end{proposition}
\begin{remark}
    Our \(\phi_{w,\sigma}\) is \(\phi_{w,\chi}\) in Casselman's notation where \(\chi=\chi_r^\sigma\).

    We will sometimes omit the index \(\sigma\).
\end{remark}

Assuming for a while that \(\sigma\) is regular, i.e. for \(w\in\fW_r\), \((w.\sigma=\sigma)\Longrightarrow w=1\)
\begin{definition}[Intertwining Operator \cite{casselman1980unramified}*{Section 3, (1)}]
    For each \(x\in L_r^{-\circ}\) representing \(w\in \fW_r\), we denote the intertwining operator in loc. cit. also by \(T_w\).
\end{definition}
\begin{definition}
    For each \(1\le i\le r\), we define 
    \[
        c_{\epsilon_i}(\sigma):=c_{\epsilon_i}(\chi_r^\sigma)=\frac{(1-q^{-\frac{3}{2}-\sigma_i})(1+q^{-\frac{1}{2}-\sigma_i})}{1-q^{-2\sigma_i}}
    \]
    Note that this is different with that in \ref{def:casselman_c}.

    and for \(1\le i<j\le r\), we define
    \[
        c_{\epsilon_j-\epsilon_i}(\sigma):=c_{\epsilon_j-\epsilon_i}(\chi_r^\sigma)=\frac{1-q^{-1}q^{\sigma_i-\sigma_j}}{1-q^{\sigma_i-\sigma_j}}.
    \]

\end{definition}
\begin{remark}
    Compared to the quasi-split case (Definition \ref{def:casselman_c}), the new feature is the presence of multipliable roots \(\alpha=\epsilon_i\) and hence \(q_{\alpha/2}>1\). For our special choice \(L_r^{-\circ}\), we have \(q_\alpha=q\) and \(q_{\alpha/2}=q\), so \(q_\alpha^{-1}q_{\alpha/2}^{-1}=q^{-2}\) appears in Casselman's relations.

    Use the notation from \cite{kato1981irreducibility}*{1.13}, we know that there are polynomials \(e_\alpha(\sigma),d_\alpha(\sigma)\) in \(\dC[q^{\sigma_i},q^{-\sigma_i}]\) such that \(c_\alpha(\sigma)=\frac{e_\alpha(\sigma)}{d_\alpha(\sigma)}\).
\end{remark}
\begin{remark}
    Also note that our definitions of \(T_w\) and \(c_\alpha(\sigma)\) are not correct for \emph{irregular} \(\sigma\), but most results will be extended to irregular ones by holomorphicity.
\end{remark}
\begin{definition}
    For \(w\in\fW_r\), we define
    \[
        c_w(\sigma):=c_w(\chi_r^\sigma)=\prod_{\alpha\in\Phi^+,w\alpha<0} c_\alpha(\sigma)
    \]
\end{definition}
The action of \(T_w\) on \((\rI^\sigma_V)^{I_r^-}\) can be described explicitly by the following theorem of Casselman:
\begin{theorem}\label{thm:casselman_relation_non_split}\cite{casselman1980unramified}*{Theorem 3.4}
    If \(\alpha\in\Delta\) and \(\ell(w_\alpha w)>\ell(w)\), then
    \begin{equation}
        \begin{aligned}
            T_{w_\alpha}(\phi_{w,\sigma})=&(c_\alpha(\sigma)-1)\phi_{w,w_\alpha\sigma}+q_\alpha^{-1}q_{\alpha/2}^{-1}\phi_{w_\alpha w,w_\alpha\sigma}\\
            T_{w_\alpha}(\phi_{w_\alpha w,\sigma})=&\phi_{w,w_\alpha\sigma}+(c_\alpha(\sigma)-q_\alpha^{-1}q_{\alpha/2}^{-1})\phi_{w_\alpha w,w_\alpha\sigma}
        \end{aligned}
    \end{equation}
\end{theorem}
\begin{remark}
    We restate the theorem here because for \(\alpha=\epsilon_i\), we have \(q_\alpha=q_{\alpha/2}=q\), hence \(q_\alpha^{-1}q_{\alpha/2}^{-1}=q^{-2}\).
\end{remark}

\subsection{Spherical Representations with Respect to the Unimodular Lattice}
In this subsection we summarize the basic properties of irreducible admissible representations of \(H_r\) with \(L_r^{-\circ}\)-fixed vector (the special maximal compact subgroup chosen above).

Because the Hecke algebra \(\cH(H_r//L_r^{-\circ})\) is commutative (where we endow \(L_r^{-\circ}\) with Haar measure 1), we know that, if \((\pi,\cV)\) is an irreducible representation of \(H_r\) such that \(\cV^{L_r^{-\circ}}\neq 0\), then \(\dim_\dC \cV^{L_r^{-\circ}}=1\). Furthermore, if we consider \(\rI_{V}^\sigma\) by Theorem \ref{thm:borel_iwahori_non_split}, we will have \((\rI_{V}^\sigma)^{L_r^{-\circ}}\) is spanned by the function
\[
    \phi_{L_r^{-\circ},\sigma}(p\ell):=\delta_{P_r^-}^{\frac{1}{2}}(p)\chi^\sigma_r(p)\qquad p\in P_r^-,\ell\in L_r^{-\circ}
\]
and
\[
    \phi_{L_r^{-\circ}}=\sum_{w\in\fW_r}\phi_w.
\]
\begin{remark}
    Even if \(\sigma\) is not regular, \(\phi_w\)'s and \(\phi_{L_r^{-\circ}}\) are well-defined.
\end{remark}
\begin{theorem}\label{thm:casselman_spherical_vector_non_split}\cite{casselman1980unramified}*{Theorem 3.1}
    For regular \(\sigma\),
    \[T_w(\phi_{L_r^{-\circ},\sigma})=c_w(\sigma)\phi_{L_r^{-\circ},w\sigma}.\]
\end{theorem}
\begin{remark}
    This is the same as Theorem \ref{thm:casselman_spherical_vector}.
\end{remark}
Now we define spherical representations with respect to the special maximal compact subgroup \(L_r^{-\circ}\).

\begin{definition}\label{def:spherical_representation_standard_non_split}
    A spherical representation (for \(H_r\)) is an irreducible representation \((\pi,\cV)\) of \(H_r\) such that
    \(\cV^{L_r^{-}}\neq 0\).
\end{definition}
\begin{remark}
    This is the analogue of the usual notion of an unramified/spherical representation for a special maximal compact subgroup. In this case, it is equivalent to \(\cV^{L_r^{-\circ}}\neq 0\) because we can embed such representations into an unramified principal series representation, and the dimension of the space of \(L_r^{-\circ}\)-invariants is 1 (which is different from the quasi-split case).
\end{remark}
As an easy consequence of the Satake isomorphism, the isomorphism classes of spherical representations with respect to \(L_r^{-\circ}\) are in one-to-one correspondence with the set of \(\sigma\) modulo the action of \(\fW_r\).


\subsection{Spherical Representations with Respect to the almost \texorpdfstring{\(\pi\)}{π}-modular Lattice}
Recall that from \ref{notion:stabilizer_almost_self_dual}, \(K_r^-\) is the stabilizer of the semi-standard lattice
\[
\Lambda_{V_r^-}'=(\bigoplus_{i=1}^{r}\cO_Ev_i)\oplus(\bigoplus_{i=r+1}^{2r}\cO_Ev_i)\oplus(\cO_Ev_{2r+1}\oplus\cO_Ev_{2r+2}).
\]
It is a special compact subgroup of \(H_r\) and we have \(K_r^-=\bigcup_{w\in\fW_r}I_r^-\omega(w)I_r^-\).
    
    It is easy to prove by induction that for any \(w\in\fW_r\)
    \[
        \Omega(w)(\omega(w))^{-1}\in M_r^-
    \]
    \[
        (\omega(w))^{-1}\Omega(w)\in M_r^-.
    \]
    In fact, the above elements lie in the intersection of \(L_r^{-\circ}\) and the product of the diagonal torus with \(\langle t\rangle\) where \(t=\diag\{1_{2r},1,-1\}\).

Now we consider another choice of basis of \((\rI_V^\sigma)^{I_r^-}\):
\begin{definition}
    For \(w \in\fW_r\), we define \(\phi'_{w,\sigma}=\phi_{w,\sigma}(\omega(w))^{-1}\phi_{w,\sigma}\), then \(\phi'_{w,\sigma}(\omega(w))=1\).
\end{definition}
\begin{proposition}\label{prop:basis_relation_non_quasi_split}
    For \(w=w_Iw_\tau\)
        \begin{equation}
            \begin{aligned}
                \phi_{w,\sigma}(\omega(w))
                =&\phi_{w,\sigma}(\omega(w)(\Omega(w))^{-1}\Omega(w))\\
                =&\phi_w(\omega(w_I)\omega(w_\tau)(\Omega(w_\tau))^{-1}(\Omega(w_I))^{-1}\Omega(w))\\
                =&\chi_r^\sigma\delta^{\frac{1}{2}}_{P_r^-}(\omega(w_I)(\Omega(w_I))^{-1})\\
                =&q^{\sum_{i\in I} (\sigma_i+i+\frac{1}{2})}
            \end{aligned}
        \end{equation}
    so
    \[
    \phi'_{w,\sigma}=q^{-\sum_{i\in I} (\sigma_i+i+\frac{1}{2})}\phi_{w,\sigma}
    \]
    and 
    \[
    \phi_{w,\sigma}=q^{\sum_{i\in I} (\sigma_i+i+\frac{1}{2})}\phi_{w,\sigma}'
    \]
\end{proposition}

\begin{remark}
    By our choice of the unimodular lattice, this coincides with Proposition \ref{prop:basis_relation}.
\end{remark}

\begin{lemma}\label{lem:casselman_relation_self_dual_non_split}
    Assume that \(\sigma\) is regular, \(w=w_Iw_\tau\), \(\alpha\in\Delta\) and \(\ell(w_\alpha w)>\ell(w)\).
    \begin{enumerate}[(1)]
        \item 
        For \(\alpha=\epsilon_1\), we have
        \begin{equation}
            \begin{aligned}
                T_{w_1}\phi'_{w,\sigma}=&(c_\alpha(\sigma)-1)\phi'_{w,w_1\sigma}+q^{-\sigma_1-\frac{1}{2}}\phi_{w_1w,w_1\sigma}'\\
                T_{w_1}\phi'_{w_1w,\sigma}=&q^{-\sigma_1-\frac{3}{2}}\phi'_{w,w_1\sigma}+q^{-2\sigma_1}(c_{\epsilon_1}(\sigma)-q^{-2})\phi_{w_1w,w_1\sigma}'
            \end{aligned}
        \end{equation}
        \item
        For \(\alpha=\epsilon_{i+1}-\epsilon_i\),
        \begin{enumerate}[(i)]
        \item
        If (\(i\in I\) and \((i+1)\in I\)) or ((\(i\notin I\)) and \((i+1)\notin I\)) then
        \begin{equation}
            \begin{aligned}
                T_{w_{1,2}}\phi'_{w,\sigma}=&(c_\alpha(\sigma)-1)\phi'_{w,w_{1,2}\sigma}+q^{-1}\phi_{w_{1,2}w,w_{1,2}\sigma}'\\
                T_{w_{1,2}}\phi'_{w_{1,2}w,\sigma}=&\phi'_{w,w_{1,2}\sigma}+(c_\alpha(\sigma)-q^{-1})\phi_{w_{1,2}w,w_{1,2}\sigma}'
            \end{aligned}
        \end{equation}
        \item
        if \(i\in I,(i+1)\notin I\) then
        \begin{equation}
            \begin{aligned}
                T_{w_{1,2}}\phi'_{w,\sigma}=&(c_\alpha(\sigma)-q^{-1})\phi'_{w,w_{1,2}\sigma}+\phi_{w_{1,2}w,w_{1,2}\sigma}'\\
                T_{w_{1,2}}\phi'_{w_{1,2}w,\sigma}=&q^{-1}\phi'_{w,w_{1,2}\sigma}+(c_\alpha(\sigma)-1)\phi_{w_{1,2}w,w_{1,2}\sigma}'
            \end{aligned}
        \end{equation}
        \item It is not the case that \(i\notin I, (i+1)\in I\) provided \(\ell(w_{i,i+1}w)>\ell(w)\). 
        \end{enumerate}
    \end{enumerate}
\end{lemma}
\begin{proof}
    \begin{enumerate}[(1)]
        \item By Theorem \ref{thm:casselman_relation_non_split}, we have
        \[
            T_{w_1}\phi_{w,\sigma}=(c_{\epsilon_1}(\sigma)-1)\phi_{w,w_1\sigma}+q^{-2}\phi_{w_1w,w_1\sigma}
        \]
        and
        \[
            T_{w_1}\phi_{w_1w,\sigma}=\phi_{w,w_1\sigma}+(c_{\epsilon_1}(\sigma)-q^{-2})\phi_{w_1w,w_1\sigma}
        \]
        proceed with Proposition \ref{prop:basis_relation_non_quasi_split} and the definition of \(\phi_w'\), we have
        \begin{equation}
            \begin{aligned}
                T_{w_1}\phi'_{w,\sigma}=&\phi_{w,\sigma}(\omega(w))^{-1}T_{w_1}\phi_{w,\sigma}\\
                =&\phi_{w,\sigma}(\omega(w))^{-1}((c_{\epsilon_1}(\sigma)-1)\phi_{w,w_1\sigma}+q^{-2}\phi_{w_1w,w_1\sigma})\\
                =&\phi_{w,\sigma}(\omega(w))^{-1}((c_{\epsilon_1}(\sigma)-1)(\phi_{w,w_1\sigma}(\omega(w)))\phi'_{w,w_1\sigma}\\
                +&q^{-2}(\phi_{w_1w,w_1\sigma}(\omega(w_1w)))\phi'_{w_1w,w_1\sigma})\\
                =&(\phi_{w,\sigma}(\omega(w_I))^{-1}\phi_{w,w_1\sigma}(\omega(w_I)))(c_{\epsilon_1}(\sigma)-1)\phi'_{w,w_1\sigma}\\
                +&(\phi_{w,\sigma}(\omega(w_I))^{-1}\phi_{w_1w,w_1\sigma}(\omega(w_1w_I)))q^{-2}\phi'_{w_1w,w_1\sigma}
            \end{aligned}
        \end{equation}
        By Lemma \ref{lem:trick}, we know that \(1\notin I\), then we have 
        \begin{equation}
            \begin{aligned}
                T_{w_1}\phi'_{w,\sigma}=&(c_{\epsilon_1}(\sigma)-1)\phi'_{w,w_1\sigma}+(\phi_{w,\sigma}(\omega(w_I))^{-1}\phi_{w_1w,w_1\sigma}(\omega(w_1w_I)))q^{-2}\phi_{w_1w,w_1\sigma}'\\
                =&(c_{\epsilon_1}(\sigma)-1)\phi'_{w,w_1\sigma}+q^{-2+(w_1\sigma)_1+\frac{3}{2}}\phi_{w_1w,w_1\sigma}'\\
                =&(c_\alpha(\sigma)-1)\phi'_{w,w_1\sigma}+q^{-\sigma_1-\frac{1}{2}}\phi_{w_1w,w_1\sigma}'
            \end{aligned}
        \end{equation}
        Similarly, we have
        \begin{equation}
            \begin{aligned}
                T_{w_1}\phi'_{w_1w,\sigma}=&\phi_{w_1w,\sigma}(\omega(w_1w))^{-1}T_{w_1}\phi_{w_1w,\sigma}\\
                =&\phi_{w_1w,\sigma}(\omega(w_1w))^{-1}(\phi_{w,w_1\sigma}+(c_{\epsilon_1}(\sigma)-q^{-2})\phi_{w_1w,w_1\sigma})\\
                =&\phi_{w_1w,\sigma}(\omega(w_1w))^{-1}(\phi_{w,w_1\sigma}(\omega(w))\phi'_{w,w_1\sigma}\\
                +&\phi_{w_1w,w_1\sigma}(\omega(w_1w))(c_{\epsilon_1}(\sigma)-q^{-2})\phi'_{w_1w,w_1\sigma})\\
                =&\phi_{w_1w,\sigma}(\omega(w_1w))^{-1}\phi_{w,w_1\sigma}(\omega(w))\phi'_{w,w_1\sigma}\\
                +&\phi_{w_1w,\sigma}(\omega(w_1w))^{-1}\phi_{w_1w,w_1\sigma}(\omega(w_1w))(c_{\epsilon_1}(\sigma)-q^{-2})\phi'_{w_1w,w_1\sigma}\\
                =&q^{-\frac{3}{2}-\sigma_1}\phi'_{w,w_1\sigma}+q^{-2\sigma_1}(c_{\epsilon_1}(\sigma)-q^{-2})\phi_{w_1w,w_1\sigma}'
            \end{aligned}
        \end{equation}
        \item 
        \begin{enumerate}[(i)]
            \item Note that the action of \(w_{i,i+1}\) on $w_I$ is trivial in this case. Combine this fact with Proposition \ref{prop:basis_relation_non_quasi_split}, we have
            \begin{equation}
                \begin{aligned}
                    &\begin{pmatrix}
                        \phi_{w,\sigma}(\omega(w))^{-1}\phi_{w,w_{i,i+1}\sigma}(\omega(w))\\\phi_{w,\sigma}(\omega(w))^{-1}\phi_{w_{i,i+1}w,w_{i,i+1}\sigma}(\omega(w_{i,i+1}w))\\
                        \phi_{w_{i,i+1}w,\sigma}(\omega(w_{i,i+1}w))^{-1}\phi_{w,w_{i,i+1}\sigma}(\omega(w))\\\phi_{w_{i,i+1}w,\sigma}(\omega(w_{i,i+1}w))^{-1}\phi_{w_{i,i+1}w,w_{i,i+1}\sigma}(\omega(w_{i,i+1}w))
                    \end{pmatrix}=
                    \begin{pmatrix}
                        1\\1\\
                        1\\1
                    \end{pmatrix}
                \end{aligned}
            \end{equation}
            The following is a direct calculation using the above equation and Theorem \ref{thm:casselman_relation}.
            \item Note that the action of \(w_{i,i+1}\) on \(w_I\) is non-trivial in this case, it replaces \(I\) by \(I\cup\{i+1\}-\{i\}\). Combine this fact with Proposition \ref{prop:basis_relation_non_quasi_split}, we have
            \begin{equation}
                \begin{aligned}
                    &\begin{pmatrix}
                        \phi_{w,\sigma}(\omega(w))^{-1}\phi_{w,w_{i,i+1}\sigma}(\omega(w))\\\phi_{w,\sigma}(\omega(w))^{-1}\phi_{w_{i,i+1}w,w_{i,i+1}\sigma}(\omega(w_{i,i+1}w))\\
                        \phi_{w_{i,i+1}w,\sigma}(\omega(w_{i,i+1}w))^{-1}\phi_{w,w_{i,i+1}\sigma}(\omega(w))\\\phi_{w_{i,i+1}w,\sigma}(\omega(w_{i,i+1}w))^{-1}\phi_{w_{i,i+1}w,w_{i,i+1}\sigma}(\omega(w_{i,i+1}w))
                    \end{pmatrix}=
                    \begin{pmatrix}
                        q^{\sigma_{i+1}-\sigma_i}\\q\\
                        q^{-1}\\q^{\sigma_i-\sigma_{i+1}}
                    \end{pmatrix}
                \end{aligned}
            \end{equation}
            \item This follows from the proof of Lemma \ref{lem:trick} and \cite{bjorner2005combinatorics}*{Proposition 8.1.1}.
        \end{enumerate}
    \end{enumerate}
\end{proof}
\begin{remark}
    In Casselman's notation, it is easy to verify that 
    \begin{equation*}
        \begin{aligned}
            &\frac{(1-q_{\alpha/2}^{-\frac{1}{2}}q_{\alpha}^{-1}\chi(a_\alpha))(q_{\alpha/2}^{-\frac{1}{2}}+\chi(a_\alpha))\chi(a_\alpha)}{1-\chi(a_\alpha)^2}\\
            =&c_{\alpha}(\chi)-1+q_{\alpha/2}^{-\frac{1}{2}}q_{\alpha}^{-1}\chi(a_\alpha)\\
            =&q_{\alpha/2}^{-1}q_{\alpha}^{-1}\cdot q_\alpha q_{\alpha}^{\frac{1}{2}}\cdot\chi(a_\alpha)+\chi(a_\alpha)^2(c_{\alpha}(\chi)-q_{\alpha/2}^{-1}q_{\alpha}^{-1})
        \end{aligned}
    \end{equation*}
     which is expected but not obvious.
\end{remark}

The following is an analogue of Theorem \ref{thm:casselman_relation_semi_standard}
\begin{theorem}\label{thm:casselman_relation_semi_standard_non_split}
    Assume that \(\sigma\) is regular, then
    
    we have
    \[
        T_w(\phi_{K_r^-})=\prod_{\alpha\in\Phi^+,w\alpha<0,\alpha\neq 2\epsilon_i}c_\alpha'(\sigma)\phi_{K_r^-}
    \]
    where
    \begin{equation}
        \begin{aligned}
            c'_{\epsilon_i}(\sigma):=&c_{\epsilon_i}(\sigma)-1+\phi_{w_i}(\omega(w_i))^{-1}\\
            =&c_{\epsilon_i}(\sigma)-1+q^{-\sigma_i-\frac{3}{2}}\\
            =&\frac{q^{-\sigma_i}(1-q^{-\frac{3}{2}-\sigma_i})(q^{-\frac{1}{2}}+q^{-\sigma_i})}{1-q^{-2\sigma_i}}
        \end{aligned}
    \end{equation}
    and
    \[
        c'_{\alpha}(\sigma)=c_{\alpha}(\sigma)
    \]
    for \(\alpha=\epsilon_{j}\pm\epsilon_{i},i<j\).
\end{theorem}
\begin{proof}
    Because \(\phi_{L^-_r}=\sum_{w\in\fW_r}\phi'_w\), the result follows from the argument in Theorem \ref{thm:casselman_relation_almost_spherical} and Lemma \ref{lem:casselman_relation_self_dual_non_split}.
\end{proof}

We also have the analogue of Theorem \ref{thm:spherical_L_is_spherical_K}:
\begin{theorem}\label{thm:spherical_L_is_spherical_K_non_split}
    Let \((\pi,\cV)\) be a representation of \(H_r\) which is spherical with respect to \(L_r^-\), and assume that the Satake parameter does not contain \(-q^{\pm\frac{1}{2}}\), then \(\pi\) is spherical with respect to \(K_r^-\).
\end{theorem}

\subsection{Irreducibility of Unitary Principal Series Representations}
In this subsection, we state the results about non-quasi-split unitary groups which are analogues of those in Subsection \ref{subsec:irreducibility_unitary_principal_series}.

We will prove in the next subsection that unitary principal series are in fact irreducible.
\begin{theorem}\label{thm:irreducibility_unitary_principal_series_non_split}
    Let \(\sigma\) be unitary, then the principal series representation \(\rI^\sigma_{V}\) is irreducible.
\end{theorem}
\begin{proof}
    The proof is the same as that of Theorem \ref{thm:irreducibility_unitary_principal_series}.
\end{proof}
\begin{remark}
    We further know that unramified unitary principal series are exactly those spherical representations with respect to \(L_r^{-\circ}\) which are tempered, this is different from the quasi-split case as \(L_r\) there is not a special maximal compact subgroup.
\end{remark}
\section{Doubling Zeta Integrals and the Local Theta Correspondence}\label{sec:doubling}
In this section we compute some doubling zeta integrals and deduce some results on the local theta lifting.

The main results in this section can be found in Proposition \ref{prop:zeta_spherical_self_dual},\ref{prop:zeta_spherical_unimodular_split}, \ref{prop:zeta_almost_spherical_almost_self_dual} (this is used for theta correspondence), \ref{prop:zeta_almost_spherical_unimodular_non_split} (this is used for the computation of normalised zeta integrals of Theorem \ref{thm:aipf}) and Theorem \ref{thm:L_function}.

\subsection{Recollections on the Doubling Method}
Using the notations from \cite{liu2022theta}*{Section 4}, we consider the unitary group \(G^\square_r\cong G_{2r}\) associated to the skew-Hermitian space \(W^\square_{r}\cong W_{2r}\).
\begin{proposition}\label{prop:double_iwasawa}\label{prop:double_bruhat}
    Recall that we have a weaker version of Bruhat decomposition \cite{liu2022theta}*{Page 214, above Remark 5.1}:
    \[
        K_r^\square=\bigsqcup_{i=0}^{2r}\cB^\square_i=\bigsqcup_{i=0}^{2r}I^\square_{r}\omega(w_1\cdots w_i) I^\square_r
    \]
    and Iwasawa decomposition
    \[
        G_r^\square=\bigsqcup_{I\subset\{\pm1\}^{2r}}P_r^\square \omega(w_I)I_r^\square=\bigsqcup_{i=0}^{2r} P_r^\square \cB_i^\square
    \]
\end{proposition}

We have a degenerate principal series representation:
\begin{definition}
    For each complex number \(s\), the degenerate principal series representation of \(G^\square_r\) is defined by
    \[
    \rI^\square_r(s):=\Ind^{G^\square_r}_{P^\square_r}(|\cdot|_E^s\circ\Delta)
    \]
\end{definition}
\begin{remark}\label{rem:degenerate_principal_series}
\begin{enumerate}[(i)]
    \item 
    Let \(\sigma_s^\square:=(s+r-\frac{1}{2},\cdots,s-r+\frac{1}{2})\in(\dC/\frac{2\pi i}{\log q}\dZ)^{2r}\), then we know that \(\rI^\square_r(s)\subset\rI^{\sigma_s^\square}_{2r}\).
    \item 
    For \(i=0,\cdots, 2r\), if we consider the function \(\phi_{i,s}\) given by
    \[
    \phi_{i,s}(pb)=\begin{cases}
        |\Delta(p)|^s_E&b\in \cB_i^\square\\
        0&b\notin\cB_i^\square
    \end{cases}
    \]
    for \(p\in P_r^\square \) and \(b\in K_r^\square\).
    
    We have
    \[
    \phi_{i,s}=\sum_{w,\ell_{long}(w)=i}\phi_{w,\sigma_s^\square}
    \]
    by evaluating at each \(\omega(w)\).
\end{enumerate}
\end{remark}

\begin{remark}\label{rem:recall_sections}    
    Furthermore, in \(\rI^{\sigma}_{2r}\) we have
    \begin{itemize}
        \item \[\phi_{K_{2r},\sigma}=\sum_{I\subset\{1,\cdots,2r\}}\sum_{\tau\in S_r}\phi_{w_Iw_\tau}\]
        \item \[\phi_{L_{2r},\sigma}=\sum_{I\subset\{1,\cdots,2r\}}q^{-\sum_{i\in I} (\sigma_i+i-\frac{1}{2})}\sum_{\tau\in S_r}\phi_{w_Iw_\tau}\]
        \item \[\phi^-_{K_{2r},\sigma}=\sum_{I\subset\{1,\cdots,2r\}}(-q)^{-|I|}\sum_{\tau\in S_r}\phi_{w_Iw_\tau}\]
        \item \[\phi^-_{L_{2r},\sigma}=\sum_{I\subset\{1,\cdots,2r\}}(-1)^{|I|}q^{-\sum_{i\in I} (\sigma_i+i-\frac{1}{2})}\sum_{\tau\in S_r}\phi_{w_Iw_\tau}\]
 \end{itemize}
    
\end{remark}
Here we recall the definition of doubling zeta integral from \cite{liu2022theta}.
\begin{definition}[Doubling zeta integral]\label{def:doubling_zeta_integral}

    Let \(\pi\) be an irreducible admissible representation of \(G_r(F)\). For every element \(\xi \in\pi^\vee \boxtimes\pi\), we denote by \(H_\xi \in C^\infty(G_r(F))\) the associated matrix coefficient. Then for every meromorphic section \(f^{(s)}\) of \(\rI^\square_r(s)\), we have the (doubling) zeta integral:
    \[Z(\xi, f^{(s)}) :=\int_{G_r(F)}H_\xi(g)f^{(s)}(\tw_r(g, 1_{2r})) \rd g,\]
    which is absolutely convergent for Re s large enough and has a meromorphic continuation.

    We let \(L(s, \pi)\) and \(\varepsilon(s, \pi, \psi_F )\) be the doubling \(L\)-factor and the doubling epsilon factor of \(\pi\), respectively, defined in \cite{yamana2014lfunctions}*{Theorem 5.2}.
    
\end{definition}
In the latter subsections, we are going to compute \(Z(\xi,f^{(s)})\) where \(\xi\) comes from a spherical representation with respect to \(K_r\) or an almost spherical representation with respect to \(K_r\) and \(f^{(s)}\) comes from certain Siegel--Weil section associated to a characteristic function of some unimodular lattices. 

We briefly recall the notions of Schr\"odinger model of the Weil representation here \cite{liu2022theta}:

Let \((V,(\cdot,\cdot)_V)\) be a Hermitian space of dimension \(2d\) and of local root number \(\epsilon\). For an element \(x=(x_1,\cdots,x_r)\in V^r\), we denote by
\[T(x) := ((x_i, x_j)_V )_{1\leq i,j\leq r} \in \mathrm{Herm}_r(F)\]
the moment matrix of \(x\). Put \(\Sigma_r(V) := \left\{x \in V^r \mid T(x) = 0_r\right\}\). We have the Fourier transform \(C^\infty_c(V^r) \rightarrow C^\infty_c(V^r)\) sending \(\phi\) to \(\hat\phi\) defined by the formula
    \[
    \hat{\phi} (x) :=\int_{V^r}\phi(y)\psi_E\left(\sum_{i=1}^r(x_i, y_i)_V\right)\rd y,
    \]
where \(\rd y\) is the self-dual Haar measure on \(V^r\) with respect to \(\psi_E\).

Let \((\omega_{W_r,V} , \cV_{W_r,V} )\) be the Weil representation of \(G_r(F) \times \rU(V)(F)\) (with respect to the additive character \(\psi_F\) and the trivial splitting character). We recall the action under
the Schr\"odinger model \(\cV_{W_r,V}\cong C^\infty_c(V^r)\) as follows:
\begin{itemize}
    \item 
    for \(a \in \GL_r(E)\) and \(\phi \in C^\infty_c(V^r)\), we have
    \[\omega_{W_r,V} (m(a))\phi(x) = |\det a|_E^d \cdot \phi(xa);\]
    \item 
    for \(b \in \mathrm{Herm}_r(F)\) and \(\phi \in C^\infty_c(V^r)\), we have
    \[\omega_{W_r,V} (m(b))\phi(x) = \psi_F (\tr bT(x)) \cdot \phi(x);\]
    \item 
    for \(\phi \in C^\infty_c(V^r)\), we have
    \[\omega_{W_r,V}\left(
    \begin{pmatrix}
        &1_r\\
        -1_r&
    \end{pmatrix}\right)
    \phi(x) = (\epsilon 1)^r \cdot \hat{\phi}(x);\]
    \item 
    for \(h \in \rU(V)(F)\) and \(\phi \in C^\infty_c(V^r)\), we have
    \[\omega_{W_r,V}(h)\phi(x) = \phi(h^{-1}x).\]
\end{itemize}
\begin{definition}[The Standard Section]\label{def:standard_section}
    Let \(s_0=d-r\), in our application, we will have \(d=r\) so \(s_0=0\).

    We have a \(G_r^\square\)-intertwining morphism
    \[
        C_c^\infty(V^{2r})\longrightarrow\rI^\square_r(s_0),\quad \Phi\mapsto f_\Phi:=(g\mapsto(\omega_{W_{2r},V}(g)\Phi)(0)).
    \]
    Let \(f^{(s)}_\Phi\) be the unique section of \(\rI_r^\square(s)\) whose value at \(s=s_0\) is \(f_\Phi\) and \(f^{(s)}_\Phi|_{K_r^\square}\) is independent of \(s\), it is called the standard section associated to \(\Phi\).
\end{definition}

We prove a lemma which will be used in the computation
\begin{lemma}\label{lem:volume_computation}
    \[
    \vol(\cB_j)=\frac{q^{\binom{j+1}{2}}\qbinom{r}{j}_q}{\prod_{i=1}^r(q^i+1)}
    \]
    Moreover, as a direct consequence, we have
    \[
        [\cB_i:\cB_0]=q^{\binom{i+1}{2}}\qbinom{r}{i}_q,
    \]
    \[
    C:=\sum_{i=0}^{r}q^{-2i}\vol(\cB_i)=\frac{2(q+1)}{q(q^r+1)(q^{r-1}+1)},
    \]
    \[
    C_+(s):=\sum_{i=0}^{r}q^{si}\vol(\cB_i)=\prod_{i=1}^r\frac{1+q^{i+s}}{1+q^i},
    \]
    \[
    C_+'(s):=\sum_{i=0}^{r}q^{-i(s+r)}\vol(\cB_i)=\prod_{i=1}^{r}\frac{1+q^{-s-i+1}}{1+q^i},
    \]
    \[
    C_-(s):=\sum_{i=0}^{r}q^{-i+si}\vol(\cB_i)=\prod_{i=1}^r\frac{1+q^{i-1+s}}{1+q^i}
    \]
    and
    \[
    C_-'(s):=\sum_{i=0}^{r}q^{-i-i(s+r)}\vol(\cB_i)=\prod_{i=1}^{r}\frac{1+q^{-s-i}}{1+q^i}
    \]
\end{lemma}
\begin{proof}
    This is similar to \cite{liu2022beilinson}*{Lemma B.4.5}\footnote{But the total space here is a skew-symplectic space. Moreover, In \cite{liu2022beilinson}, the extension is unramified so the parahoric subgroup for skew-Hermitian space and Hermitian space are the same, but now they are not the same.}.
    
    For \(w\in\mathfrak{M}_r\)
    \[
        I_r^0\omega(w) I_r^0=\coprod_{i\in I_r^0/(I_r^0\cap \omega(w)I_r^0w^{-1})}i\omega(w)I_r^0
    \]
    So we need to enumerate \(I_r^0/(I_r^0\cap \omega(w)I_r^0\omega(w^{-1}))\) for \(w=w_1\cdots w_i,0\leq i\leq r\).
    
    Note that \(\mathrm{char} (\cO_E/\fp_E)\neq 2\), replace the identity \(^tA^c+A=0\) in \cite{liu2022beilinson}*{Lemma B.4.5} by \(^tA=A\), then we know that for a given maximal isotropic subspace, the number of maximal isotropic subspaces whose overlap has dimension \((r-i)\) is \(q^{\binom{i+1}{2}}\qbinom{r}{i}_q\).

    Then we can prove the following by the property of Gaussian \(q\)-binomial coefficient:
    \[
        \sum_{i=0}^r q^{\binom{i}{2}}\qbinom{r}{i}_qx^i=\prod_{i=0}^{r-1}(1+q^ix)
    \]
    In particular,
    \[
    \sum_{i=0}^r q^{\binom{i+1}{2}}\qbinom{r}{i}_q=\prod_{i=1}^r\frac{q^{2i}-1}{q^i-1}=\prod_{i=0}^{r-1}(q^{i+1}+1),
    \]
    \[
    \sum_{i=0}^r q^{\binom{i+1}{2}-2i}\qbinom{r}{i}_q=\frac{q+1}{q}\prod_{i=0}^{r-2}(q^i+1)
    \]
    and
    \[
    \sum_{i=0}^r q^{\binom{i+1}{2}-(r+1)i}\qbinom{r}{i}_q=\prod_{i=1}^{r-1}(q^{i-1-r}+1)
    \]
    Then 
    \[
    \vol(\cB_j)=\frac{q^{\binom{j+1}{2}}\qbinom{r}{j}_q}{\prod_{i=1}^r(q^i+1)},
    \]
    \[
    \sum_{j=0}^{r}q^{-2j}\vol(\cB_j)=\frac{\frac{q+1}{q}\prod_{i=0}^{r-2}(q^i+1)}{\prod_{i=1}^r(q^i+1)}=\frac{2(q+1)}{q(q^r+1)(q^{r-1}+1)}
    \]
    The remaining are similar computations.
\end{proof}

\subsection{Doubling Zeta Integrals for Spherical Representations}
In this subsection, we compute the doubling zeta integral \(Z(\xi,f_{\Phi}^{(s)})\) where \(\xi\) is the matrix coefficient associated to certain spherical representation with respect to \(K_r\) and \(\Phi=1_{(\Lambda'_{V_r^+})^{2r}}\).

Firstly, we analyze the property of the standard section \(f_{\Phi}^{(s)}\).
    
Consider the action of \(G_r \times \rU(V_d^+)\cong G_r\times G_d\) on \(\Phi=1_{(\Lambda'_{V_d^+})^r}=1_{(\Lambda'_{V_d^+})\otimes_{\cO_E} (\cO_E)^r}\) under the Schr\"odinger model of the Weil representation \(\omega_{W_r,V_d^+}\). We have the following proposition:
\begin{proposition}\label{prop:split_sigel_weil_section}
    \begin{enumerate}[(1)]
        \item \(\Phi\) is invariant under the action of \(L_d\subset G_d\).
        \item \(\Phi\) is invariant under the action of \(I_r^0\subset G_r\).
        \item If \(d=r\), then 
        \[
            f_\Phi=\sum_{i=0}^{2r} q^{-ri}\phi_{i,0}.
        \]
    \end{enumerate}
\end{proposition}
\begin{proof}
    \begin{enumerate}[(1)]
        \item This is because \(L_d\) is the stabilizer of the lattice \(\Lambda'_{V_d^+}\).
        \item Because \(I_r^0=\bigcup_{w\in S_r} I_r\omega(w)I_r\), and \(\omega(S_r)\subset m(GL_r(\cO_E))\), it suffices to prove that \(\Phi\) is invariant under the action of \(I_r\).
        
        As \(\Phi\) is the characteristic function of \((\Lambda'_{V_d^+})^r\), the action of \( P_r^0\cap K_r\) is trivial. Because we have the Iwahori decomposition \cite{kaletha2023bruhat}*{Axiom 4.1.16}, we know that 
        \begin{equation}
        \begin{aligned}
            I_r=(N_r^-\cap I_r)(T\cap K_r)(N_r\cap I_r)&\subset \omega(w_\ell) (N_r\cap I_r)\omega(w_\ell)(T\cap K_r)(N_r\cap I_r)\\
            &\subset\omega(w_\ell) (P_r^0\cap I_r)\omega(w_\ell)(P_r^0\cap K_r)
        \end{aligned}    
        \end{equation}
        where \(w_\ell\) is the longest element of the Weyl group. Note that although the action of \(\omega(w_\ell)\) is the Fourier transform (the local root number is \(1\) here) which is not trivial, the Fourier transformation of \(1_{\Lambda_{V_d^+}^r}\) is \((\vol(\Lambda_{V_d^+}))^r1_{(\Lambda^\vee_{V_d^+})^r}\) which is still invariant under the action of \(K_r\cap P_r^0\), and its Fourier transformation again gives \(1_{\Lambda_{V_d^+}^r}\).
        \item It suffices to evaluate \(f_\Phi\) at \(\omega(w)\) for each \(w\in \fW_{2r}\) as \(f_\Phi\in \rI^{\sigma_r^\square}_{2r}\). For \(0\le i\le 2r\), consider the evaluation at \(\omega(w_1\cdots w_i)\), we have 
        \begin{equation}
            \begin{aligned}
                f_\Phi(\omega(w_1\cdots w_i))=&\omega_{W_{2r},V_r^+}(\omega(w_1\cdots w_i))\Phi(0)\\
                =&\vol(\Lambda_{V_r^+})^i=q^{-ri}
            \end{aligned}
        \end{equation}
       
    \end{enumerate}
\end{proof}
\begin{remark}
        If \(s_0=0\), we know that \(\sigma_0^\square=(r-\frac{1}{2},\dots,-r+\frac{1}{2})\), hence \(f_{\Phi}^{(0)}=\phi_{L_{2r},\sigma_0^\square}\) by \((3)\). 
        Recall that our aim is to compute the doubling zeta integral \(Z(\xi,f_{\Phi}^{(0)})\) where \(\xi\) is the matrix coefficient associated to a spherical representation with respect to \(K_r\), it suffices to compute \(Z(\xi,\phi_{L_{2r},\sigma^\square_s})\).
\end{remark}

The following fact is from \cite{li2022chow}*{Proposition 3.6}
\begin{proposition}\label{prop:zeta_spherical_self_dual}
    For \(\sigma\in\left(\dC/\frac{2\pi i}{\log q}\dZ\right)^r\), and \(\xi^\sigma\) be the matrix coefficient associated to the spherical component \(\pi^\sigma\) with respect to \(K_r\) in \(\rI^\sigma_r\), and \(f_{\Phi}^{(s)}\) be the standard section associated to \(1_{(\Lambda_{V_r^+})^r}\) we have
    \[
        Z(\xi^\sigma,f_{\Phi}^{(s)})=\frac{L^\sigma(s+\frac{1}{2})}{b_{2r}(s)}.
    \]
    where \(L^\sigma(s)=\prod_{i=1}^{r}\frac{1}{(1-q^{\sigma_i-s})(1-q^{-\sigma_i-s})}\) and \(b_{2r}(s)=\prod_{i=1}^r\frac{1}{1-q^{-2s-2i}}\).
\end{proposition}

The main result in this subsection is the following:
\begin{proposition}\label{prop:zeta_spherical_unimodular_split}
    For \(\sigma\in\left(\dC/\frac{2\pi i}{\log q}\dZ\right)^r\), and \(\xi^\sigma\) be the matrix coefficient associated to the spherical component \(\pi^\sigma\) with respect to \(K_r\) in \(\rI^\sigma_r\), we have
    \[
        Z(\xi^\sigma,\phi_{L_{2r},\sigma^\square_s})=(\prod_{i=1}^r\frac{(1+q^{-s-i})(1+q^{-s-i+1})}{(1+q^i)(1+q^i)})\frac{1+q^{-s}}{1+q^{-s-r}}\frac{L^\sigma(s+\frac{1}{2})}{b_{2r}(s)}.
    \]
    where \(L^\sigma(s)=\prod_{i=1}^{r}\frac{1}{(1-q^{\sigma_i-s})(1-q^{-\sigma_i-s})}\) and \(b_{2r}(s)=\prod_{i=1}^r\frac{1}{1-q^{-2s-2i}}\). 
    
    In particular,
    \[
        Z(\xi^\sigma,f_{1_{(\Lambda'_{V_r^+})^{2r}}}^{(0)})=\frac{4}{q^{r^2-r}(1+q^r)^2}\cdot\frac{L^\sigma(\frac{1}{2})}{b_{2r}(0)}.
    \]
\end{proposition}
\begin{proof}
    The proof follows the steps in \cite{liu2022theta}*{Proposition 5.6}.

    For convenience, let \(f^{(s)}=\phi_{L_{2r},\sigma^\square_s}\).

    Recall the isomorphism \(m:\Res_{E/F}\GL_r\overset{\sim}{\longrightarrow}M_r^0\). Let \(\tau\) be the unramified constituent of the normalized induction of \(\boxtimes_{i=1}^r|\cdot|_E^{\sigma_i}\). We fix vectors \(v_0\in\tau\) and \(v_0^\vee\) in \(\tau^\vee\) fixed by \(M_r^0(F)\cap K_r=m(\GL_r(\cO_E))\) such that \(\langle v_0^\vee,v_0\rangle_\tau=1\). Put \(\Pi:=\rI_{P_r^0}^{G_r}(\tau)\) and identify \(\Pi^\vee\) with \(\rI^{G_r}_{P_r^0}(\tau^\vee)\) via the pairing   
    \[
        \langle\varphi^\vee,\varphi\rangle_{\Pi}:=\int_{K_r}\langle\varphi^\vee(k),\varphi(k)\rangle_\tau\rd k.
    \]
    
    Let \(\varphi_0\in\rI_{P_r^0}^{G_r}\tau\) and \(\varphi_0^\vee\in\rI_{P_r^0}^{G_r}(\tau^\vee)\) be the unique spherical vectors with respect to \(K_r\) such that \(\varphi_0(k)=\tau_0\) and \(\varphi_0^\vee(k)=\tau_0^\vee\) for \(k\in K_r\). 

    We know that
    \begin{equation}\label{eq:zeta_spherical_split_unimodular_1}
        \begin{aligned}
            Z(\xi^\sigma, f^{(s)}) :=&\int_{G_r(F)}H_{\xi^\sigma}(g)f^{(s)}(\tw_r(g, 1_{2r})) \rd g\\
            =&\int_{G_r(F)}f^{(s)}(\tw_r(g, 1_{2r}))\langle \Pi^\vee(g)\varphi_0^\vee,\varphi_0\rangle_\Pi \rd g
        \end{aligned}
    \end{equation}
    As in \cite{liu2022theta}*{Page 216}, for every \(k\in K_r\), we have \(\tw_r(k,k)\tw_r^{-1}\in I_r^\square\), hence \[|\Delta(\tw_r(k,k)\tw_r^{-1})|_E=1.\]
    
    We equip \(M_r^0\) and \(N_r^0(F)\) with the Haar measure given there, respectively.
    \begin{equation}\label{eq:zeta_spherical_split_unimodular_2}
        \begin{aligned}
            (\ref{eq:zeta_spherical_split_unimodular_1})=& \int_{G_r(F)} f^{(s)}(\tw_r(g, 1_{2r})) \int_{K_r} \langle \varphi_0^\vee(kg), \varphi_0(k) \rangle_\tau \rd k \rd g \\
        =& \int_{G_r(F)} f^{(s)}(\tw_r(g, 1_{2r}))\sum_{i=0}^{r}\int_{\cB_i}\langle \varphi_0^\vee(kg), \varphi_0(1) \rangle_\tau\rd k \rd g \\
        =&\int_{G_r(F)}\sum_{i=0}^{r}\int_{\cB_i}f^{(s)}(\tw_r(k^{-1}g, 1_{2r}))\langle \varphi_0^\vee(g), \varphi_0(1) \rangle_\tau\rd k \rd g \\
        =&\sum_{i=0}^{r}\int_{G_r(F)}\int_{\cB_i}f^{(s)}(\tw_r(g, k))\langle \varphi_0^\vee(g), \varphi_0(1) \rangle_\tau\rd k \rd g \\
        =&\int_{M_r^0(F)}\int_{N_r^0(F)}\int_{K_r}\sum_{i=0}^r(\int_{\cB_i}f^{(s)}(\tw_r(mnk', k))\langle \varphi_0^\vee(mnk'), \varphi_0(1) \rangle_\tau\rd k)\rd k'\rd n \rd m \\
        =&\int_{M_r^0(F)}\int_{N_r^0(F)}\int_{K_r}\sum_{i=0}^r(\int_{\cB_i}f^{(s)}(\tw_r(mnk', k))\rd k)\langle \varphi_0^\vee(mk'), \varphi_0(1) \rangle_\tau\rd k'\rd n \rd m \\
        =&\int_{M_r^0(F)}\sum_{j=0}^{r}\sum_{i=0}^{r}\int_{\cB_i}\int_{\cB_j}(\int_{N_r^0(F)}f^{(s)}(\tw_r(mnk', k))\rd n)\rd k\rd k'\langle \varphi_0^\vee(m), \varphi_0(1) \rangle_\tau\rd m \\
        \end{aligned}
    \end{equation}
    Set 
    \[
        F^{(s)}(g):=\int_{N_r^0(F)}f^{(s)}(\tw_r(n, 1_{2r})g)\rd n
    \]
    we have
    \begin{equation*}
        \begin{aligned}
            F^{(s)}(m(a)k',k)=&\int_{N_r^0(F)}f^{(s)}(\tw_r(n, 1_{2r})(m(a)k',k))\rd n\\
            =&\int_{N_r^0(F)}f^{(s)}(\tw_r(n, 1_{2r})(m(a),k'))\rd n\\
            =&\int_{N_r^0(F)}f^{(s)}(\tw_r(m(a)nk, k'))\rd n
        \end{aligned}
    \end{equation*}
    where the third equality is because \(N_r^0(F)\) is normal in \(P_r^0(F)\).
    so
    \begin{equation}\label{eq:zeta_spherical_split_unimodular_3}
        \begin{aligned}
            (\ref{eq:zeta_spherical_split_unimodular_2})=&\int_{GL_r(E)}\sum_{j=0}^{r}\sum_{i=0}^{r}\int_{\cB_i}\int_{\cB_j}F^{(s)}(m(a)k',k)\rd k\rd k'\langle \varphi_0^\vee(m(a)), \varphi_0(1) \rangle_\tau\rd a \\
            =&\int_{GL_r(E)}\sum_{j=0}^{r}\sum_{i=0}^{r}\int_{\cB_i}\int_{\cB_j}F^{(s)}(m(a)k',k)\rd k\rd k'|\det a|_E^{-\frac{r}{2}}\langle \tau_0^\vee(a)v_0^
            \vee, \tau_0 \rangle_\tau\rd a
        \end{aligned}
    \end{equation}
    Use the notations in \cite{liu2022theta}*{Page 217, between (5.3) and (5.4)}, we have
    \[
        F^{(s)}(g)=T_{\tw'_r}(\phi_{L_{2r},\sigma^\square_s})(\tw''_rg),
    \]
    and we have
    \[
        T_{\tw'_r}(\phi_{L_{2r},\sigma^\square_s})=C^{+'}_{\tw'_r}(s)\cdot\phi_{L_{2r},\tw'_r\sigma^\square_s}
    \]
    where \[C^{+'}_{\tw'_r}(s)=\prod_{\alpha>0,\tw'_r\alpha<0}c_\alpha'(\sigma^\square_s)\] by Theorem \ref{thm:casselman_relation_semi_standard} and
    \[
    \tw'_r\sigma^\square_s=(-s-\frac{1}{2},\cdots,-s-r+\frac{1}{2},s-\frac{1}{2},\cdots,s-r+\frac{1}{2}),\]
    and
    \begin{equation}\label{eq:twist_sigma}
        (\tw'_r \sigma^\square_s)_i+i-\frac{1}{2}=\begin{cases}
            -s&1\le i\le r\\
            s+r&r+1\le i\le 2r
        \end{cases}
    \end{equation}
    \begin{remark}\label{rmk:twist_sigma}
        This is crucial in the proof of Proposition \ref{prop:zeta_almost_spherical_unimodular_non_split}.
    \end{remark}
    Explicitly, note that \(\tw_r'\) sends \(\epsilon_i\) to \(-\epsilon_{r+1-i}\) and \(\epsilon_{r+i}\) are invariant for \(1\le i\le r\), we have
    \begin{equation}\label{eq:c+'}
        \begin{aligned}
            C^{+'}_{\tw'_r}(s)=&\prod_{\alpha>0,\tw'_r\alpha<0}c_\alpha'(\sigma^\square_s)\\
            =&\prod_{1\le i<j\le r}c'_{\epsilon_i+\epsilon_j}(\sigma^\square_s)\prod_{i=1}^r c'_{2\epsilon_i}(\sigma^\square_s)\\
            =&\prod_{1\le i<j\le r}\frac{1-(q^{-1})^{2s+2r+2-i-j}}{1-(q^{-1})^{2s+2r+1-i-j}}\prod_{i=1}^r\frac{q^{-s-r-1+i}(1-(q^{-1})^{s+r+1-i})(1+(q^{-1})^{s+r-i})}{1-(q^{-1})^{2s+2r+1-2i}}\\
            =&\prod_{1\le i<j\le r}\frac{\zeta_E(2s+2r+1-i-j)}{\zeta_E(2s+2r+2-i-j)}\prod_{i=1}^rq^{-s-r-1+i}\frac{\zeta_E(2s+2i-1)(1+(q^{-1})^{s+r-i})}{\zeta_E(s+i)}\\
            =&q^{-rs-\frac{r(r+1)}{2}}\prod_{1\le i<j\le r}\frac{\zeta_E(2s+2r+1-i-j)}{\zeta_E(2s+2r+2-i-j)}\prod_{i=1}^r\frac{\zeta_E(2s+2i-1)(1+(q^{-1})^{s+r-i})}{\zeta_E(s+i)}\\
            =&q^{-rs-\frac{r(r+1)}{2}}\prod_{i=1}^r\frac{\zeta_E(2s+2i)}{\zeta_E(2s+r+i)}\prod_{i=1}^r\frac{\zeta_E(2s+2i-1)(1+(q^{-1})^{s+r-i})}{\zeta_E(s+i)}\\
            =&q^{-rs-\frac{r(r+1)}{2}}\prod_{i=1}^r\frac{\zeta_E(2s+2i)}{\zeta_E(s+i)}\frac{(1+(q^{-1})^{s+i-1})\zeta_E(2s+2i-1)}{\zeta_E(2s+r+i)}\\
            =&q^{-rs-\frac{r(r+1)}{2}}\prod_{i=1}^r\frac{(1+(q^{-1})^{s+i-1})}{(1+(q^{-1})^{s+i})}\frac{\zeta_E(2s+2i-1)}{\zeta_E(2s+r+i)}\\
            =&q^{-rs-\frac{r(r+1)}{2}}\frac{1+q^{-s}}{1+q^{-s-r}}\prod_{i=1}^r\frac{\zeta_E(2s+2i-1)}{\zeta_E(2s+r+i)}
        \end{aligned}
    \end{equation}
    From \cite{li2022chow}*{Proposition 3.6}, we define
    \[
        C_{\tw'_r}^{+}(s):=\prod_{i=1}^{r}\frac{\zeta_E(2s+2i-1)}{\zeta_E(2s+r+i)}
    \]
    so
    \[
        \frac{C_{\tw'_r}^{+'}(s)}{C_{\tw'_r}^{+}(s)}=q^{-rs-\frac{r(r+1)}{2}}\frac{1+q^{-s}}{1+q^{-s-r}}
    \]
    and we have
    \begin{equation}\label{eq:zeta_spherical_split_unimodular_4}
        \begin{aligned}
            (\ref{eq:zeta_spherical_split_unimodular_3})=&\sum_{i=0}^r\sum_{j=0}^rq^{si}q^{-j(s+r)}\vol(\cB_i)\vol(\cB_j)C^{+'}_{\tw_r}(s)\cdot\\
            &\int_{GL_r(E)}\phi_{L_{2r},\tw'_r\sigma^\square_s}(\tw''_r(m(a),1_{2r}))|\det a|_E^{-\frac{r}{2}}\langle\tau^\vee(a)v_0^\vee,v_0\rangle_\tau\rd a\\
            =&C_+(s)C_+'(s)C^{+'}_{\tw_r}(s)\int_{GL_r(E)}\phi_{L_{2r},\tw'_r\sigma^\square_s}(\tw''_r(m(a),1_{2r}))|\det a|_E^{-\frac{r}{2}}\langle\tau^\vee(a)v_0^\vee,v_0\rangle_\tau\rd a\\
            =&C_+(s)C_+'(s)C^{+'}_{\tw_r}(s)\int_{GL_r(E)}\phi_{K_{2r},\tw'_r\sigma^\square_s}(\tw''_r(m(a),1_{2r}))|\det a|_E^{-\frac{r}{2}}\langle\tau^\vee(a)v_0^\vee,v_0\rangle_\tau\rd a\\
            =&(\prod_{i=1}^r\frac{(1+q^{i+s})(1+q^{-s-i+1})}{(1+q^i)(1+q^i)})q^{-rs-\frac{r(r+1)}{2}}\frac{1+q^{-s}}{1+q^{-s-r}}\frac{L^\sigma(s+\frac{1}{2})}{b_{2r}(s)}\\
            =&(\prod_{i=1}^r\frac{(1+q^{-s-i})(1+q^{-s-i+1})}{(1+q^i)(1+q^i)})\frac{1+q^{-s}}{1+q^{-s-r}}\frac{L^\sigma(s+\frac{1}{2})}{b_{2r}(s)}\\
        \end{aligned}
    \end{equation}
    where the first equality follows from a modified version of the argument in \cite{liu2022theta}*{Page 218, below (5.7)}: though the functions are not constant on \(\cB_{i+j}^\square\), they are constant when restricted on \(\cB_i\times\cB_j\) by (\ref{eq:twist_sigma}).The forth identity is by \cite{li2022chow}*{Proposition 3.6} and Lemma \ref{lem:volume_computation}.
\end{proof}
\subsection{Doubling Zeta Integrals for Almost Spherical Representations}
In this subsection we are going to compute (the special value of) the doubling zeta integral \(Z(\xi,f_{\Phi}^{(s)})\) where \(\xi\) is the matrix coefficient associated to certain almost spherical representation with respect to \(K_r\) and \(K_r^\square\) and \(\Phi=1_{(\Lambda_{V_{r-1}^-})^{2r}}\) or \(\Phi'=1_{(\Lambda'_{V_{r-1}^-})^{2r}}\).

Firstly, we analyze the property of the standard section \(f_{\Phi}^{(s)}\).
    
Consider the action of \(G_r \times \rU(V_d^-)\cong G_r\times H_{d}\)\footnote{Recall that \(V_d^-\) has dimension \(2d+2\).} on \(\Phi=1_{(\Lambda_{V_d^-})^r}=1_{(\Lambda_{V_d^-})\otimes_{\cO_E} (\cO_E)^r}\) and \(\Phi'=1_{(\Lambda'_{V_d^-})^r}=1_{(\Lambda'_{V_d^-})\otimes_{\cO_E} (\cO_E)^r}\) under the Schr\"odinger model of the Weil representation \(\omega_{W_r,V_d^+}\). We have the following proposition:
\begin{proposition}\label{prop:non_split_sigel_weil_section}
    \begin{enumerate}[(1)]
        \item \(\Phi\) is invariant under the action of \(K^-_d\subset H_d\).
        \item \(\Phi'\) is invariant under the action of \(L_d^-\subset H_d\).
        \item \(\Phi\) and \(\Phi'\) are invariant under the action of \(I_r^0\subset G_r\).
        \item If \(d=r-1\), then 
        \[
            f_\Phi=\sum_{i=0}^{2r} (-q)^{-i}\phi_{i,0}.
        \]
        \item If \(d=r-1\), then
        \[
            f_{\Phi'}=\sum_{i=0}^{2r} (-q)^{-ri}\phi_{i,0}.
        \]
    \end{enumerate}
\end{proposition}
\begin{proof}
    \begin{enumerate}[(1)]
        \item This is because \(K_d^-\) is the stabilizer of the lattice \(\Lambda_{V_d^-}\).
        \item This is because \(L_d^-\) is the stabilizer of the lattice \(\Lambda'_{V_d^-}\).
        \item This is the same as that in Proposition \ref{prop:split_sigel_weil_section}.
        \item It suffices to evaluate \(f_\Phi\) at \(\omega(w)\) for each \(w\in \fW_{2r}\) as \(f_\Phi\in \rI^{\sigma_r^\square}_{2r}\). For \(0\le i\le 2r\), consider the evaluation at \(\omega(w_1\cdots w_i)\), we have 
        \begin{equation}
            \begin{aligned}
                f_\Phi(\omega(w_1\cdots w_i))=&\omega_{W_{2r},V_r^+}(\omega(w_1\cdots w_i))\Phi(0)\\
                =&\vol(\Lambda_{V_r^-})^i=q^{-i}
            \end{aligned}
        \end{equation}
        \item It suffices to evaluate \(f_\Phi\) at \(\omega(w)\) for each \(w\in \fW_{2r}\) as \(f_\Phi\in \rI^{\sigma_r^\square}_{2r}\). For \(0\le i\le 2r\), consider the evaluation at \(\omega(w_1\cdots w_i)\), we have 
        \begin{equation}
            \begin{aligned}
                f_\Phi(\omega(w_1\cdots w_i))=&\omega_{W_{2r},V_r^+}(\omega(w_1\cdots w_i))\Phi'(0)\\
                =&\vol(\Lambda'_{V_r^-})^i=q^{-ri}
            \end{aligned}
        \end{equation}
    \end{enumerate}
\end{proof}
\begin{remark}
        It is easy that \(f_{\Phi}^{(s)}=\phi^-_{K_{2r},\sigma_s^\square}\) by (4). Moreover, if \(s=0\), we also have
        \(\sigma_0^\square=(r-\frac{1}{2},\dots,-r+\frac{1}{2})\), \(f_{\Phi'}^{(0)}=\phi^-_{L_{2r},\sigma_0^\square}\) by \((5)\), Remark \ref{def:another_one}. 

    \end{remark}

Recall that our aim is to compute the doubling zeta integral \(Z(\xi,f_{\Phi}^{(0)})\) and \(Z(\xi,f_{\Phi'}^{(0)})\) where \(\xi\) is the matrix coefficient associated to an almost spherical representation with respect to \(K_r\), it suffices to compute \(Z(\xi,\phi^-_{K_{2r},\sigma^\square_s})\) and \(Z(\xi,\phi^-_{L_{2r},\sigma^\square_s})\).

The first main result in this subsection is the following:
\begin{proposition}\label{prop:zeta_almost_spherical_almost_self_dual}
    For \(\sigma\in\left(\dC/\frac{2\pi i}{\log q}\dZ\right)^r\), and \(\xi^\sigma\) be the matrix coefficient associated to the almost spherical component \(\pi_{W_r,-}^\sigma\) with respect to \(K_r\) in \(\rI^\sigma_r\), we have
    \[
        Z(\xi^\sigma,\phi^-_{K_{2r},\sigma^\square_s})=\frac{2(q+1)}{(-q)^{r}q(q^r+1)(q^{r-1}+1)}\frac{1+q^{-s}}{1-q^{-2s-2r}}\frac{L_-^\sigma(s+\frac{1}{2})}{b_{2r}(s)}
    \]
    where \(L_-^\sigma(s)=(1-q^{-s+\frac{1}{2}})\prod_{i=1}^{r}\frac{1}{(1-q^{\sigma_i-s})(1-q^{-\sigma_i-s})}\) and \(b_{2r}(s)=\prod_{i=1}^r\frac{1}{1-q^{-2s-2i}}\). 
    
    In particular, if \(\sigma\) contains \(-\frac{1}{2}\) or \(\frac{1}{2}\)
    \[
        Z(\xi^\sigma,f_{1_{(\Lambda_{V_{r-1}^-})^{2r}}}^{(0)})=\frac{2(q+1)}{(-q)^{r}q(q^r+1)(q^{r-1}+1)}\frac{2}{1-q^{-2r}}\frac{L_-^\sigma(\frac{1}{2})}{b_{2r}(0)}.
    \]
\end{proposition}
\begin{proof}
    The proof follows the steps in \cite{liu2022theta}*{Proposition 5.6}.

    For convenience, let \(f^{(s)}=\phi^-_{K_{2r},\sigma^\square_s}\).

    Recall the isomorphism \(m:\Res_{E/F}\GL_r\overset{\sim}{\longrightarrow}M_r^0\). Let \(\tau\) be the unramified constituent of the normalized induction of \(\boxtimes_{i=1}^r|\cdot|_E^{\sigma_i}\). We fix vectors \(v_0\in\tau\) and \(v_0^\vee\) in \(\tau^\vee\) fixed by \(M_r^0(F)\cap K_r=m(\GL_r(\cO_E))\) such that \(\langle v_0^\vee,v_0\rangle_\tau=1\). Put \(\Pi:=\rI_{P_r^0}^{G_r}(\tau)\) and identify \(\Pi^\vee\) with \(\rI^{G_r}_{P_r^0}(\tau^\vee)\) via the pairing   
    \[
        \langle\varphi^\vee,\varphi\rangle_{\Pi}:=\int_{K_r}\langle\varphi^\vee(k),\varphi(k)\rangle_\tau\rd k.
    \]
    
    Let \(\varphi_0\in\rI_{P_r^0}^{G_r}\tau\) and \(\varphi_0^\vee\in\rI_{P_r^0}^{G_r}(\tau^\vee)\) be the unique vectors such that \(\varphi_0(k)=(-q)^{-i}\tau_0\) and \(\varphi_0^\vee(k)=(-q)^{-i}\tau_0^\vee\) for \(k\in \cB_i\). 

    By definition,
    \begin{equation}\label{eq:zeta_almost_self_dual_1}
        \begin{aligned}
            Z(\xi^\sigma, f^{(s)}) :=&\int_{G_r(F)}H_{\xi^\sigma}(g)f^{(s)}(\tw_r(g, 1_{2r})) \rd g\\
            =&C^{-1}\int_{G_r(F)}f^{(s)}(\tw_r(g, 1_{2r}))\langle \Pi^\vee(g)\varphi_0^\vee,\varphi_0\rangle_\Pi \rd g
        \end{aligned}
    \end{equation}
    where 
    \[
        C=\frac{2(q+1)}{q(q^r+1)(q^{r-1}+1)}
    \]
    by Lemma \ref{lem:volume_computation}.

    As in \cite{liu2022theta}*{Page 216}, for every \(k\in K_r\), we have \(\tw_r(k,k)\tw_r^{-1}\in I_r^\square\), hence \[|\Delta(\tw_r(k,k)\tw_r^{-1})|_E=1.\]
    
    We equip \(M_r^0\) and \(N_r^0(F)\) with the Haar measure given there, respectively.
    \begin{equation}\label{eq:zeta_almost_self_dual_2}
        \begin{aligned}
            (\ref{eq:zeta_almost_self_dual_1})=&C^{-1}\int_{G_r(F)} f^{(s)}(\tw_r(g, 1_{2r})) \int_{K_r} \langle \varphi_0^\vee(kg), \varphi_0(k) \rangle_\tau \rd k \rd g \\
        =& C^{-1}\int_{G_r(F)} f^{(s)}(\tw_r(g, 1_{2r}))\sum_{i=0}^{r}\int_{\cB_i}(-q)^{-i}\langle \varphi_0^\vee(kg), \varphi_0(1) \rangle_\tau\rd k \rd g \\
        =&C^{-1}\int_{G_r(F)}\sum_{i=0}^{r}(-q)^{-i}\int_{\cB_i}f^{(s)}(\tw_r(k^{-1}g, 1_{2r}))\langle \varphi_0^\vee(g), \varphi_0(1) \rangle_\tau\rd k \rd g \\
        =&C^{-1}\sum_{i=0}^{r}(-q)^{-i}\int_{G_r(F)}\int_{\cB_i}f^{(s)}(\tw_r(g, k))\langle \varphi_0^\vee(g), \varphi_0(1) \rangle_\tau\rd k \rd g \\
        =&C^{-1}\int_{M_r^0(F)}\int_{N_r^0(F)}\int_{K_r}\sum_{i=0}^r(-q)^{-i} \\
        &\quad\cdot\biggl(\int_{\cB_i}f^{(s)}(\tw_r(mnk', k))\langle \varphi_0^\vee(mnk'), \varphi_0(1) \rangle_\tau\rd k\biggr)\rd k'\rd n \rd m \\
        =&C^{-1}\int_{M_r^0(F)}\int_{N_r^0(F)}\int_{K_r}\sum_{i=0}^r(-q)^{-i} \\
        &\quad\cdot\biggl(\int_{\cB_i}f^{(s)}(\tw_r(mnk', k))\rd k\biggr)\langle \varphi_0^\vee(mk'), \varphi_0(1) \rangle_\tau\rd k'\rd n \rd m \\
        =&C^{-1}\int_{M_r^0(F)}\sum_{j=0}^{r}\sum_{i=0}^{r}(-q)^{-i-j}\int_{\cB_i}\int_{\cB_j} \\
        &\quad\cdot\biggl(\int_{N_r^0(F)}f^{(s)}(\tw_r(mnk', k))\rd n\biggr)\rd k\rd k'\langle \varphi_0^\vee(m), \varphi_0(1) \rangle_\tau\rd m \\
        \end{aligned}
    \end{equation}
    Set 
    \[
        F^{(s)}(g):=\int_{N_r^0(F)}f^{(s)}(\tw_r(n, 1_{2r})g)\rd n
    \]
    we have
    \begin{equation*}
        \begin{aligned}
            F^{(s)}(m(a)k',k)=&\int_{N_r^0(F)}f^{(s)}(\tw_r(n, 1_{2r})(m(a)k',k))\rd n\\
            =&\int_{N_r^0(F)}f^{(s)}(\tw_r(n, 1_{2r})(m(a),k'))\rd n\\
            =&\int_{N_r^0(F)}f^{(s)}(\tw_r(m(a)nk, k'))\rd n
        \end{aligned}
    \end{equation*}
    where the third equality is because \(N_r^0(F)\) is normal in \(P_r^0(F)\).
    so
    \begin{equation}\label{eq:zeta_almost_self_dual_3}
        \begin{aligned}
            (\ref{eq:zeta_almost_self_dual_2})=&C^{-1}\int_{GL_r(E)}\sum_{j=0}^{r}\sum_{i=0}^{r}(-q)^{-i-j}\int_{\cB_i}\int_{\cB_j}F^{(s)}(m(a)k',k)\rd k\rd k'\langle \varphi_0^\vee(m(a)), \varphi_0(1) \rangle_\tau\rd a \\
            =&C^{-1}\int_{GL_r(E)}\sum_{j=0}^{r}\sum_{i=0}^{r}(-q)^{-i-j}\int_{\cB_i}\int_{\cB_j}F^{(s)}(m(a)k',k)\rd k\rd k'|\det a|_E^{-\frac{r}{2}}\langle \tau_0^\vee(a)v_0^
            \vee, \tau_0 \rangle_\tau\rd a
        \end{aligned}
    \end{equation}
    Use the notations in \cite{liu2022theta}*{Page 217, between (5.3) and (5.4)}, we have
    \[
        F^{(s)}(g)=T_{\tw'_r}(\phi^-_{K_{2r},\sigma^\square_s})(\tw''_rg),
    \]
    and we have
    \[
        T_{\tw'_r}(\phi^-_{K_{2r},\sigma^\square_s})=C^{-}_{\tw'_r}(s)\cdot\phi^-_{K_{2r},\tw'_r\sigma^\square_s}
    \]
    where \[C^{-}_{\tw'_r}(s)=\prod_{\alpha>0,\tw'_r\alpha<0}c_\alpha''(\sigma^\square_s)\] by Theorem \ref{thm:casselman_relation_almost_spherical}.

    Explicitly, note that \(\tw_r'\) sends \(\epsilon_i\) to \(-\epsilon_{r+1-i}\) and \(\epsilon_{r+i}\) are invariant for \(1\le i\le r\), we have
    \begin{equation}\label{eq:c-}
        \begin{aligned}
            C^{-}_{\tw'_r}(s)=&\prod_{\alpha>0,\tw'_r\alpha<0}c_\alpha''(\sigma^\square_s)\\
            =&\prod_{1\le i<j\le r}c''_{\epsilon_i+\epsilon_j}(\sigma^\square_s)\prod_{i=1}^r c''_{2\epsilon_i}(\sigma^\square_s)\\
            =&\prod_{1\le i<j\le r}\frac{1-(q^{-1})^{2s+2r+2-i-j}}{1-(q^{-1})^{2s+2r+1-i-j}}\prod_{i=1}^r\frac{(q^{-1})^{2s+2r-2i}-1}{q(1-(q^{-1})^{2s+2r+1-2i})}\\
            =&(-q)^{-r}\prod_{1\le i<j\le r}\frac{\zeta_E(2s+2r+1-i-j)}{\zeta_E(2s+2r+2-i-j)}\prod_{i=1}^r\frac{\zeta_E(2s+2i-1)}{\zeta_E(2s+2i-2)}\\
            =&(-q)^{-r}\prod_{i=1}^r\frac{\zeta_E(2s+2i)}{\zeta_E(2s+2i-2)}\prod_{i=1}^r\frac{\zeta_E(2s+2i-1)}{\zeta_E(2s+r+i)}\\
            =&(-q)^{-r}\frac{1-q^{-2s}}{1-q^{-2s-2r}}\prod_{i=1}^r\frac{\zeta_E(2s+2i-1)}{\zeta_E(2s+r+i)},
        \end{aligned}
    \end{equation}
    Then
    \[
        \frac{C_{\tw'_r}^{-}(s)}{C_{\tw'_r}^{+}(s)}=(-q)^{-r}\frac{1-q^{-2s}}{1-q^{-2s-2r}}
    \]
    and we have
    \begin{equation}\label{eq:zeta_almost_self_dual_4}
        \begin{aligned}
            (\ref{eq:zeta_almost_self_dual_3})=&C^{-1}C^{-}_{\tw_r}(s)\int_{GL_r(E)}\sum_{j=0}^{r}\sum_{i=0}^{r}(-q)^{-i-j}\cdot\\
            &(\int_{\cB_i}\int_{\cB_j}\phi^-_{K_{2r},\tw'_r\sigma^\square_s}(\tw''_r(m(a)k',k))\rd k\rd k'|\det a|_E^{-\frac{r}{2}}\langle \tau_0^\vee(a)v_0^\vee, \tau_0 \rangle_\tau\rd a)\\
            =&C^{-1+2}C^{-}_{\tw_r}(s)\int_{GL_r(E)}\phi_{K_{2r},\tw'_r\sigma^\square_s}(\tw''_r(m(a),1_{2r}))|\det a|_E^{-\frac{r}{2}}\langle\tau^\vee(a)v_0^\vee,v_0\rangle_\tau\rd a\\
            =&\frac{2(q+1)}{(-q)^{r}q(q^r+1)(q^{r-1}+1)}\frac{1-q^{-2s}}{1-q^{-2s-2r}}\frac{L^\sigma(s+\frac{1}{2})}{b_{2r}(s)}\\
            =&\frac{2(q+1)}{(-q)^{r}q(q^r+1)(q^{r-1}+1)}\frac{1+q^{-s}}{1-q^{-2s-2r}}\frac{L_-^\sigma(s+\frac{1}{2})}{b_{2r}(s)}
        \end{aligned}
    \end{equation}
    where the second equality follows from the argument in \cite{liu2022theta}*{Page 218, below (5.7)} and the third identity is \cite{li2022chow}*{Proposition 3.6}.
\end{proof}

Similarly, we have the following proposition
\begin{proposition}\label{prop:zeta_almost_spherical_unimodular_non_split}
    For \(\sigma\in\left(\dC/\frac{2\pi i}{\log q}\dZ\right)^r\), and \(\xi^\sigma\) be the matrix coefficient associated to the almost spherical component \(\pi^\sigma\) with respect to \(K_r\) in \(\rI^\sigma_r\), we have
    \begin{equation}
    \begin{aligned}
        Z(\xi^\sigma,\phi^-_{L_{2r},\sigma^\square_s})=&(-1)^{r}(\prod_{i=1}^r\frac{(1+q^{-s+1-i})(1+q^{-s-i})}{(1+q^i)^2})\frac{q(q^r+1)(q^{r-1}+1)}{2q^{r}(q+1)}\frac{1}{1-q^{-s-r}}\cdot\\
        &\frac{L_-^\sigma(s+\frac{1}{2})}{b_{2r}(s)}
    \end{aligned}
    \end{equation}
    
    where \(L_-^\sigma(s)=(1-q^{-s+\frac{1}{2}})\prod_{i=1}^{r}\frac{1}{(1-q^{\sigma_i-s})(1-q^{-\sigma_i-s})}\) and \(b_{2r}(s)=\prod_{i=1}^r\frac{1}{1-q^{-2s-2i}}\). 
    
    In particular, if \(\sigma\) contains \(\frac{1}{2}\) or \(-\frac{1}{2}\)
    \begin{equation}
        \begin{aligned}
            Z(\xi^\sigma,f_{1_{(\Lambda'_{V_r^-})^{2r}}}^{(0)})=&(-1)^{r}(\frac{2}{q^{r^2}(1+q^r)})\frac{q(q^r+1)(q^{r-1}+1)}{2q^{r}(q+1)}\frac{1}{1-q^{-r}}\cdot\frac{L_-^\sigma(\frac{1}{2})}{b_{2r}(0)}\\
            =&(-1)^{r}\frac{q(q^{r-1}+1)}{q^{r^2}(q+1)}\frac{1}{q^r-1}\cdot\frac{L_-^\sigma(\frac{1}{2})}{b_{2r}(0)}
        \end{aligned}
        \end{equation}
\end{proposition}
\begin{proof}
    The proof follows the steps in \cite{liu2022theta}*{Proposition 5.6}.

    For convenience, let \(f^{(s)}=\phi^-_{L_{2r},\sigma^\square_s}\).

    Recall the isomorphism \(m:\Res_{E/F}\GL_r\overset{\sim}{\longrightarrow}M_r^0\). Let \(\tau\) be the unramified constituent of the normalized induction of \(\boxtimes_{i=1}^r|\cdot|_E^{\sigma_i}\). We fix vectors \(v_0\in\tau\) and \(v_0^\vee\) in \(\tau^\vee\) fixed by \(M_r^0(F)\cap K_r=m(\GL_r(\cO_E))\) such that \(\langle v_0^\vee,v_0\rangle_\tau=1\). Put \(\Pi:=\rI_{P_r^0}^{G_r}(\tau)\) and identify \(\Pi^\vee\) with \(\rI^{G_r}_{P_r^0}(\tau^\vee)\) via the pairing   
    \[
        \langle\varphi^\vee,\varphi\rangle_{\Pi}:=\int_{K_r}\langle\varphi^\vee(k),\varphi(k)\rangle_\tau\rd k.
    \]
    
    Let \(\varphi_0\in\rI_{P_r^0}^{G_r}\tau\) and \(\varphi_0^\vee\in\rI_{P_r^0}^{G_r}(\tau^\vee)\) be the unique vectors with respect to \(K_r\) such that \(\varphi_0(k)=(-q)^{-i}\tau_0\) and \(\varphi_0^\vee(k)=(-q)^{-i}\tau_0^\vee\) for \(k\in \cB_i\). 

    By definition,
    \begin{equation}\label{eq:zeta_unimodular_non_split_1}
        \begin{aligned}
            Z(\xi^\sigma, f^{(s)}) :=&\int_{G_r(F)}H_{\xi^\sigma}(g)f^{(s)}(\tw_r(g, 1_{2r})) \rd g\\
            =&C^{-1}\int_{G_r(F)}f^{(s)}(\tw_r(g, 1_{2r}))\langle \Pi^\vee(g)\varphi_0^\vee,\varphi_0\rangle_\Pi \rd g
        \end{aligned}
    \end{equation}
    where 
    \[
        C=\frac{2(q+1)}{q(q^r+1)(q^{r-1}+1)}
    \]
    by Lemma \ref{lem:volume_computation}.

    As in \cite{liu2022theta}*{Page 216}, for every \(k\in K_r\), we have \(\tw_r(k,k)\tw_r^{-1}\in I_r^\square\), hence \[|\Delta(\tw_r(k,k)\tw_r^{-1})|_E=1.\]
    
    We equip \(M_r^0\) and \(N_r^0(F)\) with the Haar measure given there, respectively.
    \begin{equation}\label{eq:zeta_unimodular_non_split_2}
        \begin{aligned}
            (\ref{eq:zeta_unimodular_non_split_1})=&C^{-1}\int_{G_r(F)} f^{(s)}(\tw_r(g, 1_{2r})) \int_{K_r} \langle \varphi_0^\vee(kg), \varphi_0(k) \rangle_\tau \rd k \rd g \\
        =& C^{-1}\int_{G_r(F)} f^{(s)}(\tw_r(g, 1_{2r}))\sum_{i=0}^{r}\int_{\cB_i}(-q)^{-i}\langle \varphi_0^\vee(kg), \varphi_0(1) \rangle_\tau\rd k \rd g \\
        =&C^{-1}\int_{G_r(F)}\sum_{i=0}^{r}(-q)^{-i}\int_{\cB_i}f^{(s)}(\tw_r(k^{-1}g, 1_{2r}))\langle \varphi_0^\vee(g), \varphi_0(1) \rangle_\tau\rd k \rd g \\
        =&C^{-1}\sum_{i=0}^{r}(-q)^{-i}\int_{G_r(F)}\int_{\cB_i}f^{(s)}(\tw_r(g, k))\langle \varphi_0^\vee(g), \varphi_0(1) \rangle_\tau\rd k \rd g \\
        =&C^{-1}\int_{M_r^0(F)}\int_{N_r^0(F)}\int_{K_r}\sum_{i=0}^r(-q)^{-i} \\
        &\quad\cdot\biggl(\int_{\cB_i}f^{(s)}(\tw_r(mnk', k))\langle \varphi_0^\vee(mnk'), \varphi_0(1) \rangle_\tau\rd k\biggr)\rd k'\rd n \rd m \\
        =&C^{-1}\int_{M_r^0(F)}\int_{N_r^0(F)}\int_{K_r}\sum_{i=0}^r(-q)^{-i} \\
        &\quad\cdot\biggl(\int_{\cB_i}f^{(s)}(\tw_r(mnk', k))\rd k\biggr)\langle \varphi_0^\vee(mnk'), \varphi_0(1) \rangle_\tau\rd k'\rd n \rd m \\
        =&C^{-1}\int_{M_r^0(F)}\sum_{j=0}^{r}\sum_{i=0}^{r}(-q)^{-i-j}\int_{\cB_i}\int_{\cB_j} \\
        &\quad\cdot\biggl(\int_{N_r^0(F)}f^{(s)}(\tw_r(mnk', k))\rd n\biggr)\rd k\rd k'\langle \varphi_0^\vee(m), \varphi_0(1) \rangle_\tau\rd m \\
        \end{aligned}
    \end{equation}
    Set 
    \[
        F^{(s)}(g):=\int_{N_r^0(F)}f^{(s)}(\tw_r(n, 1_{2r})g)\rd n
    \]
    we have
    \begin{equation*}
        \begin{aligned}
            F^{(s)}(m(a)k',k)=&\int_{N_r^0(F)}f^{(s)}(\tw_r(n, 1_{2r})(m(a)k',k))\rd n\\
            =&\int_{N_r^0(F)}f^{(s)}(\tw_r(n, 1_{2r})(m(a),k'))\rd n\\
            =&\int_{N_r^0(F)}f^{(s)}(\tw_r(m(a)nk, k'))\rd n
        \end{aligned}
    \end{equation*}
    where the third equality is because \(N_r^0(F)\) is normal in \(P_r^0(F)\).
    so
    \begin{equation}\label{eq:zeta_unimodular_non_split_3}
        \begin{aligned}
            (\ref{eq:zeta_unimodular_non_split_2})=&C^{-1}\int_{GL_r(E)}\sum_{j=0}^{r}\sum_{i=0}^{r}(-q)^{-i-j}\int_{\cB_i}\int_{\cB_j}F^{(s)}(m(a)k',k)\rd k\rd k'\langle \varphi_0^\vee(m(a)), \varphi_0(1) \rangle_\tau\rd a \\
            =&C^{-1}\int_{GL_r(E)}\sum_{j=0}^{r}\sum_{i=0}^{r}(-q)^{-i-j}\int_{\cB_i}\int_{\cB_j}F^{(s)}(m(a)k',k)\rd k\rd k'|\det a|_E^{-\frac{r}{2}}\langle \tau_0^\vee(a)v_0^
            \vee, \tau_0 \rangle_\tau\rd a
        \end{aligned}
    \end{equation}
    Use the notations in \cite{liu2022theta}*{Page 217, between (5.3) and (5.4)}, we have
    \[
        F^{(s)}(g)=T_{\tw'_r}(\phi^-_{L_{2r},\sigma^\square_s})(\tw''_rg),
    \]
    and we have
    \[
        T_{\tw'_r}(\phi^-_{L_{2r},\sigma^\square_s})=C^{-'}_{\tw'_r}(s)\cdot\phi^-_{L_{2r},\tw'_r\sigma^\square_s}
    \]
    where \[C^{-'}_{\tw'_r}(s)=\prod_{\alpha>0,\tw'_r\alpha<0}c_\alpha'''(\sigma^\square_s)\] by Theorem \ref{thm:casselman_relation_semi_standard}.

    Explicitly, note that \(\tw_r'\) sends \(\epsilon_i\) to \(-\epsilon_{r+1-i}\) and \(\epsilon_{r+i}\) are invariant for \(1\le i\le r\), we have
    \begin{equation}\label{eq:c-'}
        \begin{aligned}
            C^{-}_{\tw'_r}(s)=&\prod_{\alpha>0,\tw'_r\alpha<0}c_\alpha'''(\sigma^\square_s)\\
            =&\prod_{1\le i<j\le r}c'''_{\epsilon_i+\epsilon_j}(\sigma^\square_s)\prod_{i=1}^r c'''_{2\epsilon_i}(\sigma^\square_s)\\
            =&\prod_{1\le i<j\le r}\frac{1-(q^{-1})^{2s+2r+2-i-j}}{1-(q^{-1})^{2s+2r+1-i-j}}\prod_{i=1}^rq^{-2s-2r-1+2i}\frac{(1-q^{s+r-i})(1+q^{-s-r+i-1})}{(1-(q^{-1})^{2s+2r+1-2i})}\\
            =&\prod_{1\le i<j\le r}\frac{\zeta_E(2s+i+j-1)}{\zeta_E(2s+i+j)}\prod_{i=1}^r(-q^{-s-r-1+i})\frac{(1-q^{-s-r+i})(1+q^{-s-r+i-1})}{(1-(q^{-1})^{2s+2r+1-2i})}\\
            =&(-1)^{r}q^{-rs-\frac{(r+1)r}{2}}\prod_{i=1}^r\frac{\zeta_E(2s+2i)}{\zeta_E(2s+r+i)}
            \prod_{i=1}^r\frac{\zeta_E(2s+2i-1)(1+(q^{-1})^{s+i})}{\zeta_E(s+i-1)}\\
            =&(-1)^{r}q^{-rs-\frac{(r+1)r}{2}}\prod_{i=1}^r\frac{\zeta_E(s+i)}{\zeta_E(2s+r+i)}
            \prod_{i=1}^r\frac{\zeta_E(2s+2i-1)}{\zeta_E(s+i-1)}\\
            =&(-1)^{r}q^{-rs-\frac{(r+1)r}{2}}\frac{\zeta_E(s+r)}{\zeta_E(s)}
            \prod_{i=1}^r\frac{\zeta_E(2s+2i-1)}{\zeta_E(s+r+i)}\\
            =&(-1)^{r}q^{-rs-\frac{(r+1)r}{2}}\frac{1-q^{-s}}{1-q^{-s-r}}
            \prod_{i=1}^r\frac{\zeta_E(2s+2i-1)}{\zeta_E(s+r+i)},\\
        \end{aligned}
    \end{equation}
    Then
    \[
        \frac{C_{\tw'_r}^{-'}(s)}{C_{\tw'_r}^{+}(s)}=(-1)^{r}q^{-rs-\frac{(r+1)r}{2}}\frac{1-q^{-s}}{1-q^{-s-r}}
    \]
    and we have
    \begin{equation}\label{eq:zeta_unimodular_non_split_4}
        \begin{aligned}
            (\ref{eq:zeta_unimodular_non_split_3})
            =&C^{-1}C^{-}_{\tw_r}(s)\int_{GL_r(E)}\sum_{j=0}^{r}\sum_{i=0}^{r}(-q)^{-i-j}\cdot\\
            &(\int_{\cB_i}\int_{\cB_j}\phi^-_{L_{2r},\tw'_r\sigma^\square_s}(\tw''_r(m(a)k',k))\rd k\rd k'|\det a|_E^{-\frac{r}{2}}\langle \tau_0^\vee(a)v_0^\vee, \tau_0 \rangle_\tau\rd a)\\
            =&C^{-1}C^{-}_{\tw_r}(s)\int_{GL_r(E)}\sum_{j=0}^{r}\sum_{i=0}^{r}(-q)^{-i-j}\cdot\\
            &(\int_{\cB_i}\int_{\cB_j}(-1)^{i+j}(q)^{si-j(s+r)}\phi^-_{L_{2r},\tw'_r\sigma^\square_s}(\tw''_r(m(a),1_r))\rd k\rd k'|\det a|_E^{-\frac{r}{2}}\langle \tau_0^\vee(a)v_0^\vee, \tau_0 \rangle_\tau\rd a)\\
            =&C^{-1}C_-(s)C_{-}'(s)C^{-'}_{\tw_r}(s)\cdot\\
            &\int_{GL_r(E)}\phi_{K_{2r},\tw'_r\sigma^\square_s}(\tw''_r(m(a),1_{2r}))|\det a|_E^{-\frac{r}{2}}\langle\tau^\vee(a)v_0^\vee,v_0\rangle_\tau\rd a\\
            =&C^{-1}C_-(s)C_{-}'(s)C^{-'}_{\tw_r}(s)\frac{L^\sigma(s+\frac{1}{2})}{b_{2r}(s)}\\
            =&(\prod_{i=1}^r\frac{(1+q^{s-1+i})(1+q^{-s-i})}{(1+q^i)^2})\frac{q(q^r+1)(q^{r-1}+1)}{2(q+1)}(-1)^{r}q^{-rs-\frac{(r+1)r}{2}}\frac{1}{1-q^{-s-r}}\cdot\\
            &\frac{L_-^\sigma(s+\frac{1}{2})}{b_{2r}(s)}\\
            =&(\prod_{i=1}^r\frac{(1+q^{-s+1-i})(1+q^{-s-i})}{(1+q^i)^2})\frac{q(q^r+1)(q^{r-1}+1)}{2q^{r}(q+1)}(-1)^{r}\frac{1}{1-q^{-s-r}}\cdot\\
            &\frac{L_-^\sigma(s+\frac{1}{2})}{b_{2r}(s)}
        \end{aligned}
    \end{equation}
    where the second equality follows from a modified version of the argument in \cite{liu2022theta}*{Page 218, below (5.7)}: though the functions are not constant on \(\cB_{i+j}^\square\), they are constant when restricted on \(\cB_i\times\cB_j\) by (\ref{eq:twist_sigma}). The forth identity is \cite{li2022chow}*{Proposition 3.6} and the fifth is by Lemma \ref{lem:volume_computation} .
\end{proof}

To end this subsection, we give some consequences of Proposition \ref{prop:zeta_almost_spherical_almost_self_dual} which is parallel to \cite{liu2022theta}*{Corollary 5.7} and \cite{liu2022theta}*{Corollary 5.8}.

\begin{notation}
    We fix \(f^{(s)}=f^{(s)}_\Phi\) where \(\Phi=1_{(\Lambda_{V_{r-1}^-})^{2r}}\)

    We already have
    \[
    b_{2r}(s)=\prod_{i=1}^{r}\frac{1}{1-q^{-2s-2i}}=\prod_{i=1}^r\zeta_E(2s+2i)
    \]
    \[
    c_-^r(s)=\frac{2(1+q)(1+q^{-s})}{(-q)^rq(1+q^{r-1})(1+q^r)(1-q^{-2s-2r})}=\frac{2(1+q)(1+q^{-s})}{(-q)^rq(1+q^{r-1})(1+q^r)}\zeta_E(2s+2r)
    \]
    Now we set
    \[
    a_{2r}(s)=\prod\frac{1}{1-q^{-2s-1+2i}}=\prod_{i=1}^r\zeta_E(2s-2i+1)
    \]
\end{notation}
\begin{lemma}\label{lem:M_intertwine}
    Let 
    \[
        M_{\psi_F}^\dagger(s):=\frac{b_{2r}(-s)}{a_{2r}(s)}M(s)
    \]
    where \(M(s)\) is defined in \cite{yamana2014lfunctions}*{Section 3.5} and we are following the notation in \cite{liu2022theta}*{Lemma 5.7}. We have
    \[
    \frac{b_{2r}(s)}{c_-^r(s)}M^\dagger_{\psi_F}(s)f^{(s)}=-q^{-s}\frac{b_{2r}(-s)}{c_-^r(-s)}f^{(-s)}
    \]
\end{lemma}
\begin{proof}
    The proof is almost the same as that in \cite{liu2022theta}. We use the same argument there and apply the step of proof of Theorem \ref{thm:casselman_relation_almost_spherical} to compute the coefficients.\footnote{ Note that although this formula shares the same form as Theorem \ref{thm:casselman_relation_almost_spherical} does, the coefficients are different because the unipotent subgroup in the definition of the intertwining operator here is the radical of the Sigel parabolic subgroup, so some coefficients are degenerated and equal to 1.} We have

    \begin{equation*}
        \begin{aligned}
            M_{\psi_F}^\dagger(s)f^{(s)}=&q^{-2r}\prod_{1\leq i<j\leq 2r}\frac{\zeta_E(2s+2r-i-j+1)}{\zeta_E(2s+2r-i-j+2)}\prod_{i=1}^{2r}\frac{\zeta_E(2s+2r-2i+1)}{\zeta_E(2s+2r-2i)}\cdot f^{(-s)}\\
            =&q^{-2r}\cdot\prod_{j=1}^{2r}\frac{\zeta_E(2s+2r-2j+2)}{\zeta_E(2s+2r-j+1)}\prod_{i=1}^{2r}\frac{\zeta_E(2s+2r-2i+1)}{\zeta_E(2s+2r-2i)}\cdot f^{(-s)}\\
            =&q^{-2r}\frac{\zeta_E(2s+2r)}{\zeta_E(2s-2r)}\prod_{i=1}^{2r}\frac{\zeta_E(2s+2r-2i+1)}{\zeta_E(2s+2r-i+1)}\cdot f^{(-s)}\\
            =&q^{-2r}\frac{\zeta_E(2s+2r)}{\zeta_E(2s-2r)}\prod_{i=1}^{r}\frac{\zeta_E(2s-2i+1)}{\zeta_E(2s+2i)}\cdot f^{(-s)}\\
            =&-q^{-2s}\frac{\zeta_E(2s+2r)}{\zeta_E(-2s+2r)}\frac{a_{2r}(s)}{b_{2r}(s)}\cdot f^{(-s)}
        \end{aligned}
    \end{equation*}
    By the identity
    \[
    \frac{c_-^{r}(s)}{c_-^r(-s)}=\frac{1+q^{-s}}{1+q^s}\cdot\frac{\zeta_E(2s+2r)}{\zeta_E(-2s+2r)}=q^{-s}\cdot\frac{\zeta_E(2s+2r)}{\zeta_E(-2s+2r)},
    \]
    we have
    \[
        M_{\psi_F}^\dagger(s)f^{(s)}=-q^{-s}\frac{c_-^r(s)}{c_-^r(-s)}\cdot\frac{a_{2r}(s)}{b_{2r}(s)}\cdot f^{(-s)}
    \]
    Then
    \begin{equation*}
        \begin{aligned}
         M_{\psi_F}^\dagger(s)f^{(s)}=&\frac{b_{2r}(-s)}{a_{2r}(s)}M(s)f^{(s)}\\
        =&\frac{b_{2r}(-s)}{a_{2r}(s)}(-q^{-s})\frac{c_-^r(s)}{c_-^r(-s)}\cdot\frac{a_{2r}(s)}{b_{2r}(s)}\cdot f_{r,-}^{(-s)}\\
        =&-q^{-s}\frac{b_{2r}(-s)}{b_{2r}(s)}\frac{c_-^r(s)}{c_-^r(-s)}\cdot f^{(-s)}
        \end{aligned}
    \end{equation*}
\end{proof}
\begin{theorem}\label{thm:L_function}
    \begin{enumerate}[(1)]
        \item \(L(s,\pi^\sigma_{W_r,+})=L^\sigma(s)\)
        \item 
        If \(\sigma\) contains either \(\frac{1}{2}\) or \(-\frac{1}{2}\), then we have
        \[
            L(s,\pi_{W_r,-}^\sigma)=L^\sigma_-(s)
        \]
        and
        \[
        \varepsilon(s,\pi_{W_r,-}^\sigma,\psi_F)=-q^{-s+\frac{1}{2}}.
        \]
        \end{enumerate}
\end{theorem}
\begin{proof}
    (1) is \cite{li2022chow}*{Proposition 3.7}. 
    
    The proof of (2) is similar to \cite{liu2022theta}*{Theorem 5.8}. The only difference is that, for \(r=1,\sigma=\frac{1}{2}\), we still have
    \[
    Z(\xi^\frac{1}{2},f^{(s)})=-\frac{1+q^{-s}}{q^2}\frac{1}{1-q^{-1-s}},\ L(s,\pi_{W_1,-}^\frac{1}{2})=\frac{1}{(1-q^{-\frac{1}{2}-s})}
    \]
    by Proposition \ref{prop:zeta_almost_spherical_almost_self_dual}, but \(Z(\xi^\frac{1}{2},f^{(s)})\) has zeros on the line \(\mathrm{Re}s=0\) so the latter argument in \cite{liu2022theta}*{Theorem 5.8} does not apply directly. But we can replace the use of \cite{yamana2014lfunctions}*{Lemma 6.1(2)} in the inductive steps by \cite{yamana2014lfunctions}*{Theorem 6.1}. The computation of \(\varepsilon\) is easy by Proposition \ref{prop:zeta_almost_spherical_almost_self_dual} and Lemma \ref{lem:M_intertwine}.
\end{proof}
\subsection{The Local Theta Correspondence}

We prove some results on the local theta correspondence in this subsection based on Proposition \ref{prop:zeta_almost_spherical_almost_self_dual}. The idea is the same as in \cite{liu2022theta}*{Theorem 6.2} but we do not have \cite{liu2022theta}*{Assumption 3.3} anymore. However, if we restrict ourselves to the case of tempered representations, providing Theorem \ref{thm:irreducibility_unitary_principal_series}, \ref{thm:irreducibility_unitary_principal_series_non_split}, we can adopt the method of \cite{li2022chow}*{Proposition 3.9} and apply \cite{gan2016gross}*{Theorem 4.1(v)}. 

\begin{notation}
    For \(\sigma\in(\dC/\frac{i2\pi}{\log q}\dZ)^r\), we set 
    \[
    \overset{\rightarrow}{\sigma}=(\sigma_2,\cdots,\sigma_r)
    \]
\end{notation}
\begin{theorem}\label{thm:theta_correspondence}
    \begin{enumerate}[(1)]
        \item If \(\sigma\in(i\dR/\frac{i2\pi}{\log q}\dZ)^r\), then \(\theta(\pi_{W_r,+}^\sigma,V_{r-1}^+)\cong \pi_{W_r,+}^\sigma=\rI_r^\sigma\)
        \item If \(\sigma_1\in\{\pm\frac{1}{2}\}\) and \(\overset{\rightarrow}{\sigma}\in(i\dR/\frac{i2\pi}{\log q}\dZ)^{r-1}\), then \(\theta(\pi_{W_r,-}^{\sigma},V_{r-1}^-)\cong \pi_{V_{r-1}^-}^{\overset{\rightarrow}{\sigma}}=\rI_{V_r^-}^{\overset{\rightarrow}{\sigma}}\)
    \end{enumerate}
\end{theorem}
\begin{proof}
    (1) is \cite{li2022chow}*{Proposition 3.9} and Theorem \ref{thm:irreducibility_unitary_principal_series}.

    (2) is similar to \cite{liu2022theta}*{Theorem 6.2}. The only difference is that, now \(\pi_{W_r,-}^\sigma\) is tempered and we can replace the use of \cite{liu2022theta}*{Assumption 3.3} by \cite{li2022chow}*{Proposition 3.9} and \cite{gan2016gross}*{Theorem 4.1(v)}. The last step is to apply Theorem \ref{thm:irreducibility_unitary_principal_series_non_split}.
\end{proof}
\begin{remark}\label{rem:theta_correspondence}
    There are three points to be noted here:
    \begin{enumerate}[(i)]
        \item
        Though we have proved Theorem \ref{thm:spherical_L_is_spherical_K}, we still need the assumption that representations are tempered here because we need the fact that the big theta lifting is irreducible.
        \item 
        As a consequence of Theorem \ref{thm:spherical_L_is_spherical_K} and \ref{thm:spherical_L_is_spherical_K_non_split}, we know that if an irreducible admissible representation of \(\rU(V)\) is regularly spherical with respect to a unimodular lattice and it is tempered, then it is a unitary principal series representation and the theta liftings are given as above by the Howe duality \cite{gan2015proof}.
        \item
        Our proof of this theorem relies on computation of the doubling zeta integrals, but the matrix coefficient \(\xi^\sigma\) is chosen with respect to the eigenvector of an almost spherical spherical representation \(\pi^\sigma_{W_r,-}\) with respect to \(K_r\) there instead of the more natural \(L_r\).
    \end{enumerate}
\end{remark}

\subsection{More on the Hecke algebras}
In this subsection, we want to establish an isomorphism of an analogue of the almost unramified Hecke algebra in \cite{liu2022theta}*{Definition 2.8}. After this is completed, we are allowed to port most results there to our setting and in particular, combining with Theorem \ref{thm:satake_isomorphism_quasi_split}, we can construct \(\theta^\tR\) in \cite{li2021chow}*{Definition 6.8} at places \(v\in\tS'_\pi\). The process is somehow redundant but some new phenomenon has shown up here as the `special fiber' of \(L_r\) is not connected. We will keep most notations from Section \ref{sec:quasi-split}, but for convenience, we will change some notations like \(\cB_i\) which were used before.

\begin{remark}
    \begin{itemize}
        \item 
        Recall that \(V_r^+\) has a basis \(v_1,\cdots,v_r,v_{r+1},\cdots,v_{2r}\) such that the Hermitian form is given by \((v_i,v_{r+j})=\varpi_E^{-1}\delta_{ij}\) for \(1\le i,j\le r\). The Iwahori subgroup \(I_r\) were chosen to be looks like
        \[
            \begin{pmatrix}
                \cO_E^\times &\fp_E & \cdots & \fp_E & \cO_E & \cdots & \cO_E \\ 
                \cO_E &\cO_E^\times & \cdots & \fp_E & \cO_E & \cdots & \cO_E \\
                \vdots &\vdots & \ddots & \vdots & \vdots & \ddots & \vdots \\
                \cO_E &\cO_E & \cdots & \cO_E^\times & \cO_E & \cdots & \cO_E \\
                \fp_E &\fp_E & \cdots & \fp_E & \cO_E^\times & \cdots & \cO_E \\
                \vdots &\vdots & \ddots & \vdots & \vdots & \ddots & \vdots \\
                \fp_E &\fp_E & \cdots & \fp_E & \fp_E & \cdots & \cO_E^\times
            \end{pmatrix}
        \]
        with respect to this basis. But as we are interested in the unimodular lattice spanned by \(v_i,1\leq i\leq r\) and \(\varpi_Ev_{r+i},1\leq i\leq r\), the matrix form of \(I_r\) is given by a conjugation by a diagonal element \(\mathrm{diag}(1,\cdots,1,\varpi_E,\cdots,\varpi_E)\) in \(\GL_{2r}(E)\), so it looks like
        \[
            \begin{pmatrix}
                \cO_E^\times &\fp_E & \cdots & \fp_E & \fp_E & \cdots & \fp_E \\ 
                \cO_E &\cO_E^\times & \cdots & \fp_E & \fp_E & \cdots & \fp_E \\
                \vdots &\vdots & \ddots & \vdots & \vdots & \ddots & \vdots \\
                \cO_E &\cO_E & \cdots & \cO_E^\times & \fp_E & \cdots & \fp_E \\
                \cO_E &\cO_E & \cdots & \cO_E & \cO_E^\times & \cdots & \cO_E \\
                \vdots &\vdots & \ddots & \vdots & \vdots & \ddots & \vdots \\
                \cO_E &\cO_E & \cdots & \cO_E & \fp_E & \cdots & \cO_E^\times
            \end{pmatrix}
        \]
        in the new coordinates.
        \item
        Consider \(\Omega(w_r)\in L_r\) which were chosen as the representative of a generator of the Weyl group in Section \ref{sec:quasi-split}, we know that \(\Omega(w_r)v_r=\varpi_E v_{2r}\) and \(\Omega(w_r)v_{2r}=\varpi_E^{-1} v_r\), it has determinant \(-1\) (in this case the order of the roots are reversed).
        
        The interesting fact is that if we consider the group morphism \(L_r\rightarrow \rO^+(2r,\dF_q)\), then the image \(\Omega(w_r)\) normalizes \(\SO^+(2r,\dF_q)\) and \(I_r\) but not \(I_r^0\).
    \end{itemize}
\end{remark}
\begin{definition}
    We define the neutral part \(L_r^o\) of \(L_r\) to be the inverse image of \(\SO^+(2r,\dF_q)\) under the morphism \(L_r\rightarrow \rO^+(2r,\dF_q)\). Equivalently, the elements in \(L_r^o\) have determinant \(1 \modulo \fp_E\).

    Then \(L_r=L_r^o\sqcup \Omega(w_r)L^o_r\)
\end{definition}

We denote the \(D_r\)-type finite Weyl group by \(\cW_r\), which is an index \(2\) subgroup of the Weyl group \(\fW_r\). We have the Iwahori decomposition
    \[
    L_r^o=\bigsqcup_{w\in \cW_r}I_r\Omega(w)I_r
    \]
    and
    \[
    L_r=L_r^o\sqcup \Omega(w_r)L_r^o=\bigsqcup_{w\in\fW_r}I_r\Omega(w)I_r
    \]
    where \(w_i\) are the simple reflections in \(\fW_r\) defined in Section \ref{sec:quasi-split}. 

Consider the Iwahori Hecke algebra \(\dC[I_r\backslash L_r/I_r]\) where \(I_r\) has measure \(1\), by simple computation, we have the following presentation:
\begin{proposition}
    The Iwahori Hecke algebra \(\dC[I_r\backslash L_r/I_r]\) is generated by the characteristic functions \(1_{I_r\Omega(w_r)I_r}\) and \(1_{I_rw_{i,i+1}I_r}\) for \(1\le i\le r-1\) with the following relations:
    \begin{enumerate}[(1)]
        \item \((1_{I_r\Omega(w_r)I_r})^2=1\)
        \item \((1_{I_rw_{i,i+1}I_r})^2=(q-1)1_{I_rw_{i,i+1}I_r}+q\) for \(1\le i\le r-1\)
        \item Braid relations among these generators as in the Coxeter group \(\fW_r\).
    \end{enumerate}
\end{proposition}
\begin{remark}
    It might be confusing that in \cite{badea2020hecke}*{Chapter 5}, the generators have different length compared to ours. The reason is that our Iwahori subgroup \(I_r\) defined in Section \ref{sec:quasi-split} are defined for the standard Borel subgroup in the skew-Hermitian spaces, when the field extension is ramified, the way Iwahori embedding into \(L_r\) corresponds to an opposite Borel subgroup. This is also why we are considering \(\Omega(w_r)\) in this subsection, because in the standard notation, it is an affine generator of the extended affine Weyl group. If we unfold the definitions, we can check that they agree up to some conjugation.
\end{remark}
This motivates us to define the character \(\kappa_r^\pm:\dC[I_r\backslash L_r/I_r]\rightarrow\dC\) as follows:
\begin{definition}
    For \(\epsilon\in\{+,-\}\), we define a character \(\kappa_r^\epsilon:\dC[I_r\backslash L_r/I_r]\rightarrow\dC\) by
    \begin{equation*}
    1_{I_r\Omega(w_r)I_r}\mapsto \epsilon 1,\qquad 1_{I_rw_{i,i+1}I_r}\mapsto q
    \end{equation*}
\end{definition}
Below is an analogue of \cite{liu2022theta}*{Lemma 2.2} but the proof is more straightforward.
\begin{lemma}
    The eigenspace \(\dC[I_r\backslash L_r/I_r][\kappa_r^\epsilon]\) is spanned over \(\dC\) by the function
    \[
    \fe_r^\epsilon:=1_{L_r^o}+\epsilon 1_{\Omega(w_r)L_r^o}.
    \]
    and in particular, \(\fe_r^+=1_{L_r}\).
\end{lemma}
Thus we can choose a unique element \(e_r^\epsilon\in \dC[I_r\backslash L_r/I_r][\kappa_r^\epsilon]\) so that \((e_r^\epsilon)^2=e_r^\epsilon\).

Apply the same argument as in \cite{liu2022theta}*{Section 2}, we have the following isomorphism of Hecke algebras:
\begin{lemma}\label{lemma:hecke_algebra_isomorphism}
    Let \(\cH_r:=\dC[I_r\backslash G_r/I_r]\) be the Iwahori Hecke algebra of \(G_r\) where \(I_r\) has measure \(1\) and \(Z(\cH_r)\) be its center. We have isomorphisms of Hecke algebras given as follows:
    \begin{equation*}
        \begin{aligned}
            \beta_r^\epsilon:\quad & Z(\cH_r)\xrightarrow{\sim} e_r^\epsilon Z(\cH_r)=:\cH_r^\epsilon &f\mapsto e_r^\epsilon f\\
            and \quad &\cH_r^+\xrightarrow{\sim} \dC[L_r\backslash G_r/L_r] &f\mapsto [L_r:I_r]f
        \end{aligned}
    \end{equation*}
\end{lemma}

As mentioned in Remark \ref{rem:theta_correspondence}, our proof of Theorem \ref{thm:theta_correspondence} is somehow an elementary trick: we describe the theta correspondence of another representation and then prove that it is isomorphic to the one we want. It is expected that one can prove a more intrinsic version of Theorem \ref{thm:theta_correspondence} following the proposition below, which is an analogue of \cite{liu2022theta}*{Lemma 3.2}.
\begin{proposition}
    Via the Schr\"odinger model of the local
    Weil representation \(\omega_{W_r,V_{s}^-}\), the Hecke algebra \(\dC[I_r\backslash L_r/I_r]\) acts on \(1_{(\Lambda_{V_s^-}')^r}\) by the character \(\kappa_r^-\).
\end{proposition}
\begin{proof}
    For simplicity, we denote \(\Lambda_{V_s^-}'\) by \(\Lambda\) and omit the notation of Weil representation in the proof as there is no ambiguity.

    By the same argument from \cite{liu2022theta}*{Lemma 3.2}, it suffices to check that \(1_{I_1\Omega(w_1)I_1}.1_{\Lambda}=-1_{\Lambda}\).

    Note that for \(a,b\in I_1\), we have
    \begin{equation*}
        \begin{aligned}
            a\Omega(w_1)b.1_{\Lambda}=&a\omega(w_1)
            \begin{pmatrix}
            -\varpi_E&0\\0&\varpi_E^{-1}
            \end{pmatrix}.
            1_{\Lambda}\\=&a\omega(w_1).(q^{-s}1_{\Lambda^\vee})\\
        \end{aligned}
    \end{equation*}
    as \(\Lambda=\Lambda^\sharp=\varpi_E\Lambda^\vee\) and \([\Lambda^\vee:\Lambda]=q^{2s}\)
    then by \cite{he2023proof}*{Lemma 4.10}, we have
    \begin{equation*}
        \begin{aligned}
            a\Omega(w_1)b.1_{\Lambda}=&q^{-s}a\omega(w_1).1_{\Lambda^\vee}\\
            =&-q^{-s}a.q^{s}1_{\Lambda}\\
            =&-1_{\Lambda}
        \end{aligned}
    \end{equation*}
    The result follows from the measure computation of \(I_1\Omega(w_1)I_1=\Omega(w_1)I_1\).
\end{proof}

\section{Semi-global Integral Models of Unitary Shimura Varieties}\label{sec:semi-global}

In this section we examine the geometric properties of the semi-global integral models of unitary Shimura varieties. We will first introduce the notion of integral models of Shimura varieties, and then we will discuss the construction of the semi-global integral models of unitary Shimura varieties. We will also discuss the properties of the semi-global integral models, and the relation between the semi-global integral models and the local models of the unitary Shimura varieties. We will follow the construction in \cite{liu2022beilinson}*{Section 5}.

\begin{notation}
    We denote by \(T_0\) the torus over \(\dQ\) such that for every commutative \(\dQ\)-algebra \(R\), we have \(T_0(R)=\{a\in\dE\otimes_\dQ R|\rN_{\dE/\dF}(a)\in R^\times\}\).
\end{notation}

Recall that \(\dE/\dF\) is a CM extension of number fields, and we fix an embedding \(\biota:\dE\hookrightarrow \dC\). Then we fix a CM type \(\Phi\) of \(\dE\) containing \(\biota\). We denote by \(\dE'\) the subfield of \(\dC\) generated by the reflex field of \(\Phi\) and \(\dE\). Also we need to require that \(\Phi\) is admissible in the sense of the following Notation \ref{notation:CM_type_admissible}. Then we fix a skew hermitian space \(W\) over \(\dE\) of dimension \(1\) with respect to the nontrivial Galois involution of \(\dE/\dF\), whose group of rational similitude is canonically \(T_0\). Then for an open compact subgroup \(L_0\) of \(T_0(\dA^\infty)\), we have the PEL type Shimura variety \(Y\) of CM abelian varieties with CM type \(\Phi\) and level \(L_0\), which is a smooth projective scheme over \(\dE'\) of dimension 0.

For a fixed place \(u\) over \(p\), let the local field \(K\) be the subfield of \(\bar{\dQ}_p\) generated by \(\dE_u\) and the reflex field of \(\Phi\). Our local model will be defined over this field and it obviously depends on \(u\). Let \(\breve{K}\) be the completion of the maximal unramified extension of \(K\).

We need to choose a place \(u\) in \(\tV_\dE\) \cite{li2022chow}*{Notation 4.19}. When we write \(v\), we always mean a place of \(\dF\) lying above the same rational prime \(p\) as \(u\) and the underlying place of \(u\) is denoted by \(\underline{u}\) so both \(v\) and \(\underline{u}\) are in \(\tV_\dF^{(p)}\).

\begin{notation}\label{notation:CM_type_admissible}
We further require that $\Phi$ is \emph{admissible} in the following sense: if $\Phi_v\subseteq\Phi$ denotes the subset inducing the place $v$ for every $v\in\tV_\dF^{(p)}$, then it satisfies
\begin{enumerate}
  \item when $v\in\tV_\dF^{(p)}\cap\tV_\dF^\spl$, $\Phi_v$ induces the same place of $\dE$ above $v$;

  \item when $v\in\tV_\dF^{(p)}\cap\tV_\dF^\text{int}$, $\Phi_v$ is the pullback of a CM type of the maximal subfield of $\dE_v$ unramified over $\dQ_p$;

  \item when $v\in\tV_\dF^{(p)}\cap\tV_\dF^\ram$, the subfield of $\ol\dQ_p$ generated by $\dE_u$ and the reflex field of $\Phi_v$ is unramified over $\dE_u$.
\end{enumerate}
\end{notation}
\begin{remark}
    The admissible conditions are the same as that in \cite{li2022chow}*{Remark 4.19} and (3) is even more important in our setting because we need to ensure the semi-global Kr\"amer model of Shimura varieties is regular. See \cite{luo2025unitary} for further discussion.
\end{remark}

 If there is no confusion, we will omit the subscript \(u\) and \(v\), then \(E:=\dE_u\) is an extension of a \(p\)-adic field \(F:=\dF_{\underline{u}}\).

\subsection{Recollection on Unitary Shimura Varieties}
We recall the definition of unitary Shimura varieties from \cite{liu2022beilinson}*{Section 3.2}.
\begin{notation}
    \begin{enumerate}
    \item\label{notion:unitary_shimura_varieties}
    Let \(V\) be a standard \emph{indefinite} (See \cite{liu2022beilinson}*{Definition 3.2.1} for a definition) hermitian space over \(\dF\) of \(\rank\) \(N\). For every neat open compact subgroup \(L\subset \Res_{\dF/\dQ}\rU(V)(\dA_\dQ^\infty)\), we have a scheme \(\Sh(V,L)\) over \(\dF\). In our application, we will consider the \(\bu\)-nearby space \({}^\bu V\) (so it is defined over \(\dF\)) and denote the Shimura variety with level \(L\) by \(X_L\).
    \item \label{notion:unitary_shimura_sets}
    Let \(V\) be a standard \emph{definite} hermitian space over \(\dF\) of \(\rank\) \(N\). For every neat open compact subgroup \(L\subset \Res_{\dF/\dQ}\rU(V)(\dA_\dQ)\), we have a set 
    \[
    \Sh(V,L):=\rU(V)(\dF)\backslash \rU(V)(\dA^\infty_\dF)/L
    \]
    \end{enumerate}
\end{notation}

Unitary Shimura varieties are not of PEL type, but they are closely related to certain RSZ Shimura varieties which are the moduli spaces of abelian schemes. The connection between them is explained in \cites{rapoport2020arithmetic,rapoport2021shimura} and also \cite{liu2021fourier}*{Appendix C.3,C.4}. Our goal in this section is to construct the semi-global integral models of RSZ Shimura varieties in the Kr\"amer model of \cite{rapoport2021shimura}*{Section 6} and establish certain basic correspondence analogues to \cite{liu2022beilinson}*{Section 5}.

\subsection{Semi-global Integral Models of Unitary Shimura Varieties}
Because what we are interested in is the semi-global integral model at the place where the local root number is \(-1\), we are indeed considering the nearby space. 

We will only consider a \emph{definite} Hermitian space when we consider a Hermitian space (also the corresponding level structures) and thus we will omit the terminology in this subsection. Let \(L\subset H(\dA_\dF^\infty) \) be an open compact subgroup of the form \(L_\tR L^\tR\) in Notation \ref{notation:special_hecke_algebra}.

We need to formulate a finite \'etale model \(\cY\) of \(Y\). For an \(\cO_K\)-scheme \(S\), \(\cY(S)\) consists of the equivalence classes of a unitary \(\cO_\dE\)-abelian scheme over \(S\) of signature type \(\Phi\), a \(p\)-polarization, and a level structure.

\begin{definition}[Integral Model of auxiliary moduli]
    For an \(\cO_K\)-scheme \(S\), let \(\cY(S)\) be the set of equivalence classes of quadruples \((A_0,\iota_{A_0},\lambda_{A_0},\eta_{A_0}^p)\) where
\begin{itemize}
    \item \((A_0,\iota_{A_0},\lambda_{A_0})\) is a unitary \(\cO_\dE\)-abelian scheme over \(S\) of signature type \(\Phi\) such that \(\lambda_{A_0}\) is \(p\)-principal, i.e., a prime-to-\(p\) quasi-isogeny.
    \item \(\eta_{A_0}^p\) is an \(L_0^p\)-level structure \cite{liu2022beilinson}*{Definition 3.5.4}
\end{itemize}
\begin{remark}
    This level structure is imposed to give a scheme (for sufficiently small \(L\)) rather than a DM stack.
\end{remark}

Then we define the local integral model of the Shimura variety \(X\) for \(u\in \tV_\dE^{\ram}\). For our \(L=L_\tR L^\tR\), consider the projective system of Shimura varieties \(\{X_{\tilde{L}}\}\) indexed by \(\{\tilde{L}\subset L\}\) with \(\tilde{L}_v=L_v\) for \(v\in \tV_{\dF}^{(p)}\backslash \tV_\dF^{\spl}\) and we have semistable schemes \(\cX_{\tilde{L}}\) over \(\cO_K\) as our semi-global integral models so that 
\[
\cX_{\tilde{L}}\times_{\Spec \cO_K}\Spec K \cong (X_{\tilde{L}} \otimes_{\dE}Y)\otimes_{\dE'}K
\]
Here we only consider the Kr\"amer condition for convenience. In fact, this is the model \(\mathcal{M}\) defined in \cite{shi2023special}*{Section 8.1 (8.6)}.
\end{definition}
\begin{definition}[Semi-global Integral Model of Shimura variety]
\(\cX_{\tilde{L}}(S)\) is the set of equivalence classes of tuples
\[
(A_0,\iota_{A_0},\lambda_{A_0},\eta_{A_0}^p;A,\iota_{A},\lambda_{A},\eta_{A}^p,\{\eta_{A,v}\}_{v\in (\tV_\dF^{\spl}\cap\tV_{\dF}^{(p)})};\cF_A)
\]

such that \begin{itemize}
    \item \((A_0,\iota_{A_0},\lambda_{A_0},\eta_{A_0}^p)\) is an element in \(\cY(S)\)
    \item \((A,\iota_{A},\lambda_{A})\) is a unitary \(\cO_\dE\)-abelian scheme of signature type \(n\Phi-\biota+\biota^c\) over \(S\), such that 
    \begin{enumerate}[(i)]
        \item for \(v\) unramified, \(\lambda_A[v^\infty]\) is an isogeny whose kernel has order \(q_v^{1-\epsilon_v}\)
        
        \item for ramified \(v\) which is not the underlying place of \(u\), the triple 
        \[(A[v^\infty],\iota_{A}[v^\infty],\lambda_{A}[v^\infty])\otimes_{\cO_K}\cO_{\breve{K}}\]
        satisfies the sign conditions and the Eisenstein conditions \cite{shi2023speciala}*{4B} and \cite{he2023proof}*{Section 10.2 (H3) (H4)}\footnote{In \cite{he2023proof}*{Section 10.2}, they use the notation \(E\) for reflex field. Their assumption that \(v\) is unramified over \(p\) can be removed by \cite{luo2025unitary}.}
        \item \(\cF_A \subset \Lie A\) is an \(\cO_\dE\)-stable \(\cO_S\) -module local direct summand of rank \(n-1\) satisfying the Kr\"amer condition: \(\cO_\dE\) acts on \(\cF_A\) by the structure map \(s :\cO_\dE\rightarrow\cO_{K}\rightarrow\cO_S\) and acts on \(\Lie A/\cF_A\) by the complex conjugate of the structure map. 
        
        This is the same as given a filtration \((\cF_A)_v\subset (\Lie A)_v\) by \cite{rapoport2021shimura}*{Lemma 6.11}.
        \item \(\eta_A^p\) is an \(\tilde{L}^p\)-level structure \cite{liu2022beilinson}*{Definition 5.2.1}
        \item for every \(v\in \tV_\dF^{(p)}\cap \tV_\dF^{\spl}\), \(\eta_{A,v}\) is an \(\tilde{L}_v\)-level structure.
    \end{enumerate}
\end{itemize}
\end{definition}

From \cite{liu2022beilinson}*{(3.2)} we have a Hodge exact sequence
\[
    0\rightarrow\omega_{A^\vee/S}\rightarrow \rH_1^{\dr}(A/S)\rightarrow\Lie_{A/S}\rightarrow0
\]
Then following \cite{liu2022beilinson}*{Notation 3.4.10}, we have
\begin{itemize}
    \item \(\rH_1^{\dr}(A/S)\simeq\cD(A)/p\cD(A)\)
    \item \(\tV\cD(A)_{\sigma\tau}/p\cD(A)_\tau\simeq\omega_{A^\vee,\tau}\).
    \item \(\cD(A)_\tau/\tV\cD(A)_{\sigma\tau}\simeq\Lie_{A,\tau}\)
    \item An obvious short exact sequence
        \[
            0\rightarrow\tV\cD(A)_{\sigma\tau}/p\cD(A)_\tau\rightarrow \cD(A)_\tau/p\cD(A)_\tau\rightarrow\cD(A)_\tau/\tV\cD(A)_{\sigma\tau}\rightarrow0
        \]
\end{itemize}
\begin{remark}
    We also mention some important results about Rapoport--Zink spaces (or RZ spaces in short) that we need here. 
    \begin{itemize}
        \item 
        The first one is the comparison of relative RZ spaces and the absolute RZ spaces, it can be found in \cite{luo2025unitary}*{Theorem 3.17 and Remark 3.19}.
        \item
        The other is the uniformization map to the Shimura variety, which can also be found in \cite{luo2025unitary}*{Theorem 7.12} with a slight modification on the moduli description of the Kr\"amer condition.
    \end{itemize}
\end{remark}

\subsection{Geometry of the Special Fiber}
In this subsection, we omit the index \(\tilde{L}\) for convenience, and we will study the geometry of the special fiber of \(\cX\) at the place \(u\in \tV_\dE^{\ram}\).

All schemes in this subsection are considered over \(\overline{\dF}_p\) and we will use the notation \(\bX\) for \(\cX\otimes_{\cO_K}\overline{\dF}_p\) and similarly for other schemes with the level structure described in the last subsection.

We further define the balloon stratum and the ground stratum, which reflect the fact that \(\cX\) is the blow-up of the Pappas model \cite{shi2023special}*{Appendix A}.

When there is no confusion, we will set \(\Pi=\iota_{A,u}(\varpi_v)\).
\begin{definition}[Balloon Strata]\label{def:balloon_strata}
    Let \(\bX^\circ\) be the locus of \(\bX\) where \(\Pi\) acts on \(\Lie A\) by 0 \footnote{it also acts on \(\omega_{A^\vee/S}\) by zero because there exists a \(\cO_\dE\)-equivariant nondegenerate pairing, so this is compatible with \cite{zachos2024semistable}*{Section 6.1}}. Then \(\bX^\circ\) should be understood as the exceptional divisor.
\end{definition}

\begin{definition}[Ground Strata]\label{def:ground_strata}
    Let \(\bX^\bullet\) be the (Zariski closure of the) strict transform of the complement of \(\bX^\circ\) in \(\bX\). \(\bX^\bullet\) is called the ground stratum.
    
\end{definition}
\begin{remark}
    This stratum also has the following description, which comes from a similar description of its Rapoport--Zink space and the corresponding local model (see \cite{kramer2003local}*{Theorem 4.5}): let \(k\) be the field \(\overline{\dF}_p\), we can define a non-degenerate anti-symmetric pairing \(\langle,\rangle\) on \(\rH^{\dr}_1(A_k/S_k)\) via the polarization, then define another symmetric form \(\{,\}\) on \(\Pi\rH^{\dr}_1(A_k/S_k)\) by \(\{\Pi x,\Pi y\}:=\langle \Pi x,y\rangle\). This form is well-defined because \(\Pi^2\) acts as zero as we are on the special fiber. Then \(\bX^\bullet\) is the locus where \(\cF_A^\perp\subset \omega_{A^\vee_k}
        \) is isotropic with respect to the form \(\{,\}\), here \(\cF_A^\perp\) is the orthogonal complement of \(\cF_A\) in \(\rH^{\dr}_1(A_k/S_k)\) with respect to the pairing \(\langle,\rangle\).
\end{remark}
\begin{definition}[Link Strata]\label{def:link_strata}
    Let \(\bX^\dagger\) be the intersection of \(\bX^\circ\) and \(\bX^\bullet\). \(\bX^\dagger\) is called the link strata. 
\end{definition}

We construct \(\bS^\circ\) as the moduli space similar to \(\bX^\circ\) but with signature type \(n\Phi\) rather than \(n\Phi-\biota+\biota^c\). 

\begin{remark}
    Note that in this case we do not need to choose a filtration \(\cF_A\) because by \cite{rapoport2021shimura}*{Lemma 6.12}, there exists a unique choice.

    One can also see this because in a discrete set, the Lie algebra is zero and thus the Kr\"amer condition is vacuous.
\end{remark}
\begin{definition}[Source of balloon strata]
    We define \(\bS^\circ(S)\) to be the set of equivalence of tuples \((A_0,\iota_{A_0},\lambda_{A_0},\eta_{A_0}^p;A^\circ,\iota_{A^\circ},\lambda_{A^\circ},\eta_{A^\circ}^p,\{\eta_{A^\circ,v}\}_{v\in \tV_\dF^{\spl,p}})\) where
    \begin{itemize}
        \item \((A_0,\iota_{A_0},\lambda_{A_0},\eta_{A_0}^p)\) is an element in \(\cY(S)\)
        \item \((A^\circ,\iota_{A^\circ},\lambda_{A^\circ})\) is a unitary \(\cO_\dE\)-abelian scheme of signature type \(n\Phi\) over \(S\), such that \(\lambda_{A^\circ}\) is \(p\)-principal, i.e., a prime-to-\(p\) quasi-isogeny.
        \item \(\eta^p_{A^\circ}\) is given in \cite{liu2022beilinson}*{Definition 5.3.1}
    \end{itemize}
\end{definition}

\begin{definition}[Basic correspondence of balloon strata]
    We define \(\bB^\circ(S)\) to be the set of equivalence of tuples 
    \[
    (A_0,\iota_{A_0},\lambda_{A_0},\eta_{A_0}^p;
    A,\iota_{A},\lambda_{A},\eta_{A}^p,\{\eta_{A,v}\}_{v\in \tV_\dF^{\spl,p}},\cF_{A};
    A^\circ,\iota_{A^\circ},\lambda_{A^\circ},\eta_{A^\circ}^p,\{\eta_{A^\circ,v}\}_{v\in \tV_\dF^{\spl,p}};\beta)\]
    where
    \begin{itemize}
        \item 
        \((A_0,\iota_{A_0},\lambda_{A_0},\eta_{A_0}^p;
        A,\iota_{A},\lambda_{A},\eta_{A}^p,\{\eta_{A,v}\}_{v\in \tV_\dF^{\spl,p}},\cF_{A})\) is an element of \(\bX^\circ(S)\),
        \item
        \((A_0,\iota_{A_0},\lambda_{A_0},\eta_{A_0}^p;
        A^\circ,\iota_{A^\circ},\lambda_{A^\circ},\eta_{A^\circ}^p,\{\eta_{A^\circ,v}\}_{v\in \tV_\dF^{\spl,p}})\) is an element of \(\bS^\circ(S)\)
        \item
        and \(\beta:A\rightarrow A^\circ\) is an \(\cO_\dE\)-linear quasi-\(p\)-isogeny such that
        \begin{enumerate}[(a)]
            \item \(\ker\beta[p^\infty]\) is contained in \(A[\fp]\) where \(\fp\) is the maximal ideal of \(\cO_E\) above \(p\) corresponding to the place \(u\);
            \item we have \(\lambda_A=\beta^\vee\circ\lambda_{A^\circ}\circ\beta\); and
            \item the (away from \(p\)) \(L^{p\circ}\)-orbit of maps \(v\mapsto\beta_*\circ\eta^p_A(v)\) for v in \(V\otimes_\dQ \dA^{\infty,p}\) coincides with \(\eta_A^{p\circ}\).
        \end{enumerate}
    \end{itemize}
    The equivalence relation and the action of morphisms are defined similarly as in \cite{liu2022beilinson}*{Definition 4.3.3}.
\end{definition}

We have the basic correspondence of the balloon strata similar to \cite{liu2022beilinson}*{Definition 5.3.3}:
\begin{center}
    \begin{tikzcd}
        \bS^\circ& \bB^\circ\ar[l,"\pi^\circ"']\ar[r,"\iota^\circ"]&\bX^\circ 
    \end{tikzcd}
\end{center}
But for our application, we only need the following properties of the basic correspondence.
\begin{proposition}
    The morphism \(\iota^\circ\) is an isomorphism and the fiber of each \(s\) in \(\bB^\circ\) is isomorphic to \(\dP^{n-1}\), where \(n=2r\).
\end{proposition}
\begin{proof}
    The proof is much simpler than that in \cite{liu2022beilinson}*{Proposition 5.3.4}. The first part is due to \cite{rapoport2021shimura}*{Lemma 6.12} and the proof of \cite{liu2022beilinson}*{Proposition 5.3.4} (just replace the DL variety to a \(\dP^{2r-1}\) or follow \cite{he2023proof}*{Section 3}). 
    
    The second part follows from \cite{zachos2024semistable}*{Theorem 6.1}
\end{proof}

Denote by \(L^\star_{\underline{u}}\subset {}^uH(F)\) the stabilizer of a unimodular lattice in the nearby space \({}^uV\) at the place \(\underline{u}\), which may not be a speciam maximal compact subgroup as we are in the ramified case. We still have the following uniformization result of the source of the balloon strata.
\begin{proposition}[Uniformization of the source of the balloon strata]
        We have an isomorphism
        \begin{equation}
            \bS^\circ(\overline{\dF}_p)\overset{\sim}{\longrightarrow} \Sh(V^\circ,\tilde{L}^{\underline{u}}L^\star_{\underline{u}})\times (\cY\otimes_{\cO_K}\overline{\dF}_p)
        \end{equation}
\begin{proof}
    This is the same as \cite{liu2022beilinson}*{Proposition 4.4.5 and 5.3.6}.

    Another way to see this is to consider the local model of \(\bX^\circ\), which is the blow-up along a point, then the local model of \(\bS^\circ\) is just this point.
\end{proof}
\end{proposition}

For our application, we do not need the full basic correspondence of the ground strata so we omit it for now.

\subsection{Cohomology Vanishing Lemma}

\begin{definition}
  Similar to \cite{li2021chow}*{Definition 6.8}, we have a unique character \(\chi_\pi^\tR:\dT^\tR_{\dQ^{\ac}}\to\dQ^{\ac}\), in our setup, this one is well-defined at places in \(\tS\cap \tV_\dF^{\ram}\) by the argument of \cite{liu2022theta} and Lemma~\ref{lemma:hecke_algebra_isomorphism}. Then we can define \(\fm_\pi^\tR:=\ker \chi_\pi^\tR\) which is a maximal ideal of \(\dT^\tR_{\dQ^{\ac}}\).\label{def:maximal_ideal}
\end{definition}
By the basic correspondence in last section, we have the following properties which is similar to \cite{li2021chow}*{Page 868} and \cite{zachos2024semistable}*{Section 6}:
\begin{proposition}
    \begin{enumerate}[(1)]
        \item \(\bX^\circ_{\tilde{L}}\) is a \(\dP^{2r-1}\)-fibration over the Shimura set 
        \[
        {}^uH(\dF)\backslash ^uH(\dA_{\dF}^\infty)/\tilde{L}^{\underline{u}}L^\star_{\underline{u}} \times (\cY\otimes_{\cO_K}\overline{\dF}_p)
        \]
        \item \(\bX^\bullet_{\tilde{L}}\) is proper and smooth over \(\overline{\dF}_p\) of dimension \(2r-1\);
        \item The intersection \(\bX^\dagger_{\tilde{L}}\) is a relative Fermat hypersurface of degree \(2\) in \(\bX^\circ_{\tilde{L}}\).
    \end{enumerate}
\end{proposition}

\begin{remark}
    These can also be obtained from the \'etale local structure of \(\bX\) which is the same as that of the local model by \cite{zachos2024semistable}*{Theorem 6.1}.
\end{remark}

Now let \((\pi,\cV_\pi)\) be as in Assumption \ref{asp:main}, then following \cite{li2021chow}*{Lemma 9.2}, we need to show the following proposition.
\begin{proposition}\label{prop:cohomology_vanishing}
    \begin{enumerate}[(1)]
        \item \(\rH^i(\bX^\circ_L,\overline{\dQ}_\ell)_\fm=0\ \mathrm{for}\ i\le 2r-2\),
        \item \(\rH^i(\bX^\bullet_L,\overline{\dQ}_\ell)_\fm=0\ \mathrm{for}\ i\le 2r-2\),
        \item \(\rH^i(\bX^\dagger_L,\overline{\dQ}_\ell)_\fm=0 \ \mathrm{for}\ i\le 2r-3\),
    \end{enumerate}
    where \(\fm:=\fm_\pi^\tR \cap \dS^\tR_{\dQ^{\ac}}\) is well-defined by Definition~\ref{def:maximal_ideal}.
    \begin{proof}
        For (1), because \(\bX^\circ_L\) is a \(\dP^{2r-1}\)-fibration over a discrete Shimura set, we know that \(\rH^i(\bX^\circ_L,\overline{\dQ}_\ell)=0\) for odd \(i\). When \(i\) is even, \(\rH^i(\bX_L^\circ,\overline{\dQ}_\ell)_\fm\) is a direct sum of \(\tilde{L}^{\underline{u}}\times L^\star_{\underline{u}}\)-invariants of \(\pi'\) for finitely many cuspidal automorphic representations \(\pi'\) with \(\pi'_\infty\) is trivial and \(\pi'_v\simeq\pi_v\) for all but finitely many \(v\in\rV^{\mathrm{spl}}_\dF\). We thus have \(\BC(\pi)\simeq\BC(\pi')\). For a local place with root number -1, the local component is a regularly almost spherical representation which can not be the base change of a representation with \(L_{\underline{u}}^\star\)-invariants unless the Satake parameter contains \(\{-q^{\pm\frac{1}{2}}\}\), so we have \((\pi^{\chi_{\underline{u}}})^{L^\star_{\underline{u}}}=0\) (following the notations in \cite{li2021chow}*{Below Table 1, Page 869}). So we have
        \[
            \rH^i(\bX^\circ_L,\overline{\dQ}_\ell)_\fm=0.
            \]
            For (2), this follows the same argument as \cite{li2021chow} as \(\bX^\dagger_L\subset \bX^\circ_L\) is a relative quadratic surface.
            
            For (3), we compute the weight spectral sequence, the proof is identical to \cite{li2021chow}*{Lemma 9.2} except now we need to apply Theorem \ref{thm:irreducibility_unitary_principal_series_non_split}.
        \end{proof}
    \end{proposition}
    \begin{remark}\label{rem:ideal_redundant}
        Though we used Definition \ref{def:maximal_ideal} to define \(\fm_\pi^\tR\), the definition of \(\fm\) here does not depend on the ramified places because \(\tS_{\dQ^{\ac}}^\tR\) supports only at split places. Thus we did not actually use Definition \ref{def:maximal_ideal} in the proof of Proposition \ref{prop:cohomology_vanishing}.
    \end{remark}


\section{Local Indices}\label{sec:local-indices}

    We adapt the notations from last section. Fix a place \(u\in \tV^{\ram}_\dE\) above a rational prime \(p\). We also assume that \(\tV_\dF^{(p)}\cap\tR\subset \tV_\dF^{\spl}\) and \(\underline{u}\notin\tV_{\dF}^\heartsuit\), so that the reflex field is unramified over \(\dE_u\).

Similar to \cite{li2021kudla}*{Remark 3.4.2}, the relation of the local intersection indices and the local Whittaker function is given by the following proposition.
\begin{lemma}\label{lemma:local_siegel_weil}
    Let \(\gamma=\gamma_{V_{\underline{u}},\psi_{\dF,\underline{u}}}\) be the Weil constant, then we have
    \begin{equation}
        \begin{aligned}
            W_{I_{2r}^{-\epsilon}}(0,1_{4r},\Phi^0_{\underline{u}})=&\frac{\gamma^{2r}}{b_{2r,\underline{u}}(0)}\vol(L_{\underline{u}})\\
        \end{aligned}
    \end{equation}
    where \(L_{\underline{u}}\) is the stabilizer of a unimodular lattice in \(V^{-\epsilon}\) and the volume is defined by \cite{li2021chow}*{Definition 3.8}.
\begin{proof}
    By \cite{li2021chow}*{Definition 3.8}, we have 
    \begin{equation}
        \begin{aligned}
            W_{I_{2r}^{-\epsilon}}(0,1_{4r},\Phi^0_{\underline{u}})=&\frac{\gamma^{2r}}{b_{2r,\underline{u}}(0)}\int_{H(\dF_{\underline{u}})}\Phi_{\underline{u}}^0(h_{\underline{u}}^{-1}x)\rd h_{\underline{u}}\\
        \end{aligned}
    \end{equation}
    for any \(x\) with the fixed corresponding moment matrix and \(\Phi_{\underline{u}}^0\) is the Gaussian function defined in \cite{li2021chow}*{Definition 3.8}. We can \(x\) so that each \(x_i\) is a generator of the unimodular lattice, then the integral is simply the volume of \(L_{\underline{u}}\) .
\end{proof}
\end{lemma}
\begin{proposition}\label{prop:local_siegel_weil}
    \[\chi (\pi_*^{\wedge K}\cN (x)) =\frac{1}{\log q_u}\frac{W_{T^{\square}}'(0,1_{4r},\Phi_{\underline{u}})}{W_{I_{2r}^{-\epsilon}}(0,1_{4r},\Phi_{\underline{u}})}=\frac{b_{2r,\underline{u}}(0)}{\log q_u \gamma^{2r}\vol(L_{\underline{u}})}W'_{T^\square}(0,1_{4r},\Phi_{\underline{u}})\]

    where \(\Phi_{\underline{u}}\) is defined in \cite{he2023kudla}*{Section 12.2 (12.10)} and \(\epsilon\) is the sign of the Hermitian space.
\end{proposition}
\begin{proof}
    The first identity is exactly \cite{he2023kudla}*{Conjecture 12.2}. The second one follows from 
    \cite{he2023kudla}*{Theorem 6.1 and (12.14)}.
\end{proof}

\begin{proposition}
    Take an element \(u\in\tV_\dE^{\ram}\) with local root number \(\epsilon_{\underline{u}}=-1\) whose underlying rational prime \(p\) is odd and unramified in \(\dF\), and satisfies \(\tV^{(p)}_\dF\cap \tR\subset \tV_\dF^{\spl}\). Fix a \(\psi_{\dE,\underline{u}}\)-unimodular lattice \(\Lambda_{\underline{u}}^\star\) of the nearby space \(^uV_{\underline{u}}\). 
    
    Then there exist elements \(s_1^u,s_2^u\in\dS_{\dQ^{ac}}^\tR\backslash\fm_{\pi}^\tR\) such that

    \begin{align*}
        \log q_u\cdot \vol^{\natural}(L) &\cdot I_{T_1,T_2}(\phi_1^{\infty}, \phi_2^{\infty}, \mathrm{s}_1^u\mathrm{s}_1, \mathrm{s}_2^u\mathrm{s}_2, g_1, g_2)_{L,u}^{\ell} \\
                =&(\fE_{T_1,T_2}((g_1,g_2), \Phi_{\infty}^0 \otimes (\mathrm{s}_1^u\mathrm{s}_1\phi_1^{\infty} \otimes (\mathrm{s}_2^u\mathrm{s}_2\phi_2^{\infty})^c))_u \\
                +&\log q_u\sum_{i=1}^r\frac{c_{2r,i}^+}{q^{2i}_{\underline{u}}} E_{T_1,T_2}((g_1,g_2), \Phi_{\infty}^0 \otimes (\mathrm{s}_1^u\mathrm{s}_1\phi_1^{\infty,u} \otimes (\mathrm{s}_2^u\mathrm{s}_2\phi_2^{\infty,u})^c) \otimes \Phi_{\underline{u}}^i))\\
    \end{align*}
    for every $(\tR, \tR', \ell, L)$-admissible sextuple $(\phi_1^{\infty}, \phi_2^{\infty}, \mathrm{s}_1, \mathrm{s}_2, g_1, g_2)$. For every pair $(T_1, T_2)$ in $\Herm_r^{\circ}(\dF)^+$, the right-hand side is defined in \cite{li2021chow}*{Definition 3.11} with the Gaussian function $\Phi_{\infty}^0 \in \mathscr{S}(V^{2r} \otimes_{\dA_\dF} \dF_{\infty})$.
    \begin{proof}
        The proof is similar to \cite{li2021chow}*{Proposition 8.1 and 9.1}. Modify the original proof, we have
        \begin{equation}
            \begin{aligned}
                &\frac{\vol(H(F_\infty)L^{\underline{u}})}{\deg (Y/K)} \Phi^0_\infty(T_1,T_2)\langle s_1^*Z_{T_1}(\phi_1^\infty)'_L,s_2^*Z_{T_2}(\phi_2^\infty)'_L\rangle_{X'_L,K}\\
                =&\frac{\vol(H(F_\infty)L^{\underline{u}})}{\deg (Y/K)} \Phi^0_\infty(T_1,T_2)\cdot\sum_{T^\square}\sum_{x}\sum_{h}(s_1\phi_1^{\infty,\underline{u}}\otimes(s_2\phi_2^{\infty,\underline{u}})^\tc)(h^{-1}x)\cdot\chi(\pi_{*}^{\wedge K}\cN(x))\\
            \end{aligned} \end{equation}
            Use Proposition \ref{prop:local_siegel_weil}, here \(\epsilon=+\), we have
            \begin{equation}\label{eq:local_siegel_weil1}
                \begin{aligned}
                \chi(\pi_{*}^{\wedge K}\cN(x))=&\frac{1}{\log q_u}\frac{W_{T^\square}'(0,1_{4r},\Phi_{\underline{u}})}{2q_{\underline{u}}^{-4r^2}(1+ q_{\underline{u}}^{-r})\prod_{i=1}^{r-1}(1-q_{\underline{u}}^{-2i})}\\
                =&\frac{b_{2r,\underline{u}}(0)}{\log q_u}\frac{1-q_{\underline{u}}^{-2r}}{2q_{\underline{u}}^{-4r^2}(1+ q_{\underline{u}}^{-r})}W_{T^\square}'(0,1_{4r},\Phi_{\underline{u}})\\
                =&\frac{b_{2r,\underline{u}}(0)}{\log q_u}\frac{1-q_{\underline{u}}^{-r}}{2q_{\underline{u}}^{-4r^2}}W_{T^\square}'(0,1_{4r},\Phi_{\underline{u}})\\
                =&\frac{b_{2r,\underline{u}}(0)}{\log q_u}\frac{q_{\underline{u}}^{r}-1}{2q_{\underline{u}}^{-4r^2+r}}W_{T^\square}'(0,1_{4r},\Phi_{\underline{u}})\\
                \end{aligned}
            \end{equation}
            
            Recall that by the definition of \(\Phi_{\underline{u}}\), see \cite{he2023kudla}*{Section 12.2 (12.10)}, we have 

            \[
                \Phi_{\underline{u}}=\Phi_{\underline{u}}^0+\sum_{i=1}^r A_{p,+}^i(s)\cdot\Phi_{\underline{u}}^i
            \]
            so 
            \[
            W_{T^\square}'(0,1_{4r},\Phi_{\underline{u}})=W_{T^\square}'(0,1_{4r},\Phi_{\underline{u}}^0)+\log q_{\underline{u}}\sum_{i=1}^r \frac{c_{2r,i}^+}{q^{2i}_{\underline{u}}}\cdot W_{T^\square}(0,1_{4r},\Phi_{\underline{u}}^i)
            \]
            and Equation \eqref{eq:local_siegel_weil1} becomes
            \begin{equation}\label{eq:local_siegel_weil2}
                \chi(\pi_{*}^{\wedge K}\cN(x))=\frac{b_{2r,\underline{u}}(0)}{\log q_u}\frac{q_{\underline{u}}^{r}-1}{2q_{\underline{u}}^{-4r^2+r}}\cdot (W_{T^\square}'(0,1_{4r},\Phi_{\underline{u}}^0)+\log q_{\underline{u}}\sum_{i=1}^r \frac{c_{2r,i}^+}{q^{2i}_{\underline{u}}}\cdot W_{T^\square}(0,1_{4r},\Phi_{\underline{u}}^i))
            \end{equation}
            Then combine Proposition \ref{prop:local_siegel_weil} and Lemma \ref{lemma:local_siegel_weil} we have
            
            \begin{align*}
                \log q_u\cdot \vol^{\natural}(L) &\cdot I_{T_1,T_2}(\phi_1^{\infty}, \phi_2^{\infty}, \mathrm{s}_1^u\mathrm{s}_1, \mathrm{s}_2^u\mathrm{s}_2, g_1, g_2)_{L,u}^{\ell} \\
                =&(\fE_{T_1,T_2}((g_1,g_2), \Phi_{\infty}^0 \otimes (\mathrm{s}_1^u\mathrm{s}_1\phi_1^{\infty} \otimes (\mathrm{s}_2^u\mathrm{s}_2\phi_2^{\infty})^c))_u \\
                +&\log q_u\sum_{i=1}^r\frac{c_{2r,i}^+}{q^{2i}_{\underline{u}}} E_{T_1,T_2}((g_1,g_2), \Phi_{\infty}^0 \otimes (\mathrm{s}_1^u\mathrm{s}_1\phi_1^{\infty,u} \otimes (\mathrm{s}_2^u\mathrm{s}_2\phi_2^{\infty,u})^c) \otimes \Phi_{\underline{u}}^i))\\
            \end{align*}
        \end{proof}
        \end{proposition}

\section{The Arithmetic Inner Product Formula}\label{sec:aipf}

The proof of our arithmetic inner product formula is the same as \cite{li2021chow}*{Proposition 11.1}. Except that now we allow the place with local root number \((-1)\) to be ramified.

\subsection{Preparation}
Let \((\pi,\cV_\pi)\) be as in \cite{li2022chow}*{Assumption 1.3} with \(|\tS_\pi|\) odd. Take
\begin{itemize}
\item a totally positive definite hermitian space \(V\) over \(\dA_\dE\) of rank \(n=2r\) as in
Notation 2.2 satisfying that \(\epsilon(V_v) = -1\) if and only if \(v \in \tS_\pi\) (so that \(V\) is
incoherent);
\item \(\tS = \tS_\pi\cup\tS'_\pi\) (the underlying rational prime of \(v\in\tS\) is unramified in \(\dE\) if and only if \(v \in \tS_\pi\));
\item \(\tR\) is a finite subset of \(\tV_\dF^{\fin}\) containing \(\tR_\pi\) and of cardinality at least 2. We set \(\tR'=\tR\);
\item an \(\tR\)-good rational prime \(\ell\) (\cite{li2021chow}*{Definition 6.1} and \cite{li2022chow}*{Definition 4.13});
\item for \(i = 1,2\), a nonzero element \(\varphi_i = \otimes_v \varphi_{iv} \in \cV_{[\pi]}^\tR\) satisfying that \(\langle \varphi_{1v}^c, \varphi_{2v} \rangle_{\pi_v} = 1\) for \(v \in \tV^\fin_\dF \setminus \tR\);
\item for \(i = 1,2\), an element \(\phi_i^\infty = \otimes_v \phi_{iv}^\infty \in \sS(V^r \otimes_{ \dA_\dF} \dA_\dF^\infty)\) satisfying
\begin{itemize}
\item \(\phi_{iv}^\infty = 1_{(\Lambda_v^{\tR})^r}\) for \(v \in \tV_\dF^{\fin} \setminus (\tR)\);
\item \(\text{supp}(\phi_{1v}^\infty \otimes (\phi_{2v}^\infty)^c) \subseteq (V_v^{2r})^{\text{reg}}\) for \(v \in \tR\);
\end{itemize}
\item an open compact subgroup \(L\) of \(H(\dA_\dF^\infty)\) of the form \(L^{\tR} L_{\tR}\), where \(L^{\tR}\) is
defined in \cite{li2022chow}*{Notation 2.2(H8)}, that fixes both \(\varphi_1^\infty\) and \(\varphi_2^\infty\);
\item an open compact subgroup \(K \subseteq G_r(\dA_\dF^\infty)\) that fixes \(\phi_1, \phi_2, \varphi_1^\infty, \varphi_2^\infty\);
\item a set of representatives \(\{g^{(1)}, \ldots, g^{(s)}\}\) of the double coset
\[G_r(\dF) \backslash G_r(\dA_\dF^\infty) / K\]
satisfying \(g^{(j)} \in G_r(\dA_{\dF}^{\infty,\tR})\) for \(1 \leq j \leq s\), together with a Siegel fundamental domain \(\mathcal{F}^{(j)} \subseteq G_r(F_\infty)\) for the congruence subgroup \(G_r(\dF) \cap g^{(j)} K (g^{(j)})^{-1}\) for each \(1 \leq j \leq s\);
\item for \(i = 1,2\), \(s_i\) a product of two elements in \((\dS_{\mathbb{Q}^{\text{ac}}}^{\tR})^{\langle \ell \rangle}_{L_\tR}\) satisfying \(\chi_\pi^{\tR}(s_i) = 1\) (which is possible by Proposition 6.9(2));
\item for \(i = 1,2\), an element \(s_i^u \in (\dS_{\mathbb{Q}^{\text{ac}}}^{\tR})^{\langle \ell \rangle}_{L_\tR}\) for every \(u \in \tV_\dE^{\spl} \cup \tS_\dE\), where \(\tS_\dE\) denotes the subset of \(\tV_\dE^{\text{int}}\) above \(\tS\), satisfying \(\chi_\pi^{\tR}(s_i^u) = 1\) and that \(s_i^u = 1\) for all but finitely many \(u\)\footnote{In deed, we do not need to modify the definition (except the condition on \(\chi_\pi^{\tR}\)) because \( (\dS_{\mathbb{Q}^{\text{ac}}}^{\tR})^{\langle \ell \rangle}_{L_\tR}\) is only supported in the split places.}.
\end{itemize}
In what follows, we put
\[\tilde{s}_i := s_i \cdot \prod_{u \in \tV_\dE^{\spl} \cup \tS_\dE} s_i^u\]
for \(i = 1,2\).

\subsection{Proof of the main result}
\begin{proof}[Proof of Theorem \ref{thm:main}, Theorem \ref{thm:aipf} and Corollary \ref{co:aipf}]\label{proof:main}
  The proof is almost the same as \cite{li2021chow}*{Section 11} and \cite{li2022chow}*{Remark 4.32}. The main difference here is that we need to deal with the places in \(\tS'_\pi\). 

  Now we put
  \begin{equation}
    \begin{aligned}
      \fE^{S}_{T_1,T_2}((g_1,g_2),\Phi^0_\infty\otimes\Phi^\infty):=&\sum_{u\in\tV_{\dE}\setminus \tV_\dE^\spl}\fE_{T_1,T_2}((g_1,g_2),\Phi^0_\infty\otimes\Phi^\infty)_u\\
      &-\sum_{u\in \dS_{\dE},\underline{u}\in \tS_\pi}\frac{\log q_u}{q_u^r-1}E_{T_1,T_2}((g_1,g_2),\Phi^0_\infty\otimes\Phi^{\infty,\underline{u}}\otimes 1_{(\Lambda_{\underline{u}}^{\star})^{2r}})\\
      &-\sum_{u\in \dS_{\dE},\underline{u}\in \tS'_\pi}\log q_{\underline{u}}\sum_{i=1}^r\frac{c_{2r,i}^+}{q_{\underline{u}}^{2i}}E_{T_1,T_2}((g_1,g_2),\Phi_\infty^0\otimes\Phi^{\infty,\underline{u}}\otimes\Phi^i_{\underline{u}})
    \end{aligned}
  \end{equation}
  And the remaining part of the proof still works here as the new error terms from \(\underline{u}\in\tS'_\pi\) also vanish after integration by the Rallis inner product formula (see, e.g. \cite{liu2011arithmetic}*{(2)-(6)}).

  The computation of the doubling zeta integral in Corollary \ref{co:aipf} then follows Proposition \ref{prop:zeta_almost_spherical_unimodular_non_split}.
\end{proof}

\appendix
\section{Modified Poincar\'e Polynomials for \texorpdfstring{\(\fW_n\)}{Wn}}
In this appendix, we collect some combinatorial results about modified Poincar\'e polynomials which may be of independent interest.

\subsection{Formalization}
For any positive integer \(n\), we have the group \(\fW_n=\{\pm1\}^n\rtimes\fS_n\) where \(\fS_n\) is the symmetric group on \(n\) elements. For any \(w\in\fW_n\), we can write it uniquely as \(w_I\tau\) where for \(I\subset\{1,\cdots,n\}\), \(w_I\) is the element of \(\{\pm1\}^n\) whose coordinates in \(I\) are set to be \(-1\), and \(\tau\in\fS_n\).

In this appendix, we prove that, if we equip different length functions on \(\fW_n\) according to different choices of generators (corresponding to different root systems), then we have different versions of Poincar\'e polynomial identities.

The following proposition is for \(C\)-type Coxeter group which is an analogue (and a slight generalization) of Poincar\'e's polynomials (see, for example, \cite{bjorner2005combinatorics}):
\begin{proposition}[A Modified Poincar\'e polynomial identity for \(C_n\)-type Weyl group]\label{prop:Toad_identity}
     Suppose that we fix a system of generators \(\left<w_1,w_{1,2},\cdots,w_{n-1,n}\right>\), then we have the following identity:
    \begin{equation}
        \begin{aligned}
            \sum_{w\in\fW_n}q^{\ell(w)}X^w=&\sum_{I\subset[n]}\sum_{\tau\in\fS_n}q^{\ell(w_I\tau)}X^I\\
            =&[n]_q!\prod_{i=1}^n(1+q^ix_i)
        \end{aligned}
    \end{equation}
    where \(X^w=X^{W_I\tau}:=X^I:=\prod_{i\in I}x_i\) and \([n]_q!\) is the \(q\)-factorial.
\end{proposition}

\begin{proof}[Proof of Proposition \ref{prop:Toad_identity}]
    The proof is broken into two steps. 
    \begin{enumerate}[(1)]
        \item 
        First, we prove that the identity for constant terms holds, which states
        \begin{equation}
            [n]_q!=\sum_{\tau\in\fS_n}q^{\ell(\tau)}
        \end{equation}
        where \(\ell(w)\) is the length of the permutation \(w\in\fS_n\).
        \item 
        Then we compare the coefficients of \(X^I\) on both sides of the identity and prove that these new identities holds. That is to say, we need to prove that
        \begin{equation}
            \sum_{\tau\in\fS_n}q^{\ell(w_I\tau)}=\sum_{\tau\in\fS_n}q^{(\sum_{i\in I}i)+\ell(\tau)}
        \end{equation}
    \end{enumerate}
    Combine these two steps, we have 
    \begin{equation}
        \begin{aligned}
            [n]_q!\prod_{i=1}^n(1+q^ix_i)=&\sum_{I\subset[n]}X^Iq^{(\sum_{i\in I}i)}[n]_q!\\
            =&\sum_{I\subset[n]}X^I\sum_{\tau\in\fS_n}q^{(\sum_{i\in I}i)+\ell(\tau)}\\
            =&\sum_{I\subset[n]}X^I\sum_{\tau\in\fS_n}q^{\ell(w_I\tau)}\\
            =&\sum_{w\in\fW_n}q^{\ell(w)}X^w
        \end{aligned}
    \end{equation}
    The proposition is thus proved. Now we prove the two steps.
    \begin{enumerate}[(1)]
        \item This is due to the Poincar\'e's polynomial (see, for example, \cite{bjorner2005combinatorics}*{(7.5)}).

        \item The strategy is that, for any \(I\) fixed, we prove that for any \(w_I w_\tau\in\fW_{n}\), there exists a unique \(\tau'\in\fS_n\) such that \(\ell(w_Iw_\tau)=\sum_{i\in I}i+\ell(\tau')\). 

        Identify an element \(\tau\) in \(\fW_n\) with the function \(\tau\) from \([n]\) to \(\{-n,\cdots,-1,1,\cdots,n\}\) such that the set of the absolute values of the image of \(\tau\) is exactly \([n]\), this function is extended to \(\{-1,\cdots,-n\}\) by \(w(-x)=-w(x)\). We will consider its wiring diagram. We know that the length of \(\tau\) is the number of the equivalent classes of intersection points.
        
        For a fixed \(I\subset[n]\), we consider all elements of the form \(w=w_I\tau\). This means the negative elements in \(w([n])\) consists of \(-i\) for \(i\in I\).

        For any \(i\in I\), if \(0<j<i\), then the segment connecting \(w^{-1}j\) and \(j\) will intersect the segments from \(w^{-1}i\) to \(i\) at one point or two inequivalent points. The first case happens if and only if \(w^{-1}(j)>w^{-1}(-i)\) or \(w^{-1}(j)<w^{-1}(i)\), and equivalently, \(|w^{-1}(j)|>|w^{-1}(i)|\) which means in the wiring diagram of \(|w|\), they do meet each other. 

        So we know that the number of intersection points is at least
        \(\sum_{i\in I}(1+\sum_{j<i}1)=\sum_{i\in I}i\) for \(w=w_I\tau\). Moreover, we know that the additional intersection points are from three classes:
        \begin{enumerate}[(i)]
            \item \(i\in I\), \(j<i\), \(|w^{-1}(j)|<|w^{-1}(i)|\)
            \item \(i\in I\), \(j>i\), \(|w^{-1}(j)|<|w^{-1}(i)|\)
            \item \(i,j\notin I\), \(j>i\), \(|w^{-1}(j)|<|w^{-1}(i)|\)
        \end{enumerate}
        It is now clear that, if we permute the elements in \(I\) to be the least \(|I|\) elements in \([n]\), then the above comes from the length of an element in \(\fS_n\). So we have proved the identity:
        \begin{equation}
            \sum_{\tau\in\fS_n}q^{\ell(w_I\tau)}=\sum_{\tau'\in\fS_n}q^{(\sum_{i\in I}i)+\ell(\tau')}
        \end{equation}
    \end{enumerate}
     Combining these two steps, we have proved the proposition.
\end{proof}

We have another version if we give the odd generator \(w_1\) length 0, which reflects the case of type \(D_n\).
\begin{proposition}\label{prop:Toad_identity2}
    Suppose that we fix a system of generators of the subgroup of type \(D_n\) root system \(\left<w_1w_2w_{1,2},w_{1,2},\cdots,w_{n-1,n}\right>\) and include one new element \(w_1\) which flips the sign of the first coordinate. We define a new length function \(\ell\) by setting \(\ell(w_1)=0\), \(\ell(w_1w_2w_{12})=1\), and \(\ell(w_{i,i+1})=1\) for \(i=1,\cdots,n-1\), then we have the following identity:
    \begin{equation}
        \begin{aligned}
            \sum_{w\in\fW_n}q^{\ell(w)}X^w=&\sum_{I\subset[n]}\sum_{\tau\in\fS_n}q^{\ell(w_I\tau)}X^I\\
            =&[n]_q!\prod_{i=1}^n(1+q^{i-1}x_i)
        \end{aligned}
    \end{equation}
    where \(X^w=X^{W_I\tau}:=X^I:=\prod_{i\in I}x_i\) and \([n]_q!\) is the \(q\)-factorial.
\end{proposition}
The proof is similar but we need to consider the \(D_r\)-type Coxeter group instead.

If we consider the non-split orthogonal group, we have the following version:
\begin{proposition}\label{prop:Toad_identity3}
    Suppose that we fix a system of generators of the subgroup of type \(^2D_n\) root system \(\left<w_1,w_{1,2},\cdots,w_{n-1,n}\right>\) and define a function \(\ell\) by setting \(\ell(w_1)=2\), \(\ell(w_{1,2})=1\), and \(\ell(w_{i,i+1})=1\) for \(i=1,\cdots,n-1\) and extending to the whole group by \cite{ariki2008finite}, then we have the following identity:
    \begin{equation}
        \begin{aligned}
            \sum_{w\in\fW_n}q^{\ell(w)}X^w=&\sum_{I\subset[n]}\sum_{\tau\in\fS_n}q^{\ell(w_I\tau)}X^I\\
            =&[n]_q!\prod_{i=1}^n(1+q^{i+1}x_i)
        \end{aligned}
    \end{equation}
\end{proposition}


\subsection{Application}

One of the possible uses of the above combinatorial identities is to detect the relation of vectors in the unramified principal series. For example, one wants to understand whether \(\phi_{L_r}\) from Section \ref{sec:quasi-split} generates the \(K_r\)-spherical vector \(\phi_{K_r}\) in \(\rI_{W_r}^\sigma\), a possible way is to compute \(1_{K_r}.\phi_{L_r}\), which must be a scalar multiple of \(\phi_{K_r}\). For convenience, we use \(w\) to represent the representative \(\omega(w)\) of \(w\) in \(K_r\) or \(K_r^-\) as we do not need \(\Omega(w)\) here.
\begin{proposition}\label{prop:L-to-K-quasi-split}
    Let \(\phi_{L_r}\) be the \(L_r\)-spherical vector in \(\rI_{W_r}^\sigma\) defined in Section \ref{sec:quasi-split} and we equip \(I_r\) with Haar measure \(1\), then 
    \[
    (1_{K_r}.\phi_{L_r})(1)=[r]_q!\prod_{i=1}^r(1+q^{-\sigma_i+\frac{1}{2}}).
    \]
\end{proposition}
\begin{proof}
    
    Now we have
    \begin{equation}
        \begin{aligned}
            (1_{K_r}.\phi_{L_r})(1)=&\int_{K_r}\phi_{L_r}(k)\rd k\\
        =&\int_{K_r}\sum_{w\in\fW_r}\phi'_w(k)\rd k\\
        =&\int_{K_r}\sum_{w_Iw_\tau\in\fW_r}q^{-(\sum_{i\in I}\sigma_i+i-\frac{1}{2})}\phi_{w_Iw_\tau}(k)\rd k\\
        =&\sum_{w_Iw_\tau\in\fW_r}q^{-\sum_{i\in I}(\sigma_i+i-\frac{1}{2})}\int_{K_r}\phi_{w_Iw_\tau}(k)\rd k\\
        =&\sum_{w_Iw_\tau\in\fW_r}q^{-\sum_{i\in I}(\sigma_i+i-\frac{1}{2})}\int_{I_r\omega(w_Iw_\tau) I_r}\phi_{w_Iw_\tau}(k)\rd k\\
        =&\sum_{w_Iw_\tau\in\fW_r}q^{-\sum_{i\in I}(\sigma_i+i-\frac{1}{2})}\vol(I_r\omega(w_Iw_\tau) I_r)\\
        =&\sum_{w_Iw_\tau\in\fW_r}q^{-\sum_{i\in I}(\sigma_i+i-\frac{1}{2})}[I_r\omega(w_Iw_\tau) I_r:I_r]\\
        =&\sum_{w_Iw_\tau\in\fW_r}q^{-\sum_{i\in I}(\sigma_i+i-\frac{1}{2})}q^{\ell(w_Iw_\tau)}\\
        =&\sum_{I\subset\{1,\cdots,r\}}\sum_{w_\tau\in\fS_r}q^{-\sum_{i\in I}(\sigma_i+i-\frac{1}{2})}q^{\ell(w_Iw_\tau)}\\
        \overset{\text{Proposition}\ref{prop:Toad_identity}}{=} &[r]_q!\prod_{i=1}^r(1+q^iq^{-\sigma_i-i+\frac{1}{2}})\\
        =&[r]_q!\prod_{i=1}^r(1+q^{-\sigma_i+\frac{1}{2}})
        \end{aligned}
    \end{equation}
\end{proof}
\begin{remark}
    According to the above computation, we know that when certain Satake parameter satisfies \(q^{-\sigma_i}=-q^{-\frac{1}{2}}\) for some \(i\), then \((1_{K_r}.\phi_{L_r})(1)=0\). That is why we need additional assumption to prove Theorem \ref{thm:spherical_L_is_spherical_K} directly. 
\end{remark}
\begin{proposition}\label{prop:L-to-K-non-quasi-split}
    Let \(\phi_{L_r^-}\) be the \(L_r^-\)-spherical vector in \(\rI_{V_r^-}^\sigma\) defined in Section \ref{sec:non-quasi-split} and we equip \(I_r^-\) with Haar measure \(1\), then 
    \[
    (1_{L_r^-}.\phi_{K_r^-})(1)=[r]_q!\prod_{i=1}^r(1+q^{-\sigma_i-\frac{1}{2}}).
    \]
\begin{proof}
    The proof is similar to Proposition \ref{prop:L-to-K-quasi-split}.
we have
    \begin{equation}
        \begin{aligned}
            (1_{L_r^-}.\phi_{K_r^-})(1)=&\int_{L_r^-}\phi_{K_r^-}(k)\rd k\\
        =&\int_{L_r^-}\sum_{w\in\fW_r}\phi'_w(k)\rd k\\
        =&\int_{L_r^-}\sum_{w_Iw_\tau\in\fW_r}q^{-(\sum_{i\in I}\sigma_i+i+\frac{1}{2})}\phi_{w_Iw_\tau}(k)\rd k\\
        =&\sum_{w_Iw_\tau\in\fW_r}q^{-\sum_{i\in I}(\sigma_i+i+\frac{1}{2})}\int_{K_r^-}\phi_{w_Iw_\tau}(k)\rd k\\
        =&\sum_{w_Iw_\tau\in\fW_r}q^{-\sum_{i\in I}(\sigma_i+i+\frac{1}{2})}\int_{I_r^-\Omega(w_Iw_\tau) I_r^-}\phi_{w_Iw_\tau}(k)\rd k\\
        =&\sum_{w_Iw_\tau\in\fW_r}q^{-\sum_{i\in I}(\sigma_i+i+\frac{1}{2})}\vol(I_r^-\Omega(w_Iw_\tau) I_r^-)\\
        =&\sum_{w_Iw_\tau\in\fW_r}q^{-\sum_{i\in I}(\sigma_i+i+\frac{1}{2})}[I_r^-\Omega(w_Iw_\tau) I_r^-:I_r^-]\\
        \end{aligned}
    \end{equation}

Now we need to compute the index \([I_r^-\Omega(w_Iw_\tau) I_r^-:I_r^-]\). Note that in this case, the length function is defined by setting \(\ell(w_1)=2\), \(\ell(w_{1,2})=1\), \dots \(\ell(w_{i,i+1})=1\) for \(i=1,\cdots,r-1\). So we have

    \begin{equation}
        \begin{aligned}
        =&\sum_{w_Iw_\tau\in\fW_r}q^{-\sum_{i\in I}(\sigma_i+i+\frac{1}{2})}q^{\ell(w_Iw_\tau)}\\
        =&\sum_{I\subset\{1,\cdots,r\}}\sum_{w_\tau\in\fS_r}q^{-\sum_{i\in I}(\sigma_i+i+\frac{1}{2})}q^{\ell(w_Iw_\tau)}\\
        \overset{\text{Proposition}\ref{prop:Toad_identity3}}{=} &[r]_q!\prod_{i=1}^r(1+q^iq^{-\sigma_i-i-\frac{1}{2}})\\
        =&[r]_q!\prod_{i=1}^r(1+q^{-\sigma_i-\frac{1}{2}})
        \end{aligned}
    \end{equation}

\end{proof}
\end{proposition}

\bibliographystyle{plain}
\bibliography{ref}

@incollection{ariki2008finite,
  title     = {Finite dimensional {Hecke} algebras},
  booktitle = {{{EMS Series}} of {{Congress Reports}}},
  author    = {Ariki, Susumu},
  editor    = {Skowro{\'n}ski, Andrzej},
  year      = 2008,
  edition   = {1},
  volume    = {1},
  pages     = {1--48},
  publisher = {EMS Press},
  doi       = {10.4171/062-1/1},
  urldate   = {2026-01-14},
  isbn      = {978-3-03719-062-3 978-3-03719-562-8},
  langid    = {english}
}

@phdthesis{badea2020hecke,
  title   = {Hecke algebras for types of quasisplit unitary groups over local fields in the principal series},
  author  = {Badea, M. P.},
  year    = 2020,
  school  = {Radboud University},
  urldate = {2025-08-14},
  langid  = {english},
  note    = {available at \url{https://hdl.handle.net/2066/226860}}
}

@book{bjorner2005combinatorics,
  title     = {Combinatorics of Coxeter Groups},
  author    = {Bj{\"o}rner, Anders and Brenti, Francesco},
  year      = 2005,
  series    = {Graduate Texts in Mathematics},
  number    = {231},
  publisher = {Springer},
  address   = {New York, NY},
  isbn      = {978-3-540-44238-7},
  langid    = {english},
  lccn      = {QA182.5 .B56 2005}
}

@article{borel1976admissible,
  title    = {Admissible representations of a semi-simple group over a local field with vectors fixed under an {Iwahori} subgroup},
  author   = {Borel, Armand},
  year     = 1976,
  journal  = {Inventiones mathematicae},
  volume   = {35},
  number   = {1},
  pages    = {233--259},
  issn     = {1432-1297},
  doi      = {10.1007/BF01390139},
  urldate  = {2025-01-15},
  langid   = {english}
}

@article{brubaker2024colored,
  title    = {Colored vertex models and {Iwahori} {Whittaker} functions},
  author   = {Brubaker, Ben and Buciumas, Valentin and Bump, Daniel and Gustafsson, Henrik P. A.},
  year     = 2024,
  journal  = {Selecta Mathematica},
  volume   = {30},
  number   = {4},
  pages    = {78},
  issn     = {1420-9020},
  doi      = {10.1007/s00029-024-00950-6},
  urldate  = {2025-01-10},
  langid   = {english}
}

@inproceedings{cartier1979representations,
  title      = {Representations of p-adic groups: A survey},
  shorttitle = {Representations of $p$-adic groups},
  booktitle  = {Automorphic Forms, Representations and {{$L$-functions}} ({{Proc}}. {{Sympos}}. {{Pure Math}}., {{Oregon State Univ}}., {{Corvallis}}, {{Ore}}., 1977), {{Part}}},
  author     = {Cartier, Pierre},
  year       = 1979,
  volume     = {1},
  pages      = {111--155},
  urldate    = {2025-01-14}
}

@article{casselman1980unramified,
  title     = {The unramified principal series of $p$-adic groups. {I}. the spherical function},
  author    = {Casselman, W.},
  year      = 1980,
  journal   = {Compositio Mathematica},
  volume    = {40},
  number    = {3},
  pages     = {387--406},
  publisher = {Sijthoff et Noordhoff International Publishers},
  issn      = {0010-437X},
  urldate   = {2025-01-14},
  langid    = {english}
}

@article{gan2015proof,
  title     = {A proof of the {Howe} duality conjecture},
  author    = {Gan, Wee Teck and Takeda, Shuichiro},
  year      = 2015,
  journal   = {Journal of the American Mathematical Society},
  volume    = {29},
  number    = {2},
  pages     = {473--493},
  issn      = {0894-0347, 1088-6834},
  doi       = {10.1090/jams/839},
  urldate   = {2025-01-06},
  copyright = {https://www.ams.org/publications/copyright-and-permissions},
  langid    = {english}
}

@article{gan2016gross,
  title   = {The {Gross--Prasad} conjecture and local theta correspondence},
  author  = {Gan, Wee Teck and Ichino, Atsushi},
  year    = 2016,
  journal = {Inventiones mathematicae},
  volume  = {206},
  number  = {3},
  pages   = {705--799},
  issn    = {0020-9910, 1432-1297},
  doi     = {10.1007/s00222-016-0662-8},
  urldate = {2025-01-15},
  langid  = {english}
}

@article{gross1986heegner,
  title     = {Heegner points and derivatives of {$L$}-series},
  author    = {Gross, Benedict H. and Zagier, Don B.},
  year      = 1986,
  journal   = {Inventiones Mathematicae},
  volume    = {84},
  number    = {2},
  pages     = {225--320},
  issn      = {0020-9910, 1432-1297},
  doi       = {10.1007/BF01388809},
  urldate   = {2026-01-14},
  copyright = {http://www.springer.com/tdm},
  langid    = {english}
}

@article{gross1991applications,
  title     = {Some applications of {Gelfand} pairs to number theory},
  author    = {Gross, Benedict H.},
  volume    = {24},
  journal   = {Bulletin (New Series) of the American Mathematical Society},
  number    = {2},
  publisher = {American Mathematical Society},
  pages     = {277 -- 301},
  year      = {1991}
}

@article{he2023kudla,
  title   = {{Kudla--Rapoport} conjecture for {Kr\"amer} models},
  author  = {He, Qiao and Shi, Yousheng and Yang, Tonghai},
  year    = 2023,
  journal = {Compositio Mathematica},
  volume  = {159},
  number  = {8},
  pages   = {1673--1740},
  issn    = {0010-437X, 1570-5846},
  doi     = {10.1112/S0010437X23007273},
  urldate = {2025-01-01},
  langid  = {english}
}

@article{he2023proof,
  title   = {A proof of the {Kudla--Rapoport} conjecture for {Kr\"amer} models},
  author  = {He, Qiao and Li, Chao and Shi, Yousheng and Yang, Tonghai},
  year    = 2023,
  journal = {Inventiones mathematicae},
  volume  = {234},
  number  = {2},
  pages   = {721--817},
  issn    = {0020-9910, 1432-1297},
  doi     = {10.1007/s00222-023-01209-1},
  urldate = {2024-11-26},
  langid  = {english}
}

@book{kaletha2023bruhat,
  title      = {Bruhat--{{Tits Theory}}: {{A New Approach}}},
  shorttitle = {Bruhat--{{Tits Theory}}},
  author     = {Kaletha, Tasho and Prasad, Gopal},
  year       = 2023,
  edition    = {1},
  publisher  = {Cambridge University Press},
  doi        = {10.1017/9781108933049},
  urldate    = {2025-01-14},
  copyright  = {https://www.cambridge.org/core/terms},
  isbn       = {978-1-108-93304-9 978-1-108-83196-3},
  langid     = {english}
}

@article{kato1981irreducibility,
  title   = {Irreducibility of principal series representations for {Hecke} algebras of affine type},
  author  = {Kato, Shin-ichi},
  year    = 1981,
  journal = {J. Fac. Sci. Univ. Tokyo Sect. IA Math},
  volume  = {28},
  number  = {3},
  pages   = {929--943},
  urldate = {2025-01-26}
}

@article{kisin2021stable,
  title   = {The stable trace formula for {Shimura} varieties of abelian type},
  author  = {Kisin, Mark and Shin, Sug Woo and Zhu, Yihang},
  year    = 2021,
  journal = {arXiv preprint arXiv:2110.05381},
  url     = {https://arxiv.org/abs/2110.05381},
  note    = {available at \url{https://arxiv.org/abs/2110.05381}}
}

@inproceedings{kramer2003local,
  title     = {Local models for ramified unitary groups},
  booktitle = {Abhandlungen Aus Dem {{Mathematischen Seminar}} Der {{Universit\"at Hamburg}}},
  author    = {Kr{\"a}mer, Nicole},
  year      = 2003,
  volume    = {73},
  pages     = {67--80},
  publisher = {Springer}
}

@article{kudla1997central,
  title      = {Central derivatives of {Eisenstein} series and height pairings},
  author     = {Kudla, Stephen S.},
  year       = 1997,
  journal    = {Annals of Mathematics},
  volume     = {146},
  number     = {3},
  eprint     = {2952456},
  eprinttype = {jstor},
  pages      = {545--646},
  publisher  = {[Annals of Mathematics, Trustees of Princeton University on Behalf of the Annals of Mathematics, Mathematics Department, Princeton University]},
  issn       = {0003-486X},
  doi        = {10.2307/2952456},
  urldate    = {2026-01-14}
}

@article{kudla1999arithmetic,
  title   = {Arithmetic {Hirzebruch} {Zagier} cycles},
  author  = {Kudla, S. S. and Rapoport, M.},
  year    = 1999,
  journal = {Journal f\"ur die reine und angewandte Mathematik (Crelles Journal)},
  volume  = {1999},
  number  = {515},
  pages   = {155--244},
  issn    = {0075-4102, 1435-5345},
  doi     = {10.1515/crll.1999.076},
  urldate = {2026-01-22}
}

@article{kudla1999derivative,
  title   = {On the derivative of an {Eisenstein} series of weight one},
  author  = {Kudla, Stephen S. and Rapoport, Michael and Yang, Tonghai},
  year    = 1999,
  journal = {International Mathematics Research Notices},
  volume  = {1999},
  number  = {7},
  pages   = {347--385},
  issn    = {1073-7928},
  doi     = {10.1155/S1073792899000185},
  urldate = {2026-01-22}
}

@article{kudla2000height,
  title    = {Height pairings on {Shimura} curves and p-adic uniformization},
  author   = {Kudla, Stephen S. and Rapoport, Michael},
  year     = 2000,
  journal  = {Inventiones mathematicae},
  volume   = {142},
  number   = {1},
  pages    = {153--223},
  issn     = {1432-1297},
  doi      = {10.1007/s002220000087},
  urldate  = {2026-01-22},
  langid   = {english}
}

@incollection{kudla2002derivatives,
  title     = {Derivatives of {Eisenstein} series and generating functions for arithmetic cycles},
  booktitle = {S\'eminaire {{Bourbaki}} : Volume 1999/2000, Expos\'es 865-879},
  author    = {Kudla, Stephen S.},
  year      = 2002,
  series    = {Ast\'erisque},
  number    = {276},
  pages     = {341--368},
  publisher = {Soci\'et\'e math\'ematique de France},
  langid    = {english},
  mrnumber  = {1886765},
  zmnumber  = {1066.11026}
}

@article{kudla2002modular,
  title   = {Modular forms and arithmetic geometry},
  author  = {Kudla, S. S.},
  year    = 2002,
  journal = {Current Developments in Mathematics},
  volume  = {2002},
  number  = {1},
  pages   = {135--179},
  issn    = {10896384, 21644829},
  doi     = {10.4310/CDM.2002.v2002.n1.a4},
  urldate = {2026-01-14},
  langid  = {english}
}

@article{kudla2004derivatives,
  title    = {Derivatives of {Eisenstein} series and {Faltings} heights},
  author   = {Kudla, Stephen S. and Rapoport, Michael and Yang, Tonghai},
  year     = 2004,
  journal  = {Compositio Mathematica},
  volume   = {140},
  number   = {4},
  pages    = {887--951},
  issn     = {1570-5846, 0010-437X},
  doi      = {10.1112/S0010437X03000459},
  urldate  = {2026-01-22},
  langid   = {english}
}

@book{kudla2006modular,
  ids       = {kudla2010modulara,kudla2010modularb},
  title     = {Modular Forms and Special Cycles on {{Shimura}} Curves},
  author    = {Kudla, Stephen S. and Rapoport, Michael and Yang, Tonghai},
  year      = 2006,
  series    = {Annals of Mathematics Studies},
  number    = {no. 161},
  publisher = {Princeton University Press},
  address   = {Princeton},
  isbn      = {978-0-691-12550-3 978-0-691-12551-0 978-1-4008-3716-8},
  langid    = {english}
}

@article{kudla2011special,
  title={Special cycles on unitary {Shimura} varieties {I}. {Unramified} local theory},
  author={Kudla, Stephen and Rapoport, Michael},
  journal={Inventiones mathematicae},
  volume={184},
  number={3},
  pages={629--682},
  year={2011},
  publisher={Springer}
}

@article{kudla2014special,
  title={Special cycles on unitary {Shimura} varieties {II}: Global theory},
  author={Kudla, Stephen and Rapoport, Michael},
  journal={Journal f{\"u}r die reine und angewandte Mathematik (Crelles Journal)},
  volume={2014},
  number={697},
  pages={91--157},
  year={2014},
  publisher={De Gruyter}
}

@article{li1992nonvanishing,
  title   = {Non-vanishing theorems for the cohomology of certain arithmetic quotients.},
  author  = {Li, Jian-Shu},
  year    = 1992,
  journal = {Journal f\"ur die reine und angewandte Mathematik (Crelles Journal)},
  volume  = {1992},
  number  = {428},
  pages   = {177--217},
  issn    = {0075-4102, 1435-5345},
  doi     = {10.1515/crll.1992.428.177},
  urldate = {2025-01-06},
  langid  = {english}
}

@article{li2021chow,
  title     = {Chow groups and {L}-derivatives of automorphic motives for unitary groups},
  author    = {Li, Chao and Liu, Yifeng},
  journal   = {Annals of Mathematics},
  volume    = {194},
  number    = {3},
  pages     = {817--901},
  year      = {2021},
  publisher = {Department of Mathematics, Princeton University Princeton, New Jersey, USA}
}

@article{li2021kudla,
  title     = {{Kudla--Rapoport} cycles and derivatives of local densities},
  author    = {Li, Chao and Zhang, Wei},
  year      = 2021,
  journal   = {Journal of the American Mathematical Society},
  volume    = {35},
  number    = {3},
  pages     = {705--797},
  issn      = {0894-0347, 1088-6834},
  doi       = {10.1090/jams/988},
  urldate   = {2025-01-02},
  copyright = {https://www.ams.org/publications/copyright-and-permissions},
  langid    = {english}
}

@article{li2022chow,
  title   = {Chow groups and {L} -derivatives of automorphic motives for unitary groups, {II}.},
  author  = {Li, Chao and Liu, Yifeng},
  year    = 2022,
  journal = {Forum of Mathematics, Pi},
  volume  = {10},
  pages   = {e5},
  issn    = {2050-5086},
  doi     = {10.1017/fmp.2022.2},
  urldate = {2025-08-09},
  langid  = {english}
}

@article{liu2011arithmetic,
  title   = {Arithmetic theta lifting and {L}-derivatives for unitary groups, {I}},
  author  = {Liu, Yifeng},
  year    = 2011,
  journal = {Algebra \& Number Theory},
  volume  = {5},
  number  = {7},
  pages   = {849--921},
  issn    = {1944-7833, 1937-0652},
  doi     = {10.2140/ant.2011.5.849},
  urldate = {2024-11-26},
  langid  = {english}
}

@article{liu2011arithmetica,
  title   = {Arithmetic theta lifting and {L}-derivatives for unitary groups, {II}},
  author  = {Liu, Yifeng},
  year    = 2011,
  journal = {Algebra \& Number Theory},
  volume  = {5},
  number  = {7},
  pages   = {923--1000},
  issn    = {1944-7833, 1937-0652},
  doi     = {10.2140/ant.2011.5.923},
  urldate = {2024-11-26},
  langid  = {english}
}

@article{liu2021fourier,
  title   = {{Fourier--Jacobi} cycles and arithmetic relative trace formula (with an appendix by {Chao Li} and {Yihang Zhu})},
  author  = {Liu, Yifeng and Zhu, Yihang and Li, Chao},
  year    = 2021,
  journal = {Cambridge Journal of Mathematics},
  volume  = {9},
  number  = {1},
  pages   = {1--147},
  issn    = {21680930, 21680949},
  doi     = {10.4310/CJM.2021.v9.n1.a1},
  urldate = {2025-02-22},
  langid  = {english}
}

@article{liu2022beilinson,
  title   = {On the {Beilinson--Bloch--Kato} conjecture for {Rankin--Selberg} motives},
  author  = {Liu, Yifeng and Tian, Yichao and Xiao, Liang and Zhang, Wei and Zhu, Xinwen},
  year    = 2022,
  journal = {Inventiones mathematicae},
  volume  = {228},
  number  = {1},
  pages   = {107--375},
  issn    = {0020-9910, 1432-1297},
  doi     = {10.1007/s00222-021-01088-4},
  urldate = {2024-11-26},
  langid  = {english}
}

@article{liu2022theta,
  title   = {Theta correspondence for almost unramified representations of unitary groups},
  author  = {Liu, Yifeng},
  year    = 2022,
  journal = {Journal of Number Theory},
  volume  = {230},
  pages   = {196--224},
  issn    = {0022314X},
  doi     = {10.1016/j.jnt.2021.03.027},
  urldate = {2025-01-06},
  langid  = {english}
}

@misc{luo2025kudlarapoport,
  ids           = {luo2025kudlarapoporta,luo2025kudlarapoportb},
  title         = {{Kudla--Rapoport} conjecture for unramified maximal parahoric level},
  author        = {Luo, Yu},
  year          = 2025,
  number        = {arXiv:2504.18528},
  eprint        = {2504.18528},
  primaryclass  = {math},
  publisher     = {arXiv},
  doi           = {10.48550/arXiv.2504.18528},
  note          = {available at \url{https://arxiv.org/abs/2504.18528}},
  urldate       = {2025-05-08},
  archiveprefix = {arXiv}
}

@misc{luo2025unitary,
  ids           = {luo2025unitarya},
  title         = {On unitary {Shimura} varieties at ramified primes},
  author        = {Luo, Yu and Mihatsch, Andreas and Zhang, Zhiyu},
  year          = 2025,
  number        = {arXiv:2504.17484},
  eprint        = {2504.17484},
  primaryclass  = {math},
  publisher     = {arXiv},
  doi           = {10.48550/arXiv.2504.17484},
  note          = {available at \url{https://arxiv.org/abs/2504.17484}},
  urldate       = {2025-04-25},
  archiveprefix = {arXiv}
}

@article{pappas2008twisted,
  title     = {Twisted loop groups and their affine flag varieties},
  author    = {Pappas, Georgios and Rapoport, Michael},
  year      = 2008,
  journal   = {Advances in Mathematics},
  volume    = {219},
  number    = {1},
  pages     = {118--198},
  publisher = {Elsevier},
  urldate   = {2025-01-19}
}

@article{rallis1984injectivity,
  ids       = {Rallis1984},
  title     = {Injectivity properties of liftings associated to {Weil} representations},
  author    = {Rallis, Stephen},
  year      = 1984,
  journal   = {Compositio Mathematica},
  volume    = {52},
  number    = {2},
  pages     = {139--169},
  publisher = {Martinus Nijhoff Publishers},
  langid    = {english}
}

@article{rapoport2020arithmetic,
  title   = {Arithmetic diagonal cycles on unitary {Shimura} varieties},
  author  = {Rapoport, M. and Smithling, B. and Zhang, W.},
  year    = 2020,
  journal = {Compositio Mathematica},
  volume  = {156},
  number  = {9},
  pages   = {1745--1824},
  issn    = {0010-437X, 1570-5846},
  doi     = {10.1112/S0010437X20007289},
  urldate = {2024-11-26},
  langid  = {english}
}

@article{rapoport2021shimura,
  title   = {On {Shimura} varieties for unitary groups},
  author  = {Rapoport, M. and Smithling, B. and Zhang, W.},
  year    = 2021,
  journal = {Pure and Applied Mathematics Quarterly},
  volume  = {17},
  number  = {2},
  pages   = {773--837},
  issn    = {15588599, 15588602},
  doi     = {10.4310/PAMQ.2021.v17.n2.a8},
  urldate = {2024-11-26},
  langid  = {english}
}

@misc{raum2026geometrica,
  title = {The {{Geometric Unitary Kudla Conjecture}}},
  author = {Raum, Martin},
  year = 2026,
  number = {arXiv:2603.04282},
  eprint = {2603.04282},
  primaryclass = {math.NT},
  publisher = {arXiv},
  doi = {10.48550/arXiv.2603.04282},
  urldate = {2026-05-20},
  archiveprefix = {arXiv}
}

@article{shi2023special,
  title   = {Special cycles on unitary {Shimura} curves at ramified primes},
  author  = {Shi, Yousheng},
  year    = 2023,
  journal = {Manuscripta Mathematica},
  volume  = {172},
  number  = {1-2},
  pages   = {221--290},
  issn    = {0025-2611, 1432-1785},
  doi     = {10.1007/s00229-022-01412-z},
  urldate = {2024-11-26},
  langid  = {english}
}

@article{shi2023speciala,
  title   = {Special cycles on the basic locus of unitary {Shimura} varieties at ramified primes},
  author  = {Shi, Yousheng},
  year    = 2023,
  journal = {Algebra \& Number Theory},
  volume  = {17},
  number  = {10},
  pages   = {1681--1714},
  issn    = {1944-7833, 1937-0652},
  doi     = {10.2140/ant.2023.17.1681},
  urldate = {2024-11-26},
  langid  = {english}
}

@inproceedings{tits1979reductive,
  title     = {Reductive groups over local fields},
  booktitle = {Automorphic Forms, Representations and {{L-functions}} ({{Proc}}. {{Sympos}}. {{Pure Math}}., {{Oregon State Univ}}., {{Corvallis}}, {{Ore}}., 1977), {{Part}}},
  author    = {Tits, Jacques},
  year      = 1979,
  volume    = {1},
  pages     = {29--69},
  urldate   = {2025-01-14}
}

@article{yamana2014lfunctions,
  title     = {L-functions and theta correspondence for classical groups},
  author    = {Yamana, Shunsuke},
  year      = 2014,
  journal   = {Inventiones mathematicae},
  volume    = {196},
  number    = {3},
  pages     = {651--732},
  issn      = {0020-9910, 1432-1297},
  doi       = {10.1007/s00222-013-0476-x},
  urldate   = {2025-01-05},
  copyright = {http://www.springer.com/tdm},
  langid    = {english}
}

@misc{yao2024kudlarapoport,
  title         = {A {Kudla--Rapoport} formula for exotic smooth models of odd dimension},
  author        = {Yao, Haodong},
  year          = 2024,
  number        = {arXiv:2404.14431},
  eprint        = {2404.14431},
  primaryclass  = {math},
  publisher     = {arXiv},
  doi           = {10.48550/arXiv.2404.14431},
  note          = {available at \url{https://arxiv.org/abs/2404.14431}},
  urldate       = {2026-01-14},
  archiveprefix = {arXiv}
}

@article{zachos2024semistable,
  title   = {Semistable models for some unitary {Shimura} varieties over ramified primes},
  author  = {Zachos, Ioannis},
  year    = 2024,
  journal = {Algebra \& Number Theory},
  volume  = {18},
  number  = {9},
  pages   = {1715--1736},
  issn    = {1944-7833, 1937-0652},
  doi     = {10.2140/ant.2024.18.1715},
  urldate = {2025-04-22},
  langid  = {english}
}

\end{document}